\documentclass[11pt,a4paper,twoside]{article}
\usepackage{amsmath,amsfonts,amssymb,amsthm,bbm,latexsym,mathrsfs}
\usepackage{graphicx,color,epsfig,fancyhdr,dsfont}
\usepackage{enumerate}
\usepackage{hyperref}
\usepackage{indentfirst}
\usepackage[all]{hypcap}
\usepackage[affil-it]{authblk}
\allowdisplaybreaks
\newtheorem{thm}{Theorem}[section]
\newtheorem{lem}[thm]{Lemma}
\newtheorem{cor}[thm]{Corollary}
\newtheorem{prop}[thm]{Proposition}
\newtheorem{defi}[thm]{Definition} 
\newtheorem{rmk}[thm]{Remark}
\newtheorem{exmp}[thm]{Example}

\renewcommand{\theequation}{\arabic{section}.\arabic{equation}}
 \numberwithin{equation}{section}

\begin{document}

\baselineskip15pt

\title
{Anticipating Random Periodic Solutions--I. 
 SDEs with Multiplicative Linear Noise}

\author[1]{
Chunrong Feng}
\author [1,2] {Yue Wu}
\author[1]{Huaizhong Zhao}
\affil[1]{Department of Mathematical Sciences, Loughborough
University, LE11 3TU, UK}
 \affil[2] {Current address: Department of Mathematical Sciences, Xi'an Jiaotong-Liverpool University, Suzhou 215123, China}
\affil[ ]{C.Feng@lboro.ac.uk, Y.Wu@lboro.ac.uk, H.Zhao@lboro.ac.uk}
\date{}

\maketitle
\newcounter{bean}
\vskip-10pt

\begin{abstract}
In this paper, we study the existence of random periodic solutions for
semilinear stochastic differential equations. We identify them
as solutions of coupled forward-backward infinite horizon stochastic
integral equations (IHSIEs), using the "substitution theorem" of stochastic differential equations 
with anticipating initial conditions. In general, random periodic solutions and the solutions of IHSIEs,
are anticipating.  For the linear noise case, with the help of
the exponential dichotomy given in the multiplicative ergodic theorem, we can identify them as the 
solutions of infinite horizon random integral equations (IHSIEs). We then 
solve a localised forward-backward IHRIE in $C(\mathbb{R}, L^2_{loc}(\Omega))$
using an argument of truncations, the Malliavin calculus, the relative compactness of
Wiener-Sobolev spaces in $C([0, T], L^2(\Omega))$ and Schauder's fixed point theorem. 
We finally measurably glue the local solutions together to obtain a global solution in $C(\mathbb{R}, 
L^2(\Omega))$. Thus we obtain the existence of a random periodic solution and a periodic measure. 
\medskip


\end{abstract}

\pagestyle{fancy}
\fancyhf{}
\fancyhead[LE,RO]{\thepage}
\fancyhead[LO]{\small{Anticipating Random Periodic Solutions--I. 
 SDEs with Multiplicative Linear Noise}}
\fancyhead[RE]{\small{C.R. Feng, Y. Wu, H.Z. Zhao}}
\renewcommand{\headrulewidth}{0.1pt}

 \renewcommand{\theequation}{\arabic{section}.\arabic{equation}}

\section{Introduction}

Many real world phenomena having the nature of mixing periodicity and randomness, e.g., change of temperatures on earth and seasonal economic data (\cite{franses}). 
Physicists have attempted to study random perturbations to periodic solutions for some time. They used 
first linear approximations or asymptotic expansions in small noise regime, e.g. see
\cite{kampen},\cite{knoblock}.  
The approach in \cite{knoblock} was to seek $Y(t+\tau,\omega)$ returning to a neighbourhood of $Y(t,\omega)$ for each noise realisation, where
$\tau>0$ is a fixed number. This reveals certain 
information about the "periodicity" under small noise perturbations. Efforts were also made in the mathematics community to seek a periodic solution $Y$ such    
that
$Y(0,\omega)=Y(\tau,\omega)=Y(2\tau,\omega)=\cdots $ (\cite{brzezniak},\cite{peng}).  However, this kind 
of strict periodicity exists only in very special situations in stochastic contexts.
In general, random perturbations to a periodic solution break the strict periodicity immediately, similar to the case that random perturbations 
break fixed points. The concept of stationary solutions is the corresponding notion of fixed points in the stochastic counterpart and has been subject to intensive study
in mathematics literature (\cite{kh-ma-si},\cite{mattingly},\cite{si2},\cite{zh-zh}). 
One of the obstacles to
make a systematic progress in the study of the random periodicity was the lack of a rigorous mathematical definition in a great generality and appropriate mathematical
tools.  

Recently random periodic solutions has been defined with a help of an observation that they
may be regarded as stationary solutions on sequences of fixed discrete times of equal length of 
one period.
The existence was studied for cocycles  in \cite{zh-zheng} using a qualitative method, and for stochastic semiflows in \cite{F-zh1},\cite{F-zh2} using some analysis tools. 
 There have been some more recent results 
including \cite{chekroun} on random attractors of the stochastic TJ model in climate dynamics,
\cite{wang} on bifurcations of stochastic reaction diffusion equations and \cite{blw}
on stochastic lattice systems. It was recently proved in \cite{F-zh3} that random periodic solutions and periodic measures are "equivalent" 
in the following sense. 
A random periodic solution gives rise to a periodic measure. 
Conversely from a periodic measure one can construct a random periodic process 
on an enlarged probability space, of which the law is the periodic measure. It was then proved that the 
strong law of large numbers (SLLN) 
holds for periodic measures and corresponding random periodic processes thus it gives a statistical description. There are numerous physically relevant
stochastic differential equations satisfying the conditions in this paper so they have random periodic solutions (Theorems \ref{Main1}, \ref{Main2}, \ref{Main3}). 

First, recall the definition of the random periodic solutions for stochastic semi-flows given in \cite{F-zh1}.
Let $X$ be a separable Banach space. Denote by $(\Omega,{\cal F},P,(\theta (s))_{s\in \mathbb{R}})$ a metric dynamical system, where 
$\theta (s):\Omega\to\Omega$ is  assumed to be $P$-preserving and measurably invertible for all $s\in \mathbb{R}$, $\Delta:=\{(t,s)\in \mathbb{R}^2, s\leq t\}$. Consider a stochastic semi-flow $u: \Delta \times \Omega\times X\to X$, which satisfies the following standard condition
\begin{eqnarray}{\label{16}}
u(t,r,\omega)=u(t,s,\omega)\circ u(s,r,\omega),\ \ {\rm for\ all } \ r\leq s\leq t,\ r, s,t\in \mathbb{R},\ \mbox{for }a.e.\ \omega\in\Omega.
\end{eqnarray}
We do not assume the map $u(t,s,\omega): X\to X$ to be invertible for $(t,s)\in \Delta,\ \omega \in \Omega$ in the following definition. 

\begin{defi}\label{feng-zhao1}
A random periodic path of period $\tau$ of the semi-flow $u: \Delta\times \Omega\times X\to X$
 is an ${\cal F}$-measurable
map $Y:\mathbb{R}\times \Omega\to \mathbb{R}^d$ such that
\begin{eqnarray}
u(t,s, Y(s,\omega), \omega)=Y(t,\omega),\ 
Y(s+\tau,\omega)=Y(s, \theta_\tau \omega),\ \  t, s\in \mathbb{R}, \ t\geq s, \      \ a.e. \omega\in \Omega.
\end{eqnarray}
\end{defi}
  
In this paper we consider the following stochastic differential equation with multiplicative noise since the influence of noise in many cases depends on the intensity of the solution
\begin{eqnarray}\label{feng-zhao21}
{\Bigg\{\begin{array}{l}du(t)=Au(t)\,dt+F(t,u(t))\, dt+\sum_{k=1}^{M}B_k(t,u(t))\circ  dW^k_t,  \ \ \ \ \ \ t\geq s 
\\
u(s)=x\in \mathbb{R}^d,
\end{array}
}
                                    \end{eqnarray}
where  $ W_t=\{W^k_t \ 1\leq k\leq M, t\in \mathbb{R}\}$ is an M-dimensional mutually-independent standard Brownian motion under the canonical probability space $(\Omega,\mathcal{F},(\mathcal{F}^t)_{t\in \mathbb{R}},  \mathbb{P})$. Assume there exists a constant $\tau>0$ such that
$F(t+\tau,u)=F(t,u)$, $B_k(t+\tau,u)=B_k(t,u)$. Such kind of SDEs is called $\tau$-periodic. Suppose that the continuous functions $F(t, u)$ and $B_k(t,u)$ are Lipschitz continuous in $u$ so that 
the initial value problem (\ref{feng-zhao21})
has a unique solution (see \cite{ik-wa}).  It is easy to see that the Lipschitz condition is always satisfied under the conditions given in the following sections.

Note the stochastic integral is of the Stratonovich type.
Recall the connection of the It\^o and Stratonovich integrals
\begin{eqnarray}\label{Correction}
 B_k(t,u(t))\circ  dW^k_t =  B_k(t,u(t)) dW^k_t +\frac{1}{2} \nabla_x B_k(t,u(t)) B_k(t,u(t)) dt.
\end{eqnarray}
 
 Random periodic
solutions of $\tau$-periodic semilinear stochastic differential equations with additive noise were studied in \cite{F-zh1} ($B_k(t,u)=B_k(t)$). The existence was obtained under conditions that $||F||_{\infty}<\infty, ||\nabla F||_{\infty}<\infty, ||B_k||_{\infty}<\infty$, $B_k(t)$ is H\"older continuous and the matrix $A$ is hyperbolic. 
 There is no requirement about $F$ being monotone or the Lipschitz constant being controlled by
  the spectrum of $A$. In fact the system is non-dissipative as it is contracting 
 in certain directions and expanding in  certain other directions. So the random periodic solution, if exists, is not stable or completely unstable.
 The pull-back method does not work. In fact the random periodic solution is an anticipating stochastic process, which depends on the
 whole path of the noise. 

If the stochastic system is dissipative, we can use other methods such as pull-back to study the problem, where there is no anticipating 
issue at all (\cite{zh-zheng}). In this paper, we consider non-dissipative systems with the $d\times d$ matrix $A$ being assumed to be hyperbolic. Similar to the case with additive noise, if 
this equation has a random periodic solution $Y$, it is also anticipating.  However the nonadaptedness causes a real difficulty in the analysis of
$\int_0^tB_k(s,Y(s))\circ dW^k_s$, though this term is still well defined through (\ref{Correction}). The first integral on the right hand side is of the Skorohod type,
of which the $L^2$-norm is given by the Skorohod isometry (\cite{nualart2})
\begin{eqnarray}\label{zhao22}
&&||\int_0^tB_k(s,Y(s))dW^k_s||_{L^2(\Omega)}^2\\
&=&\int_0^t\|B_k(s,Y(s))\|_{L^2(\Omega)}^2ds+\int_0^t \int_0^t (\nabla_x B_k(s,Y(s))\mathcal{D}^{k}_{r_1} Y(s))^{T}\nabla_x B_k(s,Y(s))\mathcal{D}^k_{r_2} Y(s) dr_1 dr_2,  \nonumber
\end{eqnarray}
where $A^T$ is the transpose of $A$ and $\mathcal{D}Y$ the Malliavin derivative of $Y$. 

 For a finite time SDEs with a given anticipating initial condition, Nualart proved a substitution 
theorem in \cite{nualart2}. With the help of this result, we can prove that
the random periodic solution is identified as a solution of the following IHSIEs
\begin{eqnarray}\label{zhao21}
Y(t)&=&\int _{-\infty}^t{\rm e}^{A(t-r)}P^-F(r,Y(r))dr-\int _t^{+\infty}{\rm e}^{A(t-r)}P^+F(r,Y(r))dr
\\
&&
+\sum_{k=1}^M\int _{-\infty}^t{\rm e}^{A(t-r)}P^-B_k(r,Y(r))\circ dW^k_r-\sum_{k=1}^M\int _t^{+\infty}{\rm e}^{A(t-r)}P^+B_k(r,Y(r))\circ dW^k_r.\nonumber
\end{eqnarray}
Here $P^-$ is the projection from $\mathbb{R}^d$ to the subspace spanned by the eigenvectors corresponding to the eigenvalue of $A$ with negative real parts and $P^+=I-P^-$.  
Thus our problem is reduced to solve 
(\ref{zhao21}). Note here
Eqn. (\ref{zhao21}) can be represented as 
an initial value problem (\ref{feng-zhao21}) with initial value $u(s)=Y(s)$ as for $t\geq s$, from (\ref{zhao21}),
\begin{eqnarray}\label{1.10}
Y(t)&=&{\rm e}^{A(t-s)}\int _{-\infty}^s{\rm e}^{A(s-r)}P^-F(r,Y(r))dr+\int _{s}^t{\rm e}^{A(t-r)}P^-F(r,Y(r))dr\nonumber\\
&&
+{\rm e}^{A(t-s)}\sum_{k=1}^M\int _{-\infty}^s{\rm e}^{A(s-r)}P^-B_k(r,Y(r))\circ dW^k_r+\int _{s}^t{\rm e}^{A(t-r)}P^-B_k(r,Y(r))\circ dW^k_r\nonumber\\
&&-{\rm e}^{A(t-s)}\int _s^{+\infty}{\rm e}^{A(s-r)}P^+F(r,Y(r))dr+\int _{s}^t{\rm e}^{A(t-r)}P^+F(r,Y(r))dr\nonumber\\
&&-{\rm e}^{A(t-s)}\sum_{k=1}^M\int _s^{+\infty}{\rm e}^{A(s-r)}P^+B_k(r,Y(r))\circ dW^k_r+\sum_{k=1}^M\int _{s}^t{\rm e}^{A(t-r)}P^+B_k(r,Y(r))\circ dW^k_r\nonumber
\\
&=&{\rm e}^{A(t-s)}Y(s)+\int _{s}^t{\rm e}^{A(t-r)}F(r,Y(r))dr+\sum_{k=1}^M\int _{s}^t{\rm e}^{A(t-r)}B_k(r,Y(r))\circ dW^k_r.
\end{eqnarray}
 This would be exactly the type of SDEs considered in \cite{nualart2} with a given anticipating initial value $Y(s)$ at time $s$ if it were known. 
However, in our context, $Y(s)$ is not known, but part of the solution of Eqn. (\ref{zhao21}) when $t=s$.

It turns out that the anticipating issue causes some real difficulties in solving (\ref{zhao21}), even in the linear noise case, due to the fact that the $L^2$ norm of $\int_0^tB_k(s,Y(s))\circ dW^k(s)$
involves the Malliavin derivative of $Y$ in (\ref{zhao22}). 
In this paper, we only consider the linear noise case. We use the stochastic linear evolution operator and subsequently identify 
random periodic solutions as the solutions of forward-backward coupled infinite horizon random
integral equations (IHRIEs)
with the help of the exponential dichotomy given in Oseledets' multiplicative ergodic theorem (MET). 

We cannot solve the IHRIEs pathwise though the equations are given in a pathwise manner. The major flaw of a pathwise 
approach is the lack of the measurability to their solutions. Thus we seek a solution in $C((-\infty,+\infty),L^2(\Omega))$.
Relative compactness is key in this analysis. However, pointwise (fix $\omega$) relative compactness theorem such as Arzel\`a-Ascoli Lemma 
is not enough. 
Another difficulty is the pathwise 
stability of $\Phi(t,\theta_s\omega)P^-$ in the MET and the stability in $L^2(\Omega)$ are different. For example, as $t\to \infty$, $e^{-\frac{1}{2}t+W_t}\to 0$ a.s., but $\mathbb{E}\left(e^{-\frac{1}{2}t+W_t}
\right)^2\to \infty$. To overcome this difficulty, we construct a sequence of localised IHRIEs. 
We then use  the relative compactness of
Wiener-Sobolev spaces in $C([0, T], L^2(\Omega))$ and Schauder's fixed point theorem
to solve the equation in $C(\mathbb{R}, L^2_{loc}(\Omega))$. 
We finally measurably glue the local solutions together to obtain a global solution in $C(\mathbb{R}, 
L^2(\Omega))$. It is interesting to note that the local solution may not converge to the global solution.
 
 With the existence of random periodic solution, we can construct a periodic measure $\mu_s$ of the skew product on $(\Omega\times R^d, {\cal F}\bigotimes {\cal B}(R^d))$ according to \cite{F-zh3}. The factorisation of $\mu_s$ which is given by 
 $(\mu_s)_{\omega}=\delta_{Y(s,\theta(-s)\omega)}$ is anticipating. 
 
We would like to point out that anticipating equilibrium processes exist in reality. For example, an anticipating stationary solution was 
found in the transition state problem in chemical reaction processes (c.f. \cite{TB}). 

\section{Preliminaries and the equivalence of random periodic solutions for hyperbolic systems and the coupled forward-backward IHSIEs}

Consider $\Omega=C_0(\mathbb{R},\mathbb{R}^{M}):=\{\omega\in C(\mathbb{R},\mathbb{R}^{M}):\omega(0)=0 \}$, $W_t(\omega):=\omega(t)$, and ${\cal F}^t:=\vee_{s\leq t}{\cal F}_s^t$ with ${\cal F}_s^t:=\sigma(W_u-W_v, s\leq v\leq u\leq t)$. Besides, we define a shift that leaves $\Omega$ invariant by
$$\theta_t:\Omega\to \Omega,\ \ \theta_s\omega(t)=\omega(t+s)-\omega(s),\ s,t\in \mathbb{R},$$
and thus the shift $\theta$ is $P$-measure preserving.  

First we briefly recall the Skorohod integral, Stratonovich intergal and prove Malliavin derivative's norm preserving property under the measure preserving operator 
$\theta$. We only need to consider $1$-dimensional Brownian motion $W$ on $\mathbb{R}$. 
The multidimensional case can be dealt with similarly.
 
Denote 
by $\hat L^2(\mathbb{R}^m)$ the set of symmetric functions in $L^2(\mathbb{R}^m)$.
Define for $f\in \hat L^2(\mathbb{R}^m)$,
$$I_m(f)=\int_{\mathbb{R}^m}f(t_1,\ldots,t_m)dW_{t_1}\ldots dW_{t_m}.$$
 It is well known that
 \begin{eqnarray}
\mathbb{E}[I_m(f)I_n(g)]=~\bigg\{\begin{array}{l}
                                        0,~~~  \ \ \ \ \ \ \ \ \ \ \ \ \ \ \ \ \ \ \ \ \ if\ n\neq m, \\
                                        m!<f, g>_{L^2{(\mathbb{R}^m)}},\ ~if\ n= m.
                                     \end{array}
\end{eqnarray}
\begin{defi}\label{Skorohod}\upshape{(Skorohod integral \cite{bernt})} Suppose that $v(t,\omega)$ is a stochastic process  such that $v(t,\cdot)$ is $\mathcal{F}$-measurable and square-integrable for all $t\in \mathbb{R}$.
Thus it has the following Wiener-It\^o chaos expansion 
\begin{equation*}
v(t)=\sum_{m=0}^{\infty}I_m(f_m(\cdot,t)),
\end{equation*}
with $f_m(\cdot,t)\in \hat L^2(\mathbb{R}^m)$ for each $t\in \mathbb{R}$.
Then the Skorohod integral is defined as
\begin{equation}\label{Skorohod1}
\delta(v):=\int_{\mathbb{R}}v(t,\omega)\delta W_t:=\sum_{m=0}^{\infty}I_{m+1}(\tilde{f}_m),
\end{equation}
where $\tilde{f}_m\in \hat L^2(\mathbb{R}^{m+1})$ is the symmetrization of $f_m(t_1,\cdots,t_m,t)$ as a function of all $m+1$ variables.
We say $u$ is Skorohod-integrable and write $u\in Dom(\delta)$ if the series in (\ref{Skorohod1}) converges in $L^2(\Omega)$.
\end{defi}

Another kind of integral is defined in the probability sense (1-dimensional case):
\begin{defi}\label{Strato}\upshape{(Stratonovich integral \cite{nualart2})} A measurable process $v(t,\omega)$ such that $\int_{\mathbb{R}}|v_t|dt<\infty$ a.s. is Stratonovich integrable if the familiy $S^{\pi}$ 
$$S^{\pi}:=\int_{\mathbb{R}}v_tW^{\pi}_{t}dt,$$
where
$$W^{\pi}_{t}=\sum_{i=0}^{n-1}\dfrac{W_{t_{i+1}}-W_{t_{i}}}{t_{i+1}-t_i}\chi_{(t_i,t_{i+1}]}(t),$$
converges in probability as $|\pi|\to 0$ and in this case the limit is called the Stratonovich integral, denoted by $\int_{\mathbb{R}}v_t\circ dW_t.$
\end{defi}

Note that the relation (\ref{Correction}) holds between the Stratonovich integral and the Skorohod integral (c.f. \cite{nualart2}), i.e.,
 \begin{eqnarray}\label{Correction2}
\int_{\mathbb{R}} B(s,u(s))\circ  dW_s =  \int_{\mathbb{R}} B(s,u(s)) \delta W_s +\frac{1}{2} \int_{\mathbb{R}}\nabla_x B(s,u(s)) B(s,u(s)) ds,
\end{eqnarray}
where $B\in \mathcal{L}(\mathbb{R}^d)$.

 Let $\mathcal{S}$ denote the class of smooth and cylindrical random variables of the form 
$$G=f(W(h_1),\cdots,W(h_n)),$$
where $f\in C_p^\infty(\mathbb{R}^n)$, i.e., $f$ and all its partial derivatives have polynomial growth order, and $W(h_i)=\int_{\mathbb{R}}h_i(s)dW_s$, $h_1,\cdots,h_n\in L^2(\mathbb{R})$, and $n\geq 1$. The derivative of $G$ is the $L^2(\mathbb{R})$-valued random variable given by
\begin{equation*}
\mathcal{D}_rG=\sum_{i=1}^n\frac{\partial f}{\partial x_i}(W(h_1),\cdots,W(h_n))h_i(r).
\end{equation*}
Then denote by $\mathcal{D}^{1,p}$ the domain of $\mathcal{D}$ in $L^p(\Omega)$, i.e. $\mathcal{D}^{1,p}$ is the closure of $\mathcal{S}$ with respect to the norm 
$$||G||_{1,p}=\left(\mathbb{E}|G|^p+\mathbb{E}||\mathcal{D}_{\cdot}G||^p_{ L^2(\mathbb{R})}\right)^{\frac{1}{p}}.$$

The following simple result is about the ${\cal D}^{1,2}$-norm-preserving property under the measure preserving operator.  It will play a crucial role in the subsequent argument.
This property is not normally true for Malliavin derivatives, but it is here due to the fact that the time interval is the whole real line  $\mathbb{R}$.    

\begin{lem}\label{Preserving}\upshape{(Norm preserving in $\mathcal{D}^{1,2}$)} 
Suppose $G(\cdot)\in \mathcal{D}^{1,2}$, then for all $h\in \mathbb{R}$, $G(\theta_h\cdot)\in  \mathcal{D}^{1,2}$, and
$$\|G(\theta_{h}\cdot)\|_{1,2}=\|G(\cdot)\|_{1,2},$$
where $\theta_h:\Omega\to \Omega$, $h\in \mathbb{R}$ is the measure preserving measurable dynamical system on $(\Omega, \mathcal{F},\mathbb{P})$.
\end{lem}
\begin{proof}
First by the measure-preserving property, it is easy to see that
$$\mathbb{E}|G(\theta_{h}\cdot)|^2=\mathbb{E}|G(\cdot)|^2.$$
By Wiener-It\^o's chaos decomposition (c.f. \cite{bernt}),  $G(\cdot)\in \mathcal{D}_{1,2}$ can be written as 
$$G(\omega)=\sum_{k=0}^{\infty}\int_{\mathbb{R}}\cdots\int_{\mathbb{R}}f_k(t_1,\cdots,t_k)dW_{t_1}\cdots dW_{t_k},$$
where $f_k(-)$ is a symmetric element in $\hat L^2(\mathbb{R}^k)$, and 
\begin{eqnarray*}
G(\theta_{h}\omega)&=&\sum_{k=0}^{\infty}\int_{\mathbb{R}}\cdots\int_{\mathbb{R}}f_k(t_1-h,\cdots,t_k-h)dW_{t_1}\cdots dW_{t_k}.
\end{eqnarray*}
The corresponding Malliavin derivatives can be derived through Wiener-It\^o chaos decomposition,
$$\mathcal{D}_r G(\omega)=\sum_{k=0}^{\infty}\int_{\mathbb{R}}\cdots\int_{\mathbb{R}}f_k(t_1,\cdots,t_{k-1},r)dW_{t_1}\cdots dW_{t_{k-1}},$$
and 
\begin{eqnarray*}
\mathcal{D}_r G(\theta_{h}\omega)=\sum_{k=0}^{\infty}\int_{\mathbb{R}}\cdots\int_{\mathbb{R}}f_k(t_1-h,\cdots,t_{k-1}-h, r-h)dW_{t_1}\cdots dW_{t_{k-1}}.
\end{eqnarray*}
Therefore
\begin{eqnarray*}
\mathbb{E}\int_{\mathbb{R}}\|\mathcal{D}_r G(\theta_h \cdot)\|^2dr&=&\sum_{k=0}^{\infty}(k-1)!\int_{\mathbb{R}}\|f_k(t_1-h,\cdots ,t_{k-1}-h, r-h)\|^2_{L^2(\mathbb{R}^{k-1})}dr\\
&=&\sum_{k=0}^{\infty}(k-1)!\int_{\mathbb{R}^k}|f_k(t_1,\cdots,t_{k-1}, r)|^2 dt_1\cdots d t_{k-1}dr\\
&=& \mathbb{E}\int_{\mathbb{R}}\|\mathcal{D}_r G(\cdot)\|^2dr.
\end{eqnarray*}
The claim is asserted.     
\end{proof}

Note that when $(\Omega, \mathcal{F},\mathbb{P})$ is the canonical probability space associated with an $M$-dimensional Brownian motion $\{W^j_t, t\in \mathbb{R}, 1\leq j\leq M\}$, $\mathcal{D}G$ of a random variable $G\in \mathcal{D}^{1,2}$ will be an $M$-dimensional process denoted by $\{\mathcal{D}^j_r G, r\in \mathbb{R},1\leq j\leq M \}$. For example
$$\mathcal{D}^j_r W^k_t=\delta_{k,j}\chi_{(-\infty,t]}(r).$$
Here and throughout the paper $\chi_.(\cdot)$ always represents an indicator function.
Denote the solution of the initial value problem (\ref{feng-zhao21}) by $u(t,s,\omega)x$, $t\geq s$, and
\begin{eqnarray*}
u(t,s,x)=e^{A(t-s)}x+\int_s^te^{A(t-r)}F(r,u(r,s,x))dr+\sum\limits _{k=1}^{M}\int_s^te^{A(t-r)}B_k(r,u(r,s,x)) dW^k_r.
\end{eqnarray*} 
It is easy to see that $u$ satisfies (\ref{16}). Moreover, by Proposition 2.3.22 in \cite {ar} and periodicity of $F$ and $B_k$, similar to the proof cocycle case in Theorem 2.3.26 in \cite{ar}, we have that 
\begin{eqnarray*}
&&u(t+\tau,s+\tau,x)\\
&=&e^{A(t-s)}x+\int_{s+\tau}^{t+\tau}e^{A(t+\tau-r)}F(r,u(r,s+\tau,x))dr
+\sum\limits _{k=1}^{M}\int_{s+\tau}^{t+\tau}e^{A(t+\tau-r)}B_k(r,u(r,s+\tau,x)) dW^k_{r}\\
&=&e^{A(t-s)}x+\int_{s}^{t}e^{A(t-r)}F(r+\tau,u(r+\tau,s+\tau,x))dr\\
&&+\sum\limits _{k=1}^{M}\int_{s}^{t}e^{A(t-r)}B_k(r+\tau,u(r+\tau,s+\tau,x)) d{\widetilde W}^k_{r}\\
&=&e^{A(t-s)}x+\int_{s}^{t}e^{A(t-r)}F(r,u(r+\tau,s+\tau,x))dr+\sum\limits _{k=1}^{M}\int_{s}^{t}e^{A(t-r)}B_k(r,u(r+\tau,s+\tau,x)) d{\widetilde W}^k_{r},
\end{eqnarray*}
where $\widetilde W_r:=(\theta_\tau\omega)(r)=W_{r+\tau}-W_\tau$. On the other hand,
\begin{eqnarray*}
\theta_\tau u(t,s,x)&=&e^{A(t-s)}x+\int_{s}^{t}e^{A(t-r)}F(r,\theta_\tau u(r,s,x))dr+\sum\limits _{k=1}^{M}\int_{s}^{t}e^{A(t-r)}B_k(r,\theta_\tau u(r,s,x)) d{\widetilde W}^k_{r}.
\end{eqnarray*}
By the pathwise uniqueness of the solution of (\ref{feng-zhao21}), we have that
\begin{eqnarray}
u(t+\tau,s+\tau,\omega)=u(t,s,\theta_{\tau}\omega), {\rm \ for \ all}\ t\geq s, t,s\in {\mathbb R}, \ a.s.
\end{eqnarray}
Note that $u(t,s,\omega)$ is also well-defined when $t\leq s$  and satisfies (c.f. \cite{ku2})
\begin{eqnarray}\label{uinverse}
{\Bigg\{\begin{array}{l}du(t)=-[Au(t)+F(t,u(t))+\sum\limits _{k=1}^{M}\nabla_x B_k(t,u(t)) B_k(t,u(t))]dt-\sum\limits _{k=1}^{M}B_k(t,u(t)) dW^k_t, 
\\
u(s)=x\in \mathbb{R}^d.
\end{array}
}
                                    \end{eqnarray}
 This means that the stochastic semi-flow is invertible a.s. Although this is not essential to make our method
 working, it helps to derive the IHSIEs. In the case of SPDEs, this property does not hold. In \cite{fwz2}, applying the unstable manifold theorem 
 (\cite{mo-zh-zh}), we can still deduce the IHSIEs in the infinite dimensional case.

The following substitution theorem for anticipating stochastic differential equations in \cite{nualart2} will play an important role in the development of the connection between the IHSIE and random periodic solutions.

\begin{lem}\label{Stratonovich}
Consider the following stochastic differential equation on $[0,T]$, $T>0$
\begin{eqnarray}\label{stratonovich1}
X_{t,s}&=&X_0+ \sum_{i=1}^M\int_s^t\sigma_i(\hat{s},X_{\hat{s},s}) d W^i_{\hat{s}}+\int_s^t\beta(\hat{s},X_{\hat{s},s})d\hat{s},\ \ t\geq s,
\end{eqnarray}
where $\sigma_i\in C^3(\mathbb{R}^{d+1})$, $0 \leq i\leq M$, and $\beta\in C^3(\mathbb{R}^{d+1})$ have bounded partial derivatives of first order. 
Then for any random vector $X_0$, the process $X=\{\varphi_{t,s}(X_0), t\in[0,T]\}$ satisfies the anticipating SDEs (\ref{stratonovich1}),
where $\{\varphi_{t,s}(x), t\in [0,T]\}$ is the stochastic flow defined by:
\begin{eqnarray}\label{stratonovich2}
\varphi_{t,s}(x)&=&x+ \sum_{i=1}^M\int_s^t\sigma_i(\hat{s},\varphi_{\hat{s},s}(x)) d W^i_{\hat{s}}+\int_s^t\beta(\hat s, \varphi_{\hat{s},s}(x))d\hat{s}, \ \ t\geq s.
\end{eqnarray}
Besides, if $\sigma_i, 1\leq i \leq M$, and $\beta$ are of class $C^4$, and $X_0\in \mathcal{D}^{1,p}(\mathbb{R}^d)$ for some $p>4$, then the process $X$ is the unique solution to (\ref{stratonovich1}) in $L^2([0,T],\mathcal{D}^{1,4}(\mathbb{R}^d))$ and is continuous in $t$ almost surely. 
 \end{lem} 
 
 The equations considered in \cite{nualart2} is time independent case. But time dependent case can be easily deduced to time independent case by considering $\tilde{X}_t=\begin{bmatrix}
t+s\\X_{t+s,s}
\end{bmatrix}$.
  
 Now we consider the general $\tau$-periodic SDEs with multiplicative noise (\ref{feng-zhao21}).  

\begin{thm}\label{T1-2'} \upshape{(Equivalence theorem)}
 Let $F:\mathbb{R} \times \mathbb{R}^d\to \mathbb{R}^d$ and $B_k: \mathbb{R}\times\mathbb{R}^d\to \mathbb{R}^d$ be of class $C^3$, with the Jacobians $\nabla F(t,\cdot)$ and $\nabla B_k(t,\cdot)$ globally bounded, for $1\leq k\leq M$, i.e. $\sup\limits_{t\in \mathbb{R},x\in \mathbb{R}^d} ||\nabla F(t,x)||_{\mathcal L(\mathbb{R}^d)}+\sup\limits_{t\in \mathbb{R},x\in \mathbb{R}^d} ||\nabla B_k(t,x)||_{\mathcal L(\mathbb{R}^d)}<\infty$. Assume $F(t,u)=F(t+\tau,u)$ and $B_k(t,u)=B_k(t+\tau,u)$ for some fixed $\tau>0$. 
 Then a tempered  $Y\in L^2([0,\tau),\mathcal{D}^{1,p})$ for some $p>4$ such that $Y(t+\tau,\omega)=Y(t,\theta_{\tau} \omega) \ {\rm for \ any} \
t\in \mathbb{R}$ $\mathbb{P}$-a.s.  is a random periodic solution of Eqn. (\ref{feng-zhao21}) if and only if $Y$ satisfies Eqn. (\ref{zhao21}).
\end{thm}
\begin{proof}
If Eqn. (\ref{zhao21}) has a solution $Y\in \mathcal{D}^{1,p}(L^2(\mathbb{R},\mathbb{R}^d))$ for some $p>4$, then it also satisfies (\ref{1.10}). 
Thanks to Nualart's substitution theorem (Lemma \ref{Stratonovich}), which guarantees the uniqueness of solution to (\ref{feng-zhao21}) with anticipating initial value $Y(s,\omega)$, 
we have
$$
u(t,s, x,\omega)\Big|_{x=Y(s,\omega)}=u(t,s, Y(s,\omega),\omega)=Y(t,\omega).
$$
Conversely, assume Eqn. (\ref{feng-zhao21}) has a random periodic solution which is
 tempered from above. First note for any non-negative integer $l$, we then apply the substitution theorem again to obtain
\begin{eqnarray*}
 Y(t,\omega)&=&u(t\pm l\tau,t,Y(t,\theta _{\mp l\tau}\omega), \theta _{\mp l\tau}\omega)\\
 &=&{\rm e}^{\mp Al\tau}Y(t,\theta _{\mp l\tau}\omega)+\int _t^{t\pm l\tau}{\rm e}^{A(t\pm l\tau-{\hat{s}})}F({\hat{s}},u({\hat{s}},t,Y(t,\theta _{\mp l\tau}\omega),\theta _{\mp l\tau}\omega))d{\hat{s}}\\
 &&+\sum_{k=1}^M\int _t^{t\pm l\tau}{\rm e}^{A(t\pm l\tau-{\hat{s}})}B_k(\hat{s},u({\hat{s}},t,Y(t,\theta _{\mp l\tau}\omega),\theta _{\mp l\tau}\omega))\circ dW^k_{\hat{s}\mp l\tau}.
 \end{eqnarray*}
 Therefore,
 \begin{eqnarray*}
P^-Y(t,\omega)&=&P^-
u(t+l\tau,t,Y(t,\theta _{-l\tau}\omega), \theta _{-l\tau}\omega)\\
&=&{\rm e}^{Al\tau}P^-Y(t,\theta _{-l\tau}\omega)+\int ^t_{t- l\tau}{\rm e}^{A(t-{\hat{s}})}P^-
 F({\hat{s}},Y({\hat{s}},\omega))d{\hat{s}}\\
 &&
 +\sum_{k=1}^M\int ^t_{t- l\tau}{\rm e}^{A(t-{\hat{s}})}P^-
 B_k(\hat{s},Y({\hat{s}},\omega))\circ dW^k_{\hat{s}}\\
&\to&
\int _{-\infty}^t{\rm e}^{A(t-{\hat{s}})}P^-
 F({\hat{s}},Y({\hat{s}},\omega))d{\hat{s}}+ \sum_{k=1}^M\int ^t_{-\infty}{\rm e}^{A(t-{\hat{s}})}P^-
 B_k(\hat{s}, Y({\hat{s}},\omega))\circ dW^k_{\hat{s}}
\end{eqnarray*}
in $L^2$-norm as $l\to+\infty$. The last convergence can be demonstrated by the Skorohod isometry and (\ref{Correction2}). Indeed by the linear growth of $B$ and boundedness of its gradients, for each $k$
\begin{eqnarray*}
&&\mathbb{E}\left[\left|\int ^{t-l\tau}_{-\infty}{\rm e}^{A(t-{\hat{s}})}P^-
 B_k(\hat{s},Y({\hat{s}},\omega)) \delta W^k_{\hat{s}}\right|^2\right]
\\
&\leq & 2C_1\int ^{t-l\tau}_{-\infty}e^{2\mu_{m+1}(t-\hat{s})}(1+|Y(\hat{s})|^2) d\hat{s}
\\&&
+ C\|\nabla B_k\|^2_{\infty}\mathbb{E}\int ^{t-l\tau}_{-\infty}e^{2\mu_{m+1}(t-\hat{s})}\int ^{t-l\tau}_{-\infty}\|\mathcal{D}_rY(\hat{s},\omega)\|^2drd\hat{s}.
\end{eqnarray*} 
Let us consider the second term on the right hand side only, as the first term can be dealt with analogously,
\begin{eqnarray}\label{Jan28}
&&\mathbb{E}\int ^{t-l\tau}_{-\infty}e^{2\mu_{m+1}(t-\hat{s})}\int ^{t-l\tau}_{-\infty}\|\mathcal{D}_rY(\hat{s},\omega)\|^2drd\hat{s}\nonumber\\
&\leq &e^{2\mu_{m+1}l\tau}\int ^{t}_{-\infty}e^{2\mu_{m+1}(t-\hat{s})}\mathbb{E}\int ^{\infty}_{-\infty}\|\mathcal{D}_rY(\hat{s},\omega)\|^2drd\hat{s}\nonumber\\
&\leq & e^{2\mu_{m+1}l\tau} \sum_{i=-1}^{\infty}e^{2\mu_{m+1}i\tau}\int_{0}^{\tau}
\mathbb{E}\int ^{\infty}_{-\infty}\|\mathcal{D}_rY(\hat{s},\omega)\|^2drd\hat{s},
\end{eqnarray} 
where we use the periodicity $Y(\hat{s},\omega)$ and the norm preserving property  Lemma \ref{Preserving} about $\mathbb{E}\int ^{+\infty}_{-\infty}\|\mathcal{D}_rY(\hat{s},\omega)\|^2dr$. Moreover, it is easy to see that (\ref {Jan28}) tends to $0$ when $l\to \infty$. 
Analogously, as $u$ is invertible,
\begin{eqnarray*}
P^+Y(t,\omega)&=&P^+
u(t-l\tau,t,Y(t,\theta _{l\tau}\omega), \theta _{l\tau}\omega)\\
&=&{\rm e}^{-Al\tau}P^-Y(t,\theta _{l\tau}\omega)-\int _t^{t+l\tau}{\rm e}^{A(t-{\hat{s}})}P^+
 F({\hat{s}},Y({\hat{s}},\omega))d{\hat{s}}\\
 &&-\sum_{k=1}^M\int _t^{t+l\tau}{\rm e}^{A(t-{\hat{s}})}P^+
 B_k(\hat{s},Y({\hat{s}},\omega))\circ dW^k_{\hat{s}}\\
&\to&
-\int^{\infty}_t{\rm e}^{A(t-{\hat{s}})}P^+
 F({\hat{s}},Y({\hat{s}},\omega))d{\hat{s}}- \sum_{k=1}^M\int _t^{\infty}{\rm e}^{A(t-{\hat{s}})}P^-
 B_k(\hat{s},Y({\hat{s}},\omega))\circ dW^k_{\hat{s}}
\end{eqnarray*}
in $L^2$-norm as $l\to +\infty$.  
\end{proof}

In the general multiplicative noise case, it remains open to solve Eqn. (\ref{zhao21}). 
We will solve the linear noise case in the next two sections.

\section{Linear noise: the exponential dichotomy and IHRIEs}

Consider the following $\tau$-periodic semilinear SDE of Stratonovich type with multiplicative linear noise,
\begin{eqnarray}\label{origT1}
\bigg\{\begin{array}{l}du(t)=Au(t)\,dt+F(t,u(t))\, dt+\sum\limits_{k=1}^MB_k u(t)\circ  dW^k_t,  \ \ \ \ \ \ t\geq s \\
u(s)=x\in \mathbb{R}^d,
                                     \end{array}
\end{eqnarray}
where $A,\{B_k,\ 1\leq k\leq M\}$ are in $\mathcal{L}(\mathbb{R}^d)$, $W_t:=(W^1_t,W^2_t, \cdots,W^M_t)$, $t\in \mathbb{R}$, is an $M$-dimensional Brownian motion under the filtered Wiener space $(\Omega,\mathcal{F},(\mathcal{F}^t)_{t\in \mathbb{R}},  \mathbb{P})$. 
In addition, we assume that
\vskip3pt

{\bf{ Condition (C).}} The matrices $A$, $A^*$, $B_k$, and $B^*_k$  are mutually commutative.
\vskip3pt


Now define a random evolution operator $\Phi : \mathbb{R}^+\times \Omega\rightarrow \mathcal{L}(\mathbb{R}^d)$ by
 \begin{eqnarray}\label{phiT1}
\bigg\{\begin{array}{l}d\Phi(t)=A\Phi(t)\,dt+\sum\limits_{k=1}^MB_k \Phi(t)\circ  d W^k_t,\ \ t\geq 0   \\
\Phi(0)=I\in \mathcal{L}(\mathbb{R}^d),
                                     \end{array}
\end{eqnarray}
which is a cocycle over the metric dynamical system $(\Omega,{\cal F}, P, (\theta_t)_{t\in \mathbb {R}})$ (\cite{ar},\cite{bismut},\cite{elworthy},\cite{ku2},\cite{meyer}).
Due to the commutative property of $A$ and $B_k$, $\Phi$ can be written in the explicit form as 
$$\Phi(t,\omega)=\exp\big\{At+\textstyle\sum\nolimits_{k=1}^MB_k W^k_t\big\}.$$
Recall that the solution of (\ref{origT1}) via (\ref{phiT1}) can be written as (c.f. \cite{mo-zh-zh})
\begin{eqnarray}\label{mild}
u(t,s,x,\omega)=\Phi(t-s, \theta_s \omega)x+\int_s^t\Phi(t-{\hat{s}}, \theta_{\hat{s}}\omega)F(\hat{s},u({\hat{s}},s,x,\omega))\, d{\hat{s}},  \ \  \ t\geq s. 
\end{eqnarray}

\begin{rmk} Though being defined only on $\mathbb{R}^+$ in the above, $\Phi$ can be extended to $\mathbb{R}^-$ in the finite dimentional case via the relation $\Phi(t,\omega)=\Phi(-t,\theta_t\omega)^{-1}$ when $t\leq 0$ as it is measurably invertible. Here $\Phi(t,\omega)^{-1}$ is uniquely defined and satisfies (\cite{ku2})
$$d\Phi(t)^{-1}=-A\Phi(t)^{-1}\,dt+\sum\limits_{k=1}^MB^2_k \Phi(t)^{-1}\,dt-\sum\limits_{k=1}^MB_k \Phi(t)^{-1}\circ  d W^k_t,\ \  t\geq 0 . $$  
\end{rmk}

Note
it is not hard to check that $\Phi$ is a perfect two-sided linear cocycle, so it satisfies the multiplicative ergodic theorem (MET) in Euclidean space (\cite{ar}). The proof is postponed to the Appendix. 
%

\begin{lem}\label{LEMMA1C3}\upshape{(Exponential dichotomy)}
 Suppose that $\frac{A+A^{\ast}}{2}$ has only nonzero eigenvalues with the order $\mu_p<\mu_{p-1}<\dots < \mu_{m+1}<0<\mu_{m}<\dots<\mu_1$, $p\leq d$, and the corresponding eigenspaces $E_p,\cdots,E_1$ with multiplicity $d_i:=\mbox{dim}\ E_i$. Here $\sum_{i=1}^{p}d_i=d$. Then
\begin{enumerate}
\item There exists a non-random splitting
$$\mathbb{R}^d=E_p\oplus E_{p-1}\oplus \dots\oplus E_{m+1}\dots\oplus E_{1}\ \ \ \mathbb{P}-\mbox{a.s.},$$
and 
$$\mu_i=\lim_{t\to \pm \infty}\frac{1}{t}\log|\Phi(t,\omega)x|,\ \mbox{for} \ \ x\in E_i\setminus \{0\},$$
is the Lyapunov exponent of $\Phi$, with the corresponding multiplicity $d_i$. Moreover, $\mathbb{R}^d$ can be decomposed as 
$$\mathbb{R}^d=E^{-}\oplus E^{+},$$
where $E^-=E_p\oplus E_{p-1}\oplus \dots\oplus E_{m+1}$ is generated by the eigenvectors with negative eigenvalues, while $E^+=E_m\oplus E_{m-1}\oplus \dots\oplus E_{1}$ is generated by the eigenvectors with positive eigenvalues.
\item Let $P^{\pm}: \mathbb{R}^d\to E^{\pm}$ be the projection onto $ E^{\pm}$ along $E^{\mp}$. Then
$$\Phi(t,\theta_{s}\omega)P^{\pm}=P^{\pm}\Phi(t,\theta_{s}\omega)\ \ \ \mathbb{P}-\mbox{a.s.},  $$
with exponential dichotomy on an invariant set $\hat{\Omega}$ of full measure, 
 \begin{eqnarray}\label{PhiBound}
\Bigg\{\begin{array}{l} \|\Phi(t,\theta_{s}\omega)P^+\|\leq C(\theta_{s}\omega)  {\rm e}^{\frac{1}{2}\mu_{m}t}\leq C_{\Lambda}(\omega)  {\rm e}^{\frac{1}{2}\mu_{m}t}{\rm e}^{\Lambda|s|},\ \ \  \ t\leq 0,  \\
\|\Phi(t,\theta_{s}\omega)P^-\|\leq C(\theta_{s}\omega)  {\rm e}^{\frac{1}{2}\mu_{m+1}t}\leq  C_{\Lambda}(\omega) {\rm e}^{\frac{1}{2}\mu_{m+1}t}{\rm e}^{\Lambda|s|},\ \ \ \ t\geq 0,
                                     \end{array}
\end{eqnarray} 
 for any $s\in \mathbb{R}$, where $\|\cdot\|$ denotes the norm on $\mathcal{L}(\mathbb{R}^d)$, and $C(\omega)$ is a tempered random variable from above, 
$\Lambda$ is an arbitrary positive number and $C_{\Lambda}(\omega)$ a positive random variable depending on $\Lambda$.
\end{enumerate} 
 \end{lem}
 
Some elementary but useful results can be derived from (\ref{PhiBound}). Their proof is postponed to an Appendix.
\begin{cor}\label{LEMMA2C3}
For any $t\geq 0$, and $\hat{s}\in \mathbb{R}$,  we have 
$$ \mathbb{E} \|P^{-}-\Phi(t, \theta_{\hat{s}}\cdot)P^{-}\|^2=\mathbb{E} \|P^{-}-\Phi(t, \cdot)P^{-}\|^2\leq C(|t|+1){\rm e}^{2\|A\||t|+2M\|B\|^2|t|}|t|,$$
where $C$ is a generic constant that may depend on $M$, $A$, $B_k$, $\mu_{m+1}$, $\mu_m$, $F$, and $\tau$, 
and for any $t\leq 0$, and $\hat{s}\in \mathbb{R}$, we have
$$ \mathbb{E} \|P^{+}-\Phi(t, \theta_{\hat{s}}\cdot)P^{+}\|^2=\mathbb{E} \|P^{+}-\Phi(t, \cdot)P^{+}\|^2\leq C(|t|+1){\rm e}^{2\|A\||t|+2M\|B\||t|}|t|.$$
Moreover, we have 
$$\mathbb{E} \|\Phi(t, \theta_{\hat{s}}\cdot)P^{\pm}\|^2=\mathbb{E} \|\Phi(t, \cdot)P^{\pm}\|^2\leq C{\rm e}^{2\|A\||t|+2M\|B\|^2|t|}.$$
\end{cor}

We will look for a $\mathcal{B}(\mathbb{R})\otimes\mathcal{F}$-measurable map $Y: \mathbb{R}\times \Omega\rightarrow \mathbb{R}^d$ which satisfies the following coupled forward-backward IHRIE, 
\begin{eqnarray}\label{VT1}
Y(t, \omega)&=&\int_{-\infty}^t\Phi(t-{\hat{s}},\theta_{\hat{s}} \omega)P^-F({\hat{s}},Y({\hat{s}},\omega))d{\hat{s}} \nonumber\\
& &-\int^{+\infty}_t\Phi(t-{\hat{s}},\theta_{\hat{s}} \omega)P^+F({\hat{s}},Y({\hat{s}},\omega))d{\hat{s}}, 
\end{eqnarray}
for all $\omega \in \Omega$, $t\in \mathbb{R}$. 
For any $N\in \mathbb{N}$, set the truncation of $\Phi(t,\theta_{\hat{s}}\omega)P^{\pm}$ by $N$: 
\begin{eqnarray}\label{PHI-B}
 \Phi^N(t,\theta_{\hat{s}}\omega)P^{-}:&=&\Phi(t,\theta_{\hat{s}}\omega)P^{-}\min\left\{1,  \frac{N{\rm e}^{\frac{1}{2}\mu_{m+1}t}{\rm e}^{\Lambda|\hat{s}|}}{\| \Phi(t,\theta_{\hat{s}}\omega)P^{-}\|}\right\}, {\rm \ when}\ t\geq 0,
  \\
  \label{PHI+B}
 \Phi^N(t,\theta_{\hat{s}}\omega)P^{+}:&=&\Phi(t,\theta_{\hat{s}}\omega)P^{+}\min\left\{1,  \frac{N{\rm e}^{\frac{1}{2}\mu_{m}t}{\rm e}^{\Lambda|\hat{s}|}}{\| \Phi(t,\theta_{\hat{s}}\omega)P^{+}\|}\right\}, {\rm when}\  t\leq 0.
  \end{eqnarray}
  We first consider a sequence of $\mathcal{B}(\mathbb{R})\otimes\mathcal{F}$-measurable maps $\{Y^N\}_{N\geq 1}$
defined by solutions of 
\begin{eqnarray}\label{VNT1}
Y^N(t, \omega)&=&\int_{-\infty}^t\Phi^N(t-{\hat{s}},\theta_{\hat{s}} \omega)P^-F({\hat{s}},Y^N({\hat{s}},\omega))d{\hat{s}} \nonumber\\
& &-\int^{+\infty}_t\Phi^N(t-{\hat{s}},\theta_{\hat{s}} \omega)P^+F({\hat{s}},Y^N({\hat{s}},\omega))d{\hat{s}},
\end{eqnarray}
for all $\omega \in \Omega$, $t\in \mathbb{R}$. We will develop tools to solve Eqn. (\ref{VT1}) via Eqn. (\ref{VNT1}).
Denote $\mu:=\min \{-\mu_{m+1},\mu_m\}$. Set
\begin{equation}\label{omegan}
\Omega_N:=\left\{\omega:\  \sup_{s\in\mathbb{R}}\max\left\{\sup_{t\geq 0}\|\Phi(t,\theta_s\omega)P^{-}\|{\rm e}^{-\frac{1}{2}\mu|t|-\Lambda|s|},\ \sup_{t\leq 0}\|\Phi(t,\theta_s\omega)P^{+}\|{\rm e}^{-\frac{1}{2}\mu|t|-\Lambda|s|}\right\}\leq N\right\}.
\end{equation}
Note that for $\omega\in \Omega_N$, $\Phi=\Phi^N$, and consequently Eqn. (\ref{VT1}) coincides with Eqn. (\ref{VNT1}). Moreover Lemma \ref{LEMMA1C3} suggests that $\Omega_N\to \Omega$ as $N\to \infty$. Therefore $Y^N$ is a local solution of Eqn. (\ref{VT1}).

Note that Stratonovich integral is defined in the sense of convergence in probability. In the multiplicative linear noise case, with Theorem \ref{Stratonovich} in hand, we are able to identify the random periodic solution of  (\ref{origT1}) with the solution of IHRIE (\ref{VT1}) without assuming $Y$ being Malliavin differentiable.

\begin{thm}\label{feng-zhao60}
\label{T1-2} 
 Let $F:\mathbb{R} \times \mathbb{R}^d\to \mathbb{R}^d$ be of class $C^3$, globally bounded and the Jacobians $\nabla F(t,\cdot)$  globally bounded. Assume $F(t,u)=F(t+\tau,u)$ for some fixed $\tau>0$. 
  Then a tempered  $Y$ such that $Y(t+\tau,\omega)=Y(t,\theta_{\tau} \omega) \ {\rm for \ any} \
t\in \mathbb{R}$ $\mathbb{P}$-a.s.  is a random periodic solution if and only if $Y$ satisfies (\ref{VT1}).
\end{thm}
\begin{proof}
If Eqn. (\ref{VT1}) has a solution $Y(\cdot,\omega)$,
 then from Eqn. (\ref{VT1}) by using the cocycle property of $\Phi$ we have for any $t\geq s$,
\begin{eqnarray*}
Y(t, \omega)
&=&
\Phi(t-s,\theta_s\omega)Y(s,\omega)+\int_s^{t}\Phi(t-{\hat{s}},\theta_{\hat{s}}\omega)F({\hat{s}},Y({\hat{s}},\omega))d{\hat{s}}.
\end{eqnarray*}
This is to say that $Y(t,\omega)$ satisfies (\ref{mild}) with initial value $Y(s,\omega)$.
Now suppose that $u(t,s,\varphi_1,\omega)$ and $u(t,s,\varphi_2,\omega)$ are solutions of Eqn. (\ref{mild}) with 
$\mathcal{F}$-measurable initial values $\varphi_1$ and $\varphi_2$ respectively.  Then
\begin{eqnarray*}
&&|u(t,s,\varphi_1,\omega)-u(t,s,\varphi_2,\omega)|^2\\
&\leq &2 \|\Phi(t-s,\theta_s\omega)\|^2|\varphi_1-\varphi_2|^2\\
&&+2(T-s)\int_s^t\|\Phi(t-{\hat{s}}, \theta_{\hat{s}}\omega)\|^2\|\nabla F\|_{\infty}^2|u({\hat{s}},s,\varphi_1,\omega)-u({\hat{s}},s,\varphi_2,\omega)|^2\, d{\hat{s}},
\end{eqnarray*}
where
$$\|\nabla F\|^2_{\infty}:=\sup_{t\in \mathbb{R}, u\in \mathbb{R}^d}\|\nabla F(t,u)\|^2_{\mathcal{L}(\mathbb{R}^d)}.$$
For any $t>s$,
\begin{eqnarray*}
\|\Phi(t-s,\theta_s\omega)\|&=&\left\|\exp\big\{\dfrac{1}{2}(A+A^{*})(t-s)\big\}\right\|\left\|\exp\big\{\dfrac{1}{2}\textstyle\sum_{k=1}^{M}(B_k+B_k^{*})(W^k_t-W^k_s)\big\}\right\|\\
&\leq & {\rm e}^{\mu_1(t-s)}\prod_{k=1}^{M}\exp\left\{\|B\|(2C^k_{\delta,\omega}+|t|^{\delta}+|s|^{\delta})\right\}\\
&\leq & {\rm e}^{\mu_1(t-s)}\exp\left\{2M\|B\|\hat{T}+2\|B\|\textstyle\sum_{k=1}^{M}C^k_{\delta,\omega}\right\},
\end{eqnarray*}
where $\hat{T}:=\max\{|T|^{\delta}+|s|^{\delta},2|s|^{\delta}\}$, and the third line holds due to
the fact that there exists  $\Omega_1$ of full measure and a constant $\dfrac{1}{2}<\delta<1$ such that $$|W^k_t-W^k_s|\leq 2C^k_{\delta,\omega}+|t|^{\delta}+|s|^{\delta}.$$
Then for any $s\leq t\leq T$,
\begin{eqnarray*}
&&
|u(t,s,\varphi_1,\omega)-u(t,s,\varphi_2,\omega)|^2\\
&\leq &2 H_{\omega}(T-s)|\varphi_1-\varphi_2|^2+2(T-s)\|\nabla F\|_{\infty}^2H_{\omega}(T-s)\int_s^t|u({\hat{s}},s,\varphi_1,\omega)-u({\hat{s}},s,\varphi_2,\omega)|^2\, d{\hat{s}},
\end{eqnarray*}
where 
$$H_{\omega}(T-s)= {\rm e}^{2\mu_1(t-s)}\exp\left\{4M\|B\|\hat{T}+4\|B\|\textstyle\sum_{k=1}^{M}C^k_{\delta,\omega}\right\}.$$
Thus applying the Gronwall's inequality gives
$$|u(t,s,\varphi_1,\omega)-u(t,s,\varphi_2,\omega)|^2\leq  2H_{\omega}(T-s) |\varphi_1-\varphi_2|^2{\rm e}^{2\|\nabla F\|_{\infty}^2H_{\omega}(T-s)(T-s)^2}\ \ \mathbb{P}-\mbox{a.s.}$$
Now assume that $\varphi_1=\varphi_2$. Then it is easy to see that $u(t,s,\varphi_1,\omega)=u(t,s,\varphi_2,\omega)$ for any $\omega\in\Omega_1$ and $t\in [s,T]$. Hence from $\mathbb{P}(\Omega_1)=1$, 
$$\mathbb{P}\left\{u(t,s,\varphi_1,\omega)=u(t,s,\varphi_2,\omega)\ for\ any \ t\in \mathbb{Q}\cap [s,T]\right\}=1,$$
where $\mathbb{Q}$ is the set of rational numbers.
By the continuity of $t\to |u(t,s, \varphi_1,\omega)-u(t,s, \varphi_2,\omega)|$, it follows that
$$\mathbb{P}\left\{u(t,s,\varphi_1,\omega)=u(t,s,\varphi_2,\omega)\ for\ any \ \ t\in [s,T]\right\}=1.$$
This implies the uniqueness of solution of SDE (\ref{VT1}) within a finite time interval $[s, T]$.  
Then by Theorem \ref{Stratonovich} and
the uniqueness of the solution of the initial value problem (\ref{mild}), which is equivalent to (\ref{origT1}),
$$u(t,s, x,\omega)\Big|_{x=Y(s,\omega)}=u(t,s, Y(s,\omega),\omega)=Y(t,\omega).$$
The temperedness of $Y$ follows from the estimates (\ref{PhiBound}) and the boundedness of $F$. 

Conversely, assume Eqn. (\ref{origT1}) has a random periodic solution which is
 tempered from above. First note for any non-negative integer $l$, we have by Theorem \ref{Stratonovich},
\begin{eqnarray*}
 Y(t,\omega)
 &=&u(t\pm l\tau,t,Y(t,\theta _{\mp l\tau}\omega), \theta _{\mp l\tau}\omega)\\
 &=&\Phi(\pm l\tau, \theta _{t\mp l\tau}\omega)Y(t,\theta _{\mp l\tau}\omega)\\
 &&
 +\int _t^{t\pm l\tau}\Phi(t\pm l\tau-{\hat{s}}, \theta _{{\hat{s}}\mp l\tau}\omega)
 F({\hat{s}},u({\hat{s}},t,Y(t,\theta _{\mp l\tau}\omega),\theta _{\mp l\tau}\omega))d{\hat{s}}.
 \end{eqnarray*}
In particular,
\begin{eqnarray*}
P^-Y(t,\omega)&=&P^-
u(t+l\tau,t,Y(t,\theta _{-l\tau}\omega), \theta _{-l\tau}\omega)\\
&=&\Phi(l\tau, \theta _{t-l\tau}\omega)P^-Y(t,\theta _{-l\tau}\omega)+\int ^t_{t- l\tau}\Phi(t-{\hat{s}}, \theta _{{\hat{s}}}\omega)P^-
 F({\hat{s}},Y({\hat{s}},\omega))d{\hat{s}}\\
&\to&
\int _{-\infty}^t\Phi(t-{\hat{s}}, \theta _{{\hat{s}}}\omega)P^-
 F({\hat{s}},Y({\hat{s}},\omega))d{\hat{s}} \ \ \ \mathbb{P}-\mbox{a.s.},
\end{eqnarray*}
as $l\to+\infty$. The convergence deserves some justifications.
The convergence of the first term to $0$ as $l\to+\infty$ can be easily drawn from the estimate (\ref{PhiBound}) together with the tempered property of $|Y(t,\omega)|$ and $C(\omega)$. The 
convergence of the second term to the desired integral can be seen from the estimate of $\Phi$ and the boundedness of $F$.

Analogously, as $u$ is invertible,
\begin{eqnarray*}
P^+Y(t,\omega)&=&P^+
u(t-l\tau,t,Y(t,\theta _{l\tau}\omega), \theta _{l\tau}\omega)\\
&=&\Phi(-l\tau, \theta _{t+l\tau}\omega)P^+Y(t,\theta _{l\tau}\omega)-\int _t^{t+l\tau}\Phi(t-{\hat{s}}, \theta _{{\hat{s}}}\omega)P^+
 F({\hat{s}},Y({\hat{s}},\omega))d{\hat{s}}\\
&\to &
-\int ^{+\infty}_t\Phi(t-{\hat{s}}, \theta _{\hat{s}}\omega)P^+
 F({\hat{s}},Y({\hat{s}},\omega))d{\hat{s}} \ \ \ \mathbb{P}-\mbox{a.s.}
\end{eqnarray*}
as $l\to +\infty$. 
Therefore we have proved the converse part as $Y=P^+Y+P^-Y$. 
\end{proof}

\section{The existence of random periodic solutions and periodic measures}

After showing the equivalence of random periodic solutions of (\ref{origT1}) and the solutions of (\ref{VT1}), it remains to prove the existence of solutions to (\ref{VT1}). To check the relatively compactness is key to the proof of the the main result. In the following, we present the improved version of the Wiener-Sobolev compact embedding in \cite{F-zh1} with less conditions. We provide a brief proof in the Appendix for completeness. This kind of compactness in $L^2(\Omega)$ as a purely random variable
 version without including time and space variables was investigated in \cite{da-mall},\cite{peszat}, and in $L^2([a,b],L^2(\Omega))$ 
 was obtained in \cite{bally}. 

 \begin{thm}\label{B-S1}\upshape{(Relative Compactness in $C([a,b], L^2(\Omega))$.}
Consider a sequence $(v_n)_{n\in \mathbb{N}}$ of $C([a,b],L^2(\Omega))$. Suppose that:
\begin{enumerate}
\item $v_n(t,\cdot)\in {\cal D}^{1,2}$ and $\sup_{n\in \mathbb{N}}\sup_{t\in [a,b]} ||v_n(t,\cdot)||_{{1,2}}^2 <\infty$.
\item  There exists a constant $C>0$ such that for any $t, s\in [a,b]$,
         $$ \sup_n\mathbb{E}|v_n(t)-v_n(s)|^2 < C|t-s|.$$
\item (3i) There exists a constant $C>0$ such that for any $h_1 \in \mathbb{R}$, and any $t\in [a,b]$,
  $$\sup_n\int_{\mathbb{R}}\mathbb{E} |{\cal D}_{r+h_1}v_n(t)-{\cal D}_r v_n(t)|^2 dr<C|h_1|.$$
 (3ii) For any $\epsilon>0$, there exists $-\infty<\alpha<\beta<+\infty$ such that
$$\sup_n \sup_{t\in[a,b]}\int_{\mathbb{R}\backslash [\alpha,\beta]} \mathbb{E}|{\cal D}_r v_n(t)|^2 dr  <\epsilon.$$
\end{enumerate}
Then $\{v_n,\ n\in \mathbb{N}\}$ is relatively compact in $C([a,b],L^2(\Omega))$.
 \end{thm}
 
 The local existence theorem of random periodic solutions is presented below.
\begin{prop}\label{Main}
 Let $F:\mathbb{R} \times \mathbb{R}^d\to \mathbb{R}^d$ be in $C^3({\mathbb R}^{d+1})$, globally
bounded and the Jacobian $\nabla F(t,\cdot)$ be globally bounded, and $F(t,u)=F(t+\tau,u)$ for some fixed $\tau>0$, and  Condition (C) holds. Then there exists at least one ${\cal B}(\mathbb{R})\otimes\mathcal{F}$-measurable map
$Y^N: \mathbb{R}\times\Omega\rightarrow \mathbb{R}^d$ satisfying Eqn. (\ref{VNT1})
 and  $Y^N(t+\tau, \omega)=Y^N(t, \theta_{\tau}\omega)$ for
any $t\in \mathbb{R}$ and $\omega\in \Omega$.
\end{prop}

\begin{rmk}\label{rmk1}
It will be clear from the proof of this theorem that the commutativity Condition (C) is necessary only in the case when $A$ is hyperbolic with at least one eigenvalue having a positive real part and at least one eigenvalue having 
negative real part, as otherwise, projection operators are not needed.
\end{rmk}

The idea of its proof is to find a fixed point in some specific Banach space under Schauder's fixed point argument \cite{F-zh1}. The proof of this theorem is quite long, so we break into many parts. Firstly we define a Banach space $C^{\Lambda}_{\tau}(\mathbb{R}, L^2(\Omega,\mathbb{R}^d))$ 
 \begin{eqnarray}\label{BSpace}
 C^{\Lambda}_{\tau}(\mathbb{R}, L^2(\Omega, \mathbb{R}^d)):=\{f\in C^{\Lambda}(\mathbb{R}, L^2(\Omega,\mathbb{R}^d)):
 \ {\rm for \ any} \ 
t\in \mathbb{R}, 
f(t+\tau,\omega)=f(t,\theta_{\tau}\omega) \}.
\end{eqnarray}
Here the norm of the metric space $ C^{\Lambda}(\mathbb{R}, L^2(\Omega,\mathbb{R}^d))$ is given as follows,
$$\|f\|_{\Lambda}:=\sup_{t\in \mathbb{R}}{\rm e}^{-2\Lambda|t|}\|f(t,\cdot)\|_{L^2(\Omega,\mathbb{R}^d)},$$
which is indeed a weighted norm with $0<\Lambda< \frac{1}{4}\mu=\frac{1}{4}\min \{-\mu_{m+1},\mu_m\}$. Define a map ${\cal M}^N$: for any $Y^N\in  C^{\Lambda}_{\tau}(\mathbb{R}, L^2(\Omega, \mathbb{R}^d))$,  
\begin{eqnarray}\label{Map}
{\cal M}^N(Y^N)(t,\omega)&=&\int_{-\infty}^t\Phi^N(t-{\hat{s}},\theta_{\hat{s}} \omega)P^-F({\hat{s}},Y^N({\hat{s}},\omega))d{\hat{s}} \nonumber\\
&&
-\int^{+\infty}_t\Phi^N(t-{\hat{s}},\theta_{\hat{s}} \omega)P^+F({\hat{s}},Y^N({\hat{s}},\omega))d{\hat{s}}.
\end{eqnarray}


\begin{lem}\label{Lem1} Under the conditions of Proposition \ref{Main}, the map 
$${\cal M}^N: C^{\Lambda}_{\tau}(\mathbb{R}, L^2(\Omega, \mathbb{R}^d))\to  C^{\Lambda}_{\tau}(\mathbb{R}, L^2(\Omega, \mathbb{R}^d))$$
is continuous. 
\end{lem}
\begin{proof}
{\it \textbf{Step 1}}: We now show that ${\cal M}^N$ maps $ C^{\Lambda}_{\tau}(\mathbb{R}, L^2(\Omega, \mathbb{R}^d))$ into itself.
\begin{enumerate}
  \item[(A)] We first verify that for any $Y^N\in  C^{\Lambda}_{\tau}(\mathbb{R}, L^2(\Omega, \mathbb{R}^d))$,
   $$\sup\limits_{t\in \mathbb{R}}{\rm e}^{-2\Lambda|t|}\mathbb{E}|{\cal M}^N(Y^N)(t,\cdot)|^2<\infty.$$ 
   Actually  by (\ref{PHI-B}) and (\ref{PHI+B}) we have that
  \begin{eqnarray*}
&&{\rm e}^{-2\Lambda|t|}\mathbb{E}|{\cal M}^N(Y^N)(t,\cdot)|^2 \\
&\leq &2{\rm e}^{-2\Lambda|t|}\mathbb{E}\Big|\int_{-\infty}^t\Phi^N_{t-{\hat{s}},{\hat{s}}}P^-F({\hat{s}},Y^N)d{\hat{s}}\Big|^2+2{\rm e}^{-\Lambda|t|}\mathbb{E}\Big|\int^{+\infty}_t\Phi^N_{t-{\hat{s}},{\hat{s}}}P^+F({\hat{s}},Y^N)d{\hat{s}}\Big|^2 \\
  & \leq&2{\rm e}^{-2\Lambda|t|}\|F\|_{\infty}^2\bigg\{\mathbb{E}\Big(\int_{-\infty}^t\|\Phi^N_{t-{\hat{s}},{\hat{s}}}P^-\|d{\hat{s}}\Big)^2+\mathbb{E}\Big(\int^{+\infty}_t\|\Phi^N_{t-{\hat{s}},{\hat{s}}}P^+\|d{\hat{s}}\Big)^2\bigg\}\\
  & \leq&2N^2\|F\|_{\infty}^2{\rm e}^{-2\Lambda|t|}\bigg\{\Big(\int_{-\infty}^t{\rm e}^{\frac{1}{2}\mu_{m+1}(t-{\hat{s}})}{\rm e}^{\Lambda|\hat{s}|}d{\hat{s}}\Big)^2+\Big(\int^{+\infty}_t{\rm e}^{\frac{1}{2}\mu_{m}(t-{\hat{s}})}{\rm e}^{\Lambda|\hat{s}|}d{\hat{s}}\Big)^2\bigg\}.
\end{eqnarray*}
Here $\Phi_{t,\hat{s}}P^{\pm}$ is the shorthand for $\Phi(t, \theta_{\hat{s}}\omega)P^{\pm}$, and $\Phi_{t}P^{\pm}$ is the shorthand for $\Phi(t, \omega)P^{\pm}$. Note that ${\rm e}^{\Lambda|\hat{s}|}\leq {\rm e}^{-\Lambda\hat{s}}+{\rm e}^{\Lambda\hat{s}}$,  ${\rm e}^{-2\Lambda|t|}\leq {\rm e}^{-2\Lambda t}$ 
and ${\rm e}^{-2\Lambda|t|}\leq {\rm e}^{2\Lambda t}$ for all $\hat{s},t\in \mathbb{R}$. The first integral in the above can be estimated as
\begin{eqnarray*}
&&
{\rm e}^{-2\Lambda|t|}\Big(\int_{-\infty}^t{\rm e}^{\frac{1}{2}\mu_{m+1}(t-{\hat{s}})}{\rm e}^{\Lambda|\hat{s}|}d{\hat{s}}\Big)^2\\
&\leq & {\rm e}^{-2\Lambda|t|}\Big(\int_{-\infty}^t{\rm e}^{\frac{1}{2}\mu_{m+1}(t-{\hat{s}})}{\rm e}^{-\Lambda\hat{s}}d{\hat{s}}\Big)^2+\Big(\int_{-\infty}^t{\rm e}^{\frac{1}{2}\mu_{m+1}(t-{\hat{s}})}{\rm e}^{\Lambda\hat{s}}d{\hat{s}}\Big)^2\\
&\leq &\Big(\int_{-\infty}^t{\rm e}^{(\frac{1}{2}\mu_{m+1}+\Lambda)(t-{\hat{s}})}d{\hat{s}}\Big)^2+\Big(\int_{-\infty}^t{\rm e}^{(\frac{1}{2}\mu_{m+1}-\Lambda)(t-{\hat{s}})}d{\hat{s}}\Big)^2\\
&\leq & \frac{1}{(\mu_{m+1}+2\Lambda)^2}+\frac{1}{(\mu_{m+1}-2\Lambda)^2}.
\end{eqnarray*}
The second integral integral can be estimated similarly.  Putting them together, we have
\begin{eqnarray*}
&&{\rm e}^{-2\Lambda|t|}\mathbb{E}|{\cal M}^N(Y^N)(t,\cdot)|^2\\
&\leq &
8N^2\|F\|_{\infty}^2 \bigg\{\frac{1}{(\mu_{m+1}+2\Lambda)^2}+\frac{1}{(\mu_{m+1}-2\Lambda)^2}+\frac{1}{(\mu_{m}+2\Lambda)^2}+\frac{1}{(\mu_{m}-2\Lambda)^2}\bigg\}.
\end{eqnarray*}
  \item[(B)]  Next we show that ${\cal M}^N (Y^N)(\cdot, \omega)$ is continuous from $\mathbb{R}$ to $L^2(\Omega, \mathbb{R}^d)$ for any given $Y^N\in C^{\Lambda}_{\tau}(\mathbb{R},L^2(\Omega, \mathbb{R}^d))$. First note for any $t_1$, $t_2 \in \mathbb{R}$ with $t_1\leq t_2$,
\begin{eqnarray*}
&&\mathbb{E}|{\cal M}^N(Y^N)(t_1)-{\cal M}^N(Y^N)(t_2)|^2\\
&\leq &4\mathbb{E}\bigg[\Big|\int_{-\infty}^{t_1}(\Phi^N_{t_1-{\hat{s}},{\hat{s}}}P^--\Phi^N_{t_2-{\hat{s}},{\hat{s}}}P^-)F({\hat{s}},Y^N)d{\hat{s}}\Big|^2+\Big|\int_{t_1}^{t_2}\Phi^N_{t_2-{\hat{s}},{\hat{s}}}P^-F({\hat{s}},Y^N)d{\hat{s}}\Big|^2\\
&&+\Big|\int^{+\infty}_{t_2}(\Phi^N_{t_1-{\hat{s}},{\hat{s}}}P^+-\Phi^N_{t_2-{\hat{s}},{\hat{s}}}P^+)F({\hat{s}},Y^N)d{\hat{s}}\Big|^2+\Big|\int_{t_1}^{t_2}\Phi^N_{t_1-{\hat{s}},{\hat{s}}}P^+F({\hat{s}},Y^N)d{\hat{s}}\Big|^2\bigg]\\
&=:&\sum_{i=1}^4 T_i.
\end{eqnarray*}
It is easy to check that
\begin{eqnarray*}
T_2=4\mathbb{E}\Big|\int_{t_1}^{t_2}\Phi^N_{t_2-{\hat{s}},{\hat{s}}}P^-F({\hat{s}},Y^N(\hat{s},\cdot))d{\hat{s}}\Big|^2
&\leq &4N^2\|F\|_{\infty}^2 \Big(\int_{t_1}^{t_2} {\rm e}^{\frac{1}{2}\mu_{m+1}(t_2-\hat{s})}{\rm e}^{\Lambda|\hat{s}|}d{\hat{s}}\Big)^2\\
&\leq & 4N^2\|F\|_{\infty}^2\max\{{\rm e}^{2\Lambda|t_2|},{\rm e}^{2\Lambda|t_1|}\}|t_2-t_1|^2, 
\end{eqnarray*}
and similarly
\begin{eqnarray*}
T_4=4\mathbb{E}\Big|\int_{t_1}^{t_2}\Phi^N_{t_1-{\hat{s}},{\hat{s}}}P^+F({\hat{s}},Y^N(\hat{s},\cdot))d{\hat{s}}\Big|^2\leq 4N^2\|F\|_{\infty}^2\max\{{\rm e}^{2\Lambda|t_2|},{\rm e}^{2\Lambda|t_1|}\}|t_2-t_1|^2.
\end{eqnarray*}
As for $T_1$, we have the following inequalities through the estimates in Lemma \ref{LEMMA2C3}, 
\begin{eqnarray*}
T_1&:=&4\mathbb{E}\Big|\int_{-\infty}^{t_1}(\Phi^N_{t_1-{\hat{s}},{\hat{s}}}P^--\Phi^N_{t_2-{\hat{s}},{\hat{s}}}P^-)F({\hat{s}},Y^N(\hat{s},\omega))d{\hat{s}}\Big|^2\\
&\leq & 8\mathbb{E}\Big|\int_{-\infty}^{t_1}(\Phi_{t_1-{\hat{s}},{\hat{s}}}P^--\Phi_{t_2-{\hat{s}},{\hat{s}}}P^-)\min\Big\{1,\frac{N{\rm e}^{\frac{1}{2}\mu_{m+1}(t_1-\hat{s})}{\rm e}^{\Lambda|\hat{s}|}}{\|\Phi_{t_1-{\hat{s}},{\hat{s}}}P^-\|}\Big\}F({\hat{s}},Y^N)d{\hat{s}}\Big|^2\\
&+&8\mathbb{E}\Big|\int_{-\infty}^{t_1}\Phi_{t_2-{\hat{s}},{\hat{s}}}P^-\Bigg( \min\Big\{1,\frac{N{\rm e}^{\frac{1}{2}\mu_{m+1}(t_1-\hat{s})}{\rm e}^{\Lambda|\hat{s}|}}{\|\Phi_{t_1-{\hat{s}},{\hat{s}}}P^-\|}\Big\}\\
&&\hspace{3.3cm}-\min\Big\{1,\frac{N{\rm e}^{\frac{1}{2}\mu_{m+1}(t_2-\hat{s})}{\rm e}^{\Lambda|\hat{s}|}}{\|\Phi_{t_2-{\hat{s}},{\hat{s}}}P^-\|}\Big\}\Bigg) F({\hat{s}},Y^N(\hat{s},\omega))d{\hat{s}}\Big|^2.
\end{eqnarray*}
By using inequality $|\min\{1,a\}-\min\{1,b\}|\leq |a-b|$ whenever $a,b \geq 0$, so
for $s<t_1<t_2$ we have
\begin{eqnarray*}
&&\left|\min\Big\{1,\frac{N{\rm e}^{\frac{1}{2}\mu_{m+1}(t_1-\hat{s})}{\rm e}^{\Lambda|\hat{s}|}}{\|\Phi_{t_1-{\hat{s}},{\hat{s}}}P^-\|}\Big\}-\min\Big\{1,\frac{N{\rm e}^{\frac{1}{2}\mu_{m+1}(t_2-\hat{s})}{\rm e}^{\Lambda|\hat{s}|}}{\|\Phi_{t_2-{\hat{s}},{\hat{s}}}P^-\|}\Big\}\right|\\
&\leq & \left|\frac{N{\rm e}^{\frac{1}{2}\mu_{m+1}(t_1-\hat{s})}{\rm e}^{\Lambda|\hat{s}|}}{\|\Phi_{t_1-{\hat{s}},{\hat{s}}}P^-\|}-\frac{N{\rm e}^{\frac{1}{2}\mu_{m+1}(t_2-\hat{s})}{\rm e}^{\Lambda|\hat{s}|}}{\|\Phi_{t_2-{\hat{s}},{\hat{s}}}P^-\|}\right|\\
&\leq &\Bigg|\frac{N{\rm e}^{\frac{1}{2}\mu_{m+1}(t_1-\hat{s})}{\rm e}^{\Lambda|\hat{s}|}}{\|\Phi_{t_1-{\hat{s}},{\hat{s}}}P^-\|}-\frac{N{\rm e}^{\frac{1}{2}\mu_{m+1}(t_2-\hat{s})}{\rm e}^{\Lambda|\hat{s}|}}{\|\Phi_{t_1-{\hat{s}},{\hat{s}}}P^-\|}\Bigg|
+\Bigg|\frac{N{\rm e}^{\frac{1}{2}\mu_{m+1}(t_2-\hat{s})}{\rm e}^{\Lambda|\hat{s}|}}{\|\Phi_{t_1-{\hat{s}},{\hat{s}}}P^-\|}-\frac{N{\rm e}^{\frac{1}{2}\mu_{m+1}(t_2-\hat{s})}{\rm e}^{\Lambda|\hat{s}|}}{\|\Phi_{t_2-{\hat{s}},{\hat{s}}}P^-\|}\Bigg|\\
&\leq &\frac{N{\rm e}^{\Lambda|\hat{s}|}}{\|\Phi_{t_1-{\hat{s}},{\hat{s}}}P^-\|}\Big({\rm e}^{\frac{1}{2}\mu_{m+1}(t_1-\hat{s})}-{\rm e}^{\frac{1}{2}\mu_{m+1}(t_2-\hat{s})}\Big)+N{\rm e}^{\Lambda|\hat{s}|{\rm e}^{\frac{1}{2}\mu_{m+1}(t_2-\hat{s})}}\Bigg|\frac{\|\Phi_{t_1-{\hat{s}},{\hat{s}}}P^-\|-\|\Phi_{t_2-{\hat{s}},{\hat{s}}}P^-\|}{\|\Phi_{t_2-{\hat{s}},{\hat{s}}}P^-\|\|\Phi_{t_1-{\hat{s}},{\hat{s}}}P^-\|}\Bigg|\\
&\leq & \frac{N{\rm e}^{\Lambda|\hat{s}|}{\rm e}^{\frac{1}{2}\mu_{m+1}(t_1-\hat{s})}}{\|\Phi_{t_2-{\hat{s}},{\hat{s}}}P^-\|}\left((1-{\rm e}^{\frac{1}{2}\mu_{m+1}(t_2-t_1)})\|\Phi_{t_2-{t_1},{t_1}}P^-\|+\|\Phi_{t_2-t_1,t_1}P^--P^-\|\right).
\end{eqnarray*}
Therefore 
\begin{eqnarray*}
T_1
&\leq& 384N^2\|F\|_{\infty}^2{\rm e}^{2\Lambda|t_1|}\Big(\frac{1}{|\mu_{m+1}+2\Lambda|^2}+\frac{1}{|\mu_{m+1}-2\Lambda|^2}\Big)\\
&&\cdot[\mathbb{E}\|\Phi_{t_2-t_1,t_1}P^--P^-\|^2+\mu_{m+1}^2(t_2-t_1)^2\mathbb{E}\|\Phi_{t_2-t_1,t_1}P^-\|^2]\\
&\leq &CN^2\|F\|_{\infty}^2{\rm e}^{2\Lambda|t_1|}{\rm e}^{2\|A\||t_2-t_1|+2M\|B\|^2|t_2-t_1|}\Big(\frac{1}{|\mu_{m+1}+2\Lambda|^2}+\frac{1}{|\mu_{m+1}-2\Lambda|^2}\Big)\\
&&\cdot[(1+\mu^2_{m+1})|t_2-t_1|^2+|t_2-t_1|],
\end{eqnarray*}
where the last inequality follows from Lemma \ref{LEMMA2C3}.
Similarly,
\begin{eqnarray*}
T_3&:=&4\mathbb{E}\Big|\int^{+\infty}_{t_2}(\Phi^N_{t_1-{\hat{s}},{\hat{s}}}P^+
-\Phi^N_{t_2-{\hat{s}},{\hat{s}}}P^+)F({\hat{s}},Y^N(\hat{s},\cdot))d{\hat{s}}\Big|^2\\
&\leq &CN^2\|F\|_{\infty}^2{\rm e}^{2\Lambda|t_2|}{\rm e}^{2\|A\||t_2-t_1|+2M\|B\|^2|t_2-t_1|}\Big(\frac{1}{|\mu_{m}+2\Lambda|^2}+\frac{1}{|\mu_{m}-2\Lambda|^2}\Big)\\
&&\cdot[(1+\mu^2_{m})|t_2-t_1|^2+|t_2-t_1|].
\end{eqnarray*}
  \item[(C)]   We show that ${\cal M}^N(Y^N)(t,\theta _{\pm\tau}\omega)={\cal M}^N(Y^N)(t\pm\tau,\omega)$: similar as in \cite{F-zh1}, as $Y^N(t+\tau,\omega)=Y^N(t,\theta_{\tau}\omega)$, so
  \begin{eqnarray*}
  &&
  {\cal M}^N(Y^N)(t,\theta_{\tau}\omega)\\
  &=&\int_{-\infty}^t\Phi^N_{t-{\hat{s}},\hat{s}+\tau}P^-F({\hat{s}},Y^N({\hat{s}},\theta_{\tau}\omega))d{\hat{s}} -\int^{+\infty}_t\Phi^N_{t-\hat{s},\hat{s}+\tau}P^+F(\hat{s},Y^N(\hat{s},\theta_{\tau}\omega))d\hat{s}\\
  &=&\int_{-\infty}^t\Phi^N_{(t+\tau)-({\hat{s}}+\tau),\hat{s}+\tau}P^-F({\hat{s}}+\tau,Y^N({\hat{s}}+\tau,\omega))d{\hat{s}} \\
    & &-\int^{+\infty}_t\Phi^N_{(t+\tau)-({\hat{s}}+\tau),\hat{s}+\tau}P^+F({\hat{s}}+\tau,Y^N({\hat{s}}+\tau,\omega))d{\hat{s}}\\
      &=&\int_{-\infty}^{t+\tau}\Phi^N_{(t+\tau)-\hat{h},\hat{h}}P^-F(\hat{h},Y^N(\hat{h},\omega))d\hat{h} -\int^{+\infty}_{t+\tau}\Phi^N_{(t+\tau)-\hat{h},\hat{h}}P^+F(\hat{h},Y^N(\hat{h},\omega))d\hat{h}\\
        &=&{\cal M}^N(Y^N)(t+\tau,\omega).
  \end{eqnarray*}
  \end{enumerate}
Thus we completed the Step 1 and proved that ${\cal M}^N$ maps $C^{\Lambda}_{\tau}(\mathbb{R}, L^2(\Omega,\mathbb{R}^d))$ into itself.

{\it \textbf{Step 2}}: We now check the continuity of the map ${\cal M}^N: C^{\Lambda}_{\tau}(\mathbb{R}, L^2(\Omega,\mathbb{R}^d))\to  C^{\Lambda}_{\tau}(\mathbb{R}, L^2(\Omega,\mathbb{R}^d))$. For $Y_1^N,Y_2^N\in
C^{\Lambda}_{\tau}(\mathbb{R}, L^2(\Omega,\mathbb{R}^d))$ and $t\in [j\tau,(j+1)\tau)$ for some $j\in \mathbb{Z}$, we have 
\begin{eqnarray*}
&&{\rm e}^{-2\Lambda|t|}\mathbb{E}|{\cal M}^N(Y^N_1)(t,\cdot)-{\cal M}^N(Y^N_2)(t,\cdot)|^2\\
&\leq &2{\rm e}^{-2\Lambda|t|}\mathbb{E}\Big|\int_{-\infty}^t\Phi^N_{t-\hat{s},\hat{s}}P^-F(\hat{s},Y_1^N(\hat{s},\cdot))d\hat{s}-\int_{-\infty}^t\Phi^N_{t-\hat{s},\hat{s}}P^-F(\hat{s},Y_2^N(\hat{s},\cdot))d\hat{s}\Big|^2\\
&&+2{\rm e}^{-2\Lambda|t|}\mathbb{E}\Big|\int^{+\infty}_t\Phi^N_{t-\hat{s},\hat{s}}P^+F(\hat{s},Y_1^N(\hat{s},\cdot))d\hat{s}-\int^{+\infty}_t\Phi^N_{t-\hat{s},\hat{s}}P^+F(\hat{s},Y_2^N(\hat{s},\cdot))d\hat{s}\Big|^2,\\
&:=&\hat{T}_1+\hat{T}_2.
\end{eqnarray*}
By using the Cauchy-Schwarz inequality we have that
\begin{eqnarray*}
\hat{T}_1
&\leq &2\|\nabla  F\|_{\infty}^2{\rm e}^{-2\Lambda|t|}\mathbb{E}\Big(\int_{-\infty}^t\|\Phi^N_{t-\hat{s},\hat{s}}P^-\||Y_1^N(\hat{s},\cdot)-Y_2^N(\hat{s},\cdot)|d\hat{s}\Big)^2\\
&\leq & 4N^2\|\nabla  F\|_{\infty}^2\mathbb{E}\Big(\int_{-\infty}^t{\rm e}^{(\frac{1}{2}\mu_{m+1}-\Lambda)(t-\hat{s})}|Y_1^N(\hat{s},\cdot)-Y_2^N(\hat{s},\cdot)|d\hat{s}\Big)^2\\
&&+4N^2\|\nabla  F\|_{\infty}^2\mathbb{E}\Big(\int_{-\infty}^t{\rm e}^{(\frac{1}{2}\mu_{m+1}+\Lambda)(t-\hat{s})}|Y_1^N(\hat{s},\cdot)-Y_2^N(\hat{s},\cdot)|d\hat{s}\Big)^2\\
&\leq &\frac{8}{|\mu_{m+1}-2\Lambda|}N^2\|\nabla  F\|_{\infty}^2\int_{-\infty}^t{\rm e}^{(\frac{1}{2}\mu_{m+1}-\Lambda)(t-\hat{s})}\mathbb{E}|Y_1^N(\hat{s},\cdot)-Y_2^N(\hat{s},\cdot)|^2d\hat{s}\\
&&+\frac{8}{|\mu_{m+1}+2\Lambda|}N^2\|\nabla  F\|_{\infty}^2\int_{-\infty}^t{\rm e}^{(\frac{1}{2}\mu_{m+1}+\Lambda)(t-\hat{s})}\mathbb{E}|Y_1^N(\hat{s},\cdot)-Y_2^N(\hat{s},\cdot)|^2d\hat{s}.
\end{eqnarray*}
Note that $\mathbb{E}|Y_1^N(\hat{s},\cdot)-Y_2^N(\hat{s},\cdot)|^2$ is a nonegative periodic function in $C^{\Lambda}(\mathbb{R})$ with period $\tau$
as
\begin{equation*}
\mathbb{E}|Y_1^N(\hat{s}+\tau,\cdot)-Y_2^N(\hat{s}+\tau,\cdot)|^2= \mathbb{E}|Y_1^N(\hat{s},\theta_{\tau}\cdot)-Y_2^N(\hat{s},\theta_{\tau}\cdot)|^2=\mathbb{E}|Y_1^N(\hat{s},\cdot)-Y_2^N(\hat{s},\cdot)|^2.
\end{equation*}
Then we have
\begin{eqnarray*}
&&\int_{-\infty}^t{\rm e}^{(\frac{1}{2}\mu_{m+1}\pm\Lambda)(t-\hat{s})}\mathbb{E}|Y_1^N(\hat{s},\cdot)-Y_2^N(\hat{s},\cdot)|^2d\hat{s}\\
&\leq &  \sup_{s\in[0,\tau)}\mathbb{E}|Y_1^N(s,\cdot)-Y_2^N(s,\cdot)|^2\int_{-\infty}^t{\rm e}^{(\frac{1}{2}\mu_{m+1}\pm\Lambda)(t-\hat{s})}d\hat{s}\\
&\leq &\frac{2{\rm e}^{2\Lambda\tau}}{|\mu_{m+1}\pm 2\Lambda|} \sup_{s\in[0,\tau)}{\rm e}^{-2\Lambda|s|}\mathbb{E}|Y_1^N(s,\cdot)-Y_2^N(s,\cdot)|^2.
\end{eqnarray*}
This leads to
\begin{eqnarray*}
\hat{T}_1 &\leq & 6N^2\|\nabla F\|_{\infty}^2{\rm e}^{2\Lambda\tau}\Big\{\frac{1}{(\mu_{m+1}-2\Lambda)^2}+\frac{1}{(\mu_{m+1}+2\Lambda)^2}\Big\}\sup_{\hat{s}\in \mathbb{R}}{\rm e}^{-2\Lambda|\hat{s}|}\mathbb{E}|Y_1^N(\hat{s},\cdot)-Y_2^N(\hat{s},\cdot)|^2,
\end{eqnarray*}
Similarly
\begin{eqnarray*}
\hat{T}_2&\leq &16N^2\|\nabla F\|_{\infty}^2{\rm e}^{2\Lambda\tau}\Big\{\frac{1}{(\mu_{m}-2\Lambda)^2}+\frac{1}{(\mu_{m}+2\Lambda)^2}\Big\}\sup_{\hat{s}\in \mathbb{R}}{\rm e}^{-2\Lambda|\hat{s}|}\mathbb{E}|Y_1^N(\hat{s},\cdot)-Y_2^N(\hat{s},\cdot)|^2.
\end{eqnarray*}
Therefore the continuity of ${\cal M}^N: C^{\Lambda}_{\tau}(\mathbb{R}, L^2(\Omega,\mathbb{R}^d))\to
C^{\Lambda}_{\tau}(\mathbb{R}, L^2(\Omega,\mathbb{R}^d))$ is verified. 
\end{proof}
\begin{rmk} One can see that it is crucial to use the truncation of the tempered random variable $C(\omega)$ in the Step 2 of the proof. Otherwise, it would be difficult to separate $\|\Phi_{t-\hat{s},\hat{s}}^NP^\pm\|^2$ and $|Y^N_1(\hat{s},\omega)-Y^N_2(\hat{s},\omega)|^2$ inside the integrals in $\hat{T}_1$ and $\hat{T}_2$, where H\"older's inequality seems losing its power here.
Needless to say that a key step to make it work is to remove the truncation eventually.   
\end{rmk}

\begin{lem}\label{LEMMAMD} Given $\Phi^N(t,\theta_{\hat{s}}\omega)P^{\pm}$ defined by (\ref{PHI+B}) and (\ref{PHI-B}), the Malliavin derivatives of $\Phi^N(t,\theta_{\hat{s}}\omega)P^{\pm}$ with respect to the $l$-th Brownian motion, $l\in \{1,2,\cdots,M\}$, are given by: when $t\geq 0$ 
\begin{eqnarray}\label{md1}
&&{\cal D}_r^l\Phi^N(t,\theta_{\hat{s}}\omega)P^-\nonumber\\
&=& \chi_{\{\hat{s}\leq r\leq t+\hat{s}\}}(r)\Bigg\{B_l\Phi(t,\theta_{\hat{s}}\omega)P^-\min\left\{1,  \frac{N{\rm e}^{\frac{1}{2}\mu_{m+1}t}{\rm e}^{\Lambda|\hat{s}|}}{\| \Phi(t,\theta_{\hat{s}}\omega)P^{-}\|}\right\}\nonumber\\
&&\hspace{2.5cm}-\chi_{\{\| \Phi(t,\theta_{\hat{s}}\omega)P^{-}\|>N{\rm e}^{\frac{1}{2}\mu_{m+1}t}{\rm e}^{\Lambda|\hat{s}|}\}}(\omega)\frac{N{\rm e}^{\frac{1}{2}\mu_{m+1}t}{\rm e}^{\Lambda|\hat{s}|}}{\| \Phi(t,\theta_{\hat{s}}\omega)P^{-}\|^3}\nonumber\\
&&\hspace{2.9cm}\cdot \Big(\sum_{i,j=1}^{d}(\Phi(t,\theta_{\hat{s}}\omega)P^-)_{ij}\sum_{k=1}^d(B_l)_{ik}(\Phi(t,\theta_{\hat{s}}\omega)P^-)_{kj}\Big)\Phi(t,\theta_{\hat{s}}\omega)P^-\Bigg\}
\end{eqnarray} 
with the estimate that
\begin{equation}\label{MDesti1}
\|{\cal D}_r^l\Phi^N(t,\theta_{\hat{s}}\omega)P^-\|\leq (1+d^3)\|B\|N{\rm e}^{\frac{1}{2}\mu_{m+1}(t-\hat{s})}{\rm e}^{\Lambda|\hat{s}|},
\end{equation}
and when $t\leq 0$ 
\begin{eqnarray}\label{md2}
&&{\cal D}_r^l\Phi^N(t,\theta_{\hat{s}}\omega)P^+\nonumber\\
&=&\chi_{\{t+\hat{s}\leq r\leq \hat{s} \}}(r)\Bigg\{-B_l\Phi(t,\theta_{\hat{s}}\omega)P^+\min\left\{1,  \frac{N{\rm e}^{\frac{1}{2}\mu_{m}t}{\rm e}^{\Lambda|\hat{s}|}}{\| \Phi(t,\theta_{\hat{s}}\omega)P^{+}\|}\right\}\nonumber\\
&&\hspace{2.5cm}+\chi_{\{\| \Phi(t,\theta_{\hat{s}}\omega)P^{+}\|>N{\rm e}^{\frac{1}{2}\mu_{m}t}{\rm e}^{\Lambda|\hat{s}|}\}}(\omega) \frac{N{\rm e}^{\frac{1}{2}\mu_{m}t}{\rm e}^{\Lambda|\hat{s}|}}{\| \Phi(t,\theta_{\hat{s}}\omega)P^{+}\|^3}\nonumber\\
&&\hspace{2.9cm}\cdot \Big(\sum_{i,j=1}^{d}(\Phi(t,\theta_{\hat{s}}\omega)P^+)_{ij}\sum_{k=1}^d(B_l)_{ik}(\Phi(t,\theta_{\hat{s}}\omega)P^+)_{kj}\Big)\Phi(t,\theta_{\hat{s}}\omega)P^+\Bigg\}
\end{eqnarray}
with the estimate that
\begin{equation}\label{MDesti2}
\|{\cal D}_r^l\Phi^N(t,\theta_{\hat{s}}\omega)P^+\|\leq (1+d^3)\|B\|N{\rm e}^{\frac{1}{2}\mu_{m}(t-\hat{s})}{\rm e}^{\Lambda|\hat{s}|}.
\end{equation}
\end{lem}
\begin{proof}
We can calculate the Malliavin derivatives of $\Phi^N$ by the chain rule:
when $t\geq 0$, from Proposition 1.2.3 and Proposition 1.2.4 in \cite{nualart2} (or directly obtained from the proof of Proposition 2.1.10 in \cite{nualart2}), we know that $\varphi(F):=\min\{1,F\}\in \mathcal{D}^{1,2}$ if $F\in \mathcal{D}^{1,2}$, and for fixed $t$ and $s$ we have that
\begin{eqnarray}\label{340}
\mathcal{D}^l_r \min\left\{1,  \frac{N{\rm e}^{\frac{1}{2}\mu_{m+1}t}{\rm e}^{\Lambda|\hat{s}|}}{\| \Phi(t,\theta_{\hat{s}}\omega)P^{-}\|}\right\}=\chi_{\{\| \Phi(t,\theta_{\hat{s}}\omega)P^{-}\|>N{\rm e}^{\frac{1}{2}\mu_{m+1}t}{\rm e}^{\Lambda|\hat{s}|}\}}(\omega)\mathcal{D}^l_r\frac{N{\rm e}^{\frac{1}{2}\mu_{m+1}t}{\rm e}^{\Lambda|\hat{s}|}}{\| \Phi(t,\theta_{\hat{s}}\omega)P^{-}\|},
\end{eqnarray}
Thus, for $l\in \{1,2,\cdots,M\}$,
\begin{eqnarray*}
&&{\cal D}_r^l\Phi^N(t,\theta_{\hat{s}}\omega)P^-\nonumber\\
&=&
{\cal D}_r^l\big(\Phi(t,\theta_{\hat{s}}\omega)P^-\big)\min\left\{1,  \frac{N{\rm e}^{\frac{1}{2}\mu_{m+1}t}{\rm e}^{\Lambda|\hat{s}|}}{\| \Phi(t,\theta_{\hat{s}}\omega)P^{-}\|}\right\}+\Phi(t,\theta_{\hat{s}}\omega)P^-{\cal D}_r^l\min\left\{1,  \frac{N{\rm e}^{\frac{1}{2}\mu_{m+1}t}{\rm e}^{\Lambda|\hat{s}|}}{\| \Phi(t,\theta_{\hat{s}}\omega)P^{-}\|}\right\}\nonumber\\
&=&{\cal D}_r^l \left(\exp\{At+\textstyle\sum_{k=1}^MB_k\theta_{\hat{s}}(W_t)\}P^-\right)\min\left\{1,  \frac{N{\rm e}^{\frac{1}{2}\mu_{m+1}t}{\rm e}^{\Lambda|\hat{s}|}}{\| \Phi(t,\theta_{\hat{s}}\omega)P^{-}\|}\right\}\nonumber\\
&&-\chi_{\{\| \Phi(t,\theta_{\hat{s}}\omega)P^{-}\|>N{\rm e}^{\frac{1}{2}\mu_{m+1}t}{\rm e}^{\Lambda|\hat{s}|}\}}(\omega)\Phi(t,\theta_{\hat{s}}\omega)P^- \frac{N{\rm e}^{\frac{1}{2}\mu_{m+1}t}{\rm e}^{\Lambda|\hat{s}|}}{\| \Phi(t,\theta_{\hat{s}}\omega)P^{-}\|^2}{\cal D}_r^l\| \Phi(t,\theta_{\hat{s}}\omega)P^{-}\|.
\end{eqnarray*}
Note now the equivalence of the matrix norm
 $$\| \Phi(t,\theta_{\hat{s}}\omega)P^{-}\|:=\sqrt{\sum_{i,j=1}^d(\Phi(t,\theta_{\hat{s}}\omega)P^-)^2_{ij}},$$ where $(J)_{ij}$ stands for the $ij$th element of the matrix $J$, and 
$$\mathcal{D}^l_r(\Phi(t,\theta_{\hat{s}}\omega)P^-)_{ij}=\sum_{k=1}^d(B_l)_{ik}(\Phi(t,\theta_{\hat{s}}\omega)P^-)_{kj}.$$
Thus by the chain rule we have
\begin{eqnarray*}
{\cal D}_r^l\| \Phi(t,\theta_{\hat{s}}\omega)P^{-}\|&=&\frac{1}{\| \Phi(t,\theta_{\hat{s}}\omega)P^{-}\|}\sum_{i,j=1}^{d}(\Phi(t,\theta_{\hat{s}}\omega)P^-)_{ij}{\cal D}_r^l(\Phi(t,\theta_{\hat{s}}\omega)P^-)_{ij}\\
&= & \frac{1}{\| \Phi(t,\theta_{\hat{s}}\omega)P^{-}\|}\Big(\sum_{i,j=1}^{d}(\Phi(t,\theta_{\hat{s}}\omega)P^-)_{ij}\sum_{k=1}^d(B_l)_{ik}(\Phi(t,\theta_{\hat{s}}\omega)P^-)_{kj}\Big).
\end{eqnarray*}
It is easy to verify (\ref{MDesti1}). 
When $t\leq 0$, (\ref{md2}) and (\ref{MDesti2}) can be derived analogously. 
\end{proof}

Next we introduce a subset of $C^{\Lambda}_{\tau}(\mathbb{R}, L^2(\Omega,\mathbb{R}^d))$ as follows,
\begin{eqnarray*}
&&
C^{\Lambda}_{\tau}(\mathbb{R},{\cal D}^{1,2})\\
&:=&\bigg\{ f\in C^{\Lambda}_{\tau}(\mathbb{R}, L^2(\Omega,\mathbb{R}^d)):\ f|_{[0,\tau)}\in C([0,\tau),{\cal D}^{1,2}),
 \sup\limits_{\substack{t\in [0,\tau)}}\mathbb{E}\int_{\mathbb{R}}|{\cal D}_{r}^lf(t,\cdot)|^2dr<\infty,\\
&& 
\ \ \ \ \ \ \ \ \ \ \ \  \sup\limits_{\substack{t\in [0,\tau),\delta \in \mathbb{R}}} \dfrac{1}{|\delta|}\mathbb{E}\int_{\mathbb{R}}|{\cal D}_{r+\delta}^lf(t,\cdot)-{\cal D}_{r}^lf(t,\cdot)|^2dr< \infty , \ l\in \{1,\cdots, M\}
\nonumber\bigg\}.
\end{eqnarray*}

\begin{lem}\label{LEMBE3} Under the conditions of Proposition \ref{Main}, we have
$$\mathcal{M}^N(C^{\Lambda,N}_{\tau,\rho}(\mathbb{R},{\cal D}^{1,2}))\subset C^{\Lambda,N}_{\tau,\rho}(\mathbb{R},{\cal D}^{1,2}).$$  Moreover, $\mathcal{M}^N(C^{\Lambda,N}_{\tau,\rho}(\mathbb{R},{\cal D}^{1,2})|_{[0,\tau)}$ is relatively compact in $C([0,\tau), L^2( \Omega,\mathbb{R}^d))$.
\end{lem}
\begin{proof}
{\it \textbf{Step 1}}: In this step we are going to present that $\mathcal{M}^N$ maps $C^{\Lambda,N}_{\tau,\rho}(\mathbb{R},{\cal D}^{1,2})$ into itself.
\begin{enumerate}
\item First we have $\mathcal{M}^N( C^{\Lambda}_{\tau}(\mathbb{R}, L^2(\Omega,\mathbb{R}^d)))\subset  C^{\Lambda}_{\tau}(\mathbb{R}, L^2(\Omega,\mathbb{R}^d))$: the argument here is the same as in {\it\textbf{Step 1}} in the proof of Lemma \ref{Lem1}.
\item  Next to illustrate that for any $t\in [0,\tau)$, $l\in \{1,\cdots, M\}$ and any $Y^N\in C^{\Lambda,N}_{\tau,\rho}(\mathbb{R},{\cal D}^{1,2})$,
\begin{equation*}
{\rm e}^{-2\Lambda |t|}\mathbb{E}\int_{\mathbb{R}}|{\cal D}_{r}^l\mathcal{M}^N(Y^N)(t,\cdot)|^2dr<+\infty.
\end{equation*} 
By the chain rule, (\ref{md1}) and (\ref{md2}), the  Malliavin derivative of ${\cal M}^N(Y^N)(t,\omega)$ is given as:
\begin{eqnarray}\label{343}
{\cal D}_{r}^l{\cal M}^N(Y^N)(t,\omega)&=&\int_{-\infty}^r\chi_{\{r\leq t\}}(r){\cal D}_{r}^l(\Phi^N_{t-\hat{s},\hat{s}}P^-)F(\hat{s}, Y^N(\hat{s},\omega))d\hat{s}\nonumber\\
&&-\int^{+\infty}_r\chi_{\{r\geq t\}}(r){\cal D}_{r}^l(\Phi^N_{t-\hat{s},\hat{s}}P^+)F(\hat{s}, Y^N(\hat{s},\omega))d\hat{s}\nonumber\\
&&+\int_{-\infty}^t\Phi^N_{t-\hat{s},\hat{s}}P^-\nabla F(\hat{s}, Y^N(\hat{s},\omega)){\cal D}_{r}^l  Y^N(\hat{s},\omega)d\hat{s} \nonumber\\
&&-\int^{+\infty}_t\Phi^N_{t-\hat{s},\hat{s}}P^+\nabla F(\hat{s}, Y^N(\hat{s},\omega)){\cal D}_{r}^l  Y^N(\hat{s},\omega)d\hat{s}.
\end{eqnarray}
Then  we get  for any $t\in \mathbb{R}$ the following $L^2$-estimation,
\begin{eqnarray*}
&&{\rm e}^{-2\Lambda |t|}\mathbb{E}\int_{\mathbb{R}}|{\cal D}_{r}^l{\cal M}^N(Y^N)(t,\cdot)|^2dr\\
&=&{\rm e}^{-2\Lambda |t|}\mathbb{E}\int_{-\infty}^{t}|{\cal D}_{r}^l{\cal M}^N(Y^N)(t,\cdot)|^2dr+{\rm e}^{-2\Lambda |t|}\mathbb{E}\int^{+\infty}_{t}|{\cal D}_{r}^l{\cal M}^N(Y^N)(t,\cdot)|^2dr\\
&\leq&3{\rm e}^{-2\Lambda |t|}\mathbb{E}\int_{-\infty}^{t}\Big|\int_{-\infty}^{r}{\cal D}_{r}^l(\Phi^N_{t-\hat{s},\hat{s}}P^-)F(\hat{s}, Y^N(\hat{s},\cdot))d\hat{s}\Big|^2dr\\
&&+3{\rm e}^{-2\Lambda |t|}\mathbb{E}\int_{t}^{+\infty}\Big|\int^{+\infty}_{r}{\cal D}_{r}^l(\Phi^N_{t-\hat{s},\hat{s}}P^+)F(\hat{s}, Y^N(\hat{s},\cdot))d\hat{s}\Big|^2dr\\
&&+3{\rm e}^{-2\Lambda |t|}\mathbb{E}\int_{\mathbb{R}}\Big|\int_{-\infty}^t\Phi^N_{t-\hat{s},\hat{s}}P^-\nabla F(\hat{s}, Y^N(\hat{s},\cdot)){\cal D}_{r}^l  Y^N(\hat{s},\cdot)d\hat{s}\Big|^2dr\\
&&+3{\rm e}^{-2\Lambda |t|}\mathbb{E}\int_{\mathbb{R}}\Big|\int^{+\infty}_t\Phi^N_{t-\hat{s},\hat{s}}P^+\nabla F(\hat{s}, Y^N(\hat{s},\cdot)){\cal D}_{r}^l  Y^N(\hat{s},\cdot)d\hat{s}\Big|^2dr\\
&=&:\sum_{i=1}^4L_i.
\end{eqnarray*}
Applying Lemma \ref{LEMMAMD}, we have that
\begin{eqnarray*}
L_1
&\leq&6\|B\|^2N^2\|F\|^2_{\infty}(1+d^3)^2{\rm e}^{-2\Lambda |t|}\int_{-\infty}^{t}\Big(\int_{-\infty}^{r}{\rm e}^{\frac{1}{2}\mu_{m+1}(t-\hat{s})}{\rm e}^{\Lambda|\hat{s}|}d\hat{s}\Big)^2dr\\
&\leq&12\|B\|^2N^2\|F\|^2_{\infty}(1+d^3)^2\int_{-\infty}^{t}{\rm e}^{(\mu_{m+1}+2\Lambda)(t-r)}\Big(\int_{-\infty}^{r}{\rm e}^{(\frac{1}{2}\mu_{m+1}+\Lambda)(r-\hat{s})}d\hat{s}\Big)^2dr\\
&&+12\|B\|^2N^2\|F\|^2_{\infty}(1+d^3)^2\int_{-\infty}^{t}{\rm e}^{(\mu_{m+1}-2\Lambda)(t-r)}\Big(\int_{-\infty}^{r}{\rm e}^{(\frac{1}{2}\mu_{m+1}-\Lambda)(r-\hat{s})}d\hat{s}\Big)^2dr\\
&\leq &48\|B\|^2N^2\|F\|^2_{\infty}(1+d^3)^2\Big\{\dfrac{1}{|\mu_{m+1}+2\Lambda|^3}+\dfrac{1}{|\mu_{m+1}-2\Lambda|^3}\Big\}<\infty.
\end{eqnarray*}
Similarly,
\begin{eqnarray*}
L_2
&\leq&48\|B\|^2N^2\|F\|^2_{\infty}(1+d^3)^2\Big(\dfrac{1}{|\mu_{m}+2\Lambda|^3}+\dfrac{1}{|\mu_{m}-2\Lambda|^3}\Big)<\infty.
\end{eqnarray*}
As for terms $L_3$ and $L_4$, we have 
\begin{eqnarray*}
L_3
&\leq &3N^2\|\nabla F\|^2_{\infty}{\rm e}^{-2\Lambda |t|}\mathbb{E}\int_{\mathbb{R}}\Big(\int_{-\infty}^t{\rm e}^{\Lambda|\hat{s}|}{\rm e}^{\frac{1}{2}\mu_{m+1}(t-\hat{s})}|{\cal D}_{r}^l Y^N(\hat{s},\cdot)|d\hat{s}\Big)^2dr\\
&\leq &12N^2\|\nabla F\|^2_{\infty}\left(\frac{1}{|\mu_{m+1}-4\Lambda|}+\frac{1}{|\mu_{m+1}+4\Lambda|}\right)\int_{-\infty}^t{\rm e}^{\frac{1}{2}\mu_{m+1}(t-\hat{s})}\mathbb{E}\int_{\mathbb{R}}|{\cal D}_{r}^l Y^N(\hat{s},\cdot)|^2drd\hat{s}
\end{eqnarray*}
Note that $\mathbb{E}\int_{\mathbb{R}}|{\cal D}_{r}^l Y^N(\hat{s},\cdot)|^2dr$ is nonegative and periodic with period $\tau$, i.e.,
\begin{eqnarray*}
\mathbb{E}\int_{\mathbb{R}}|{\cal D}_{r}^l Y^N(\hat{s}+\tau,\cdot)|^2dr=\mathbb{E}\int_{\mathbb{R}}|{\cal D}_{r}^l Y^N(\hat{s},\theta_{\tau}\cdot)|^2dr=\mathbb{E}\int_{\mathbb{R}}|{\cal D}_{r}^l Y^N(\hat{s},\cdot)|^2dr,
\end{eqnarray*}
where the right equality is true according to Lemma \ref{Preserving}. Then we have
\begin{eqnarray*}
\int_{-\infty}^t{\rm e}^{\frac{1}{2}\mu_{m+1}(t-\hat{s})}\mathbb{E}\int_{\mathbb{R}}|{\cal D}_{r}^l Y^N(\hat{s},\cdot)|^2drd\hat{s}
&\leq & \sup\limits_{\substack{s\in [0,\tau)}}\mathbb{E}\int_{\mathbb{R}}|{\cal D}_{r}^lY^N(s,\cdot)|^2dr \int_{-\infty}^t{\rm e}^{\frac{1}{2}\mu_{m+1}(t-\hat{s})}d\hat{s}\\
&\leq & \frac{2{\rm e}^{2\Lambda \tau}}{|\mu_{m+1}|}\sup\limits_{\substack{s\in [0,\tau)}}\  {\rm e}^{-2\Lambda|s|}\mathbb{E}\int_{\mathbb{R}}|{\cal D}_{r}^lY^N(s,\cdot)|^2dr.
\end{eqnarray*}
Thus,
\begin{eqnarray*}
L_3&\leq & \Big(\frac{24N^2\|\nabla F\|^2_{\infty}{\rm e}^{2\Lambda \tau}}{|\mu_{m+1}(\mu_{m+1}-4\Lambda)|}+\frac{24N^2\|\nabla F\|^2_{\infty}{\rm e}^{2\Lambda \tau}}{|\mu_{m+1}(\mu_{m+1}+4\Lambda)|}\Big)\sup\limits_{\substack{s\in [0,\tau)}}\  {\rm e}^{-2\Lambda|s|}\mathbb{E}\int_{\mathbb{R}}|{\cal D}_{r}^lY^N(s,\cdot)|^2dr <\infty.
\end{eqnarray*}
Similarly,
\begin{eqnarray*}
L_4
&\leq&  \Big(\frac{24N^2\|\nabla F\|^2_{\infty}{\rm e}^{2\Lambda \tau}}{\mu_m(\mu_{m}-4\Lambda)}+\frac{24N^2\|\nabla F\|^2_{\infty}{\rm e}^{2\Lambda \tau}}{\mu_m(\mu_{m}+4\Lambda)}\Big)\sup\limits_{\substack{s\in [0,\tau)}}\  {\rm e}^{-2\Lambda|s|}\mathbb{E}\int_{\mathbb{R}}|{\cal D}_{r}^lY^N(s,\cdot)|^2dr<\infty.
\end{eqnarray*}

\item It remains to show that for any $l\in \{1,\cdots,M\}$ and $\delta\in \mathbb{R}$,
$$\sup\limits_{\substack{t\in [0,\tau)}}\dfrac{{\rm e}^{-2\Lambda |t|}}{|\delta|}\int_{\mathbb{R}}\mathbb{E}|{\cal D}_{r+\delta}^l\mathcal{M}^N(Y^N)(t,\cdot)-{\cal D}_{r}^l\mathcal{M}^N(Y^N)(t,\cdot)|^2dr< \infty.$$

In fact the left hand side of the  above can be separated into three integrals, 
\begin{eqnarray}\label{NJ}
&&\sup\limits_{\substack{t\in [0,\tau),\delta\in \mathbb{R}}}\dfrac{{\rm e}^{-2\Lambda |t|}}{|\delta|} \int_{\mathbb{R}}\mathbb{E}|{\cal D}_{r+\delta}^l{\cal M}^N(Y^N)(t,\cdot)-{\cal D}_{r}^l{\cal M}^N(Y^N)(t,\cdot)|^2dr\nonumber\\
=&& \sup\limits_{\substack{t\in [0,\tau),\delta\in \mathbb{R}}}\dfrac{{\rm e}^{-2\Lambda |t|}}{|\delta|} \mathbb{E} \int_{-\infty}^{t-\delta}|{\cal D}_{r+\delta}^l{\cal M}^N(Y^N)(t,\cdot)-{\cal D}_{r}^l{\cal M}^N(Y^N)(t,\cdot)|^2dr\nonumber\\
&&+\sup\limits_{\substack{t\in [0,\tau),\delta\in \mathbb{R}}}\dfrac{{\rm e}^{-2\Lambda |t|}}{|\delta|} \mathbb{E} \int_{t-\delta}^{t}|{\cal D}_{r+\delta}^l{\cal M}^N(Y^N)(t,\cdot)-{\cal D}_{r}^l{\cal M}^N(Y^N)(t,\cdot)|^2dr\nonumber\\
&&+\sup\limits_{\substack{t\in [0,\tau),\delta\in \mathbb{R}}}\dfrac{{\rm e}^{-2\Lambda |t|}}{|\delta|} \mathbb{E}\int_{t}^{+\infty}|{\cal D}_{r+\delta}^l{\cal M}^N(Y^N)(t,\cdot)-{\cal D}_{r}^l{\cal M}^N(Y^N)(t,\cdot)|^2dr\nonumber\\
:=&&\hat{K}_1+\hat{K}_2+\hat{K}_3.
\end{eqnarray}
To estimate $\hat{K}_1$ in (\ref{NJ}),  note when $r\leq t-\delta$, by (\ref{343}) we have 
\begin{eqnarray*}
\mathcal{D}^l_r {\cal M}^N(Y^N) (t,\omega)&=&\int^{r}_{-\infty}\mathcal{D}^l_r(\Phi^N_{t-\hat{s},\hat{s}}P^-) F(\hat{s},Y^N (\hat{s},\omega))d\hat{s}\\
&&+\int_{-\infty}^t\Phi^N_{t-\hat{s},\hat{s}}P^-\nabla F(\hat{s}, Y^N(\hat{s},\omega)){\cal D}_{r}^l Y^N(\hat{s},\omega)d\hat{s}\\
&&-\int^{+\infty}_t\Phi^N_{t-\hat{s},\hat{s}}P^+\nabla F(\hat{s}, Y^N(\hat{s},\omega)){\cal D}_{r}^l Y^N(\hat{s},\omega)d\hat{s},
\end{eqnarray*}
and
\begin{eqnarray*}
\mathcal{D}^l_{r+\delta} {\cal M}^N(Y^N) (t,\omega)&=&\int^{r+\delta}_{-\infty}\mathcal{D}^l_{r+\delta}(\Phi^N_{t-\hat{s},\hat{s}}P^-) F(\hat{s},Y^N (\hat{s},\omega))d\hat{s}\\
&&+\int_{-\infty}^t\Phi^N_{t-\hat{s},\hat{s}}P^-\nabla F(\hat{s}, Y^N(\hat{s},\omega)){\cal D}_{r+\delta}^l Y^N(\hat{s},\omega)d\hat{s}\\
&&-\int^{+\infty}_t\Phi^N_{t-\hat{s},\hat{s}}P^+\nabla F(\hat{s}, Y^N(\hat{s},\omega)){\cal D}_{r+\delta}^l Y^N(\hat{s},\omega)d\hat{s}.
\end{eqnarray*}
Thus
\begin{eqnarray*}
\hat{K}_1
&\leq & \sup_{\substack{t\in [0,\tau),\delta\in \mathbb{R}}}\dfrac{3{\rm e}^{-2\Lambda |t|}}{|\delta|} \mathbb{E}  \int_{-\infty}^{t-\delta}\bigg\{\Big|\int_{-\infty}^t \Phi_{t-\hat{s},\hat{s}}^NP^-\nabla F(\hat{s},Y^N(\hat{s},\cdot))({\cal D}_{r+\delta}^l-{\cal D}_{r}^l)(Y^N)(t,\cdot)d\hat{s}\Big|^2\\
&&+\Big|\int_{-\infty}^{r+\delta} {\cal D}_{r+\delta}^l\Phi_{t-\hat{s},\hat{s}}^NP^-F(\hat{s},Y^N(\hat{s},\cdot))d\hat{s}-\int_{-\infty}^{r} {\cal D}_{r}^l\Phi_{t-\hat{s},\hat{s}}^NP^-F(\hat{s},Y^N(\hat{s},\cdot))d\hat{s}\Big|^2\\
&&+\Big|\int^{+\infty}_t \Phi_{t-\hat{s},\hat{s}}^NP^+\nabla F(\hat{s},Y^N(\hat{s},\cdot))({\cal D}_{r+\delta}^l-{\cal D}_{r}^l)(Y^N)(t,\cdot)d\hat{s}\Big|^2\bigg\}dr\\
&:=&\sup_{\substack{t\in [0,\tau),\delta\in \mathbb{R}}}\sum_{i=1}^{3}Q_i.
\end{eqnarray*}
First note that  $Q_1$ is bounded via measure preserving result in Lemma \ref{Preserving},
\begin{eqnarray*}
&&Q_1\\
&\leq &\dfrac{6N^2\|\nabla F\|^2_{\infty}}{|\delta|}{\rm e}^{-2\Lambda |t|}\mathbb{E}\int_{\mathbb{R}}\Big(\int_{-\infty}^t {\rm e}^{\Lambda \hat{s}}{\rm e}^{\frac{1}{2}\mu_{m+1}(t-\hat{s})} |({\cal D}_{r+\delta}^l-{\cal D}_{r}^l)(Y^N)(\hat{s},\cdot)|d\hat{s}\Big)^2dr\\
&&+\dfrac{6N^2\|\nabla F\|^2_{\infty}}{|\delta|}{\rm e}^{-2\Lambda |t|}\mathbb{E}\int_{\mathbb{R}}\Big(\int_{-\infty}^t {\rm e}^{-\Lambda \hat{s}}{\rm e}^{\frac{1}{2}\mu_{m+1}(t-\hat{s})} |({\cal D}_{r+\delta}^l-{\cal D}_{r}^l)(Y^N)(\hat{s},\cdot)|d\hat{s}\Big)^2dr\\
&\leq &\Big(\frac{12N^2\|\nabla F\|^2_{\infty}}{|\delta||\mu_{m+1}+4\Lambda|}+\frac{12N^2\|\nabla F\|^2_{\infty}}{|\delta||\mu_{m+1}-4\Lambda|}\Big)\mathbb{E}\int_{\mathbb{R}}\int_{-\infty}^t {\rm e}^{\frac{1}{2}\mu_{m+1}(t-\hat{s})} |({\cal D}_{r+\delta}^l-{\cal D}_{r}^l)(Y^N)(\hat{s},\cdot)|^2d\hat{s}dr
\end{eqnarray*}
Note that $\mathbb{E}\int_{\mathbb{R}}|({\cal D}_{r+\delta}^l-{\cal D}_{r}^l)(Y^N)(\hat{s},\cdot)|^2dr$ is a nonegative periodic function in $C^{\Lambda}(\mathbb{R})$ with period $\tau$. Thus we have
\begin{eqnarray*}
&&\int_{-\infty}^t{\rm e}^{\frac{1}{2}\mu_{m+1}(t-\hat{s})}\mathbb{E}\int_{\mathbb{R}}|({\cal D}_{r+\delta}^l-{\cal D}_{r}^l)(Y^N)(\hat{s},\cdot)|^2drd\hat{s}\\
&\leq &  \sup_{s\in[0,\tau)}\mathbb{E}\int_{\mathbb{R}}|({\cal D}_{r+\delta}^l-{\cal D}_{r}^l)(Y^N)(\hat{s},\cdot)|^2dr\int_{-\infty}^t{\rm e}^{\frac{1}{2}\mu_{m+1}(t-\hat{s})}d\hat{s}\\
&\leq &\frac{2{\rm e}^{2\Lambda\tau}}{|\mu_{m+1}|} \sup_{s\in[0,\tau)}{\rm e}^{-2\Lambda|s|}\mathbb{E}\int_{\mathbb{R}}|({\cal D}_{r+\delta}^l-{\cal D}_{r}^l)(Y^N)(\hat{s},\cdot)|^2dr.
\end{eqnarray*}
This leads to
\begin{eqnarray*}
Q_1 &\leq& {\rm e}^{2\Lambda \tau}\frac{24N^2\|\nabla F\|^2_{\infty}}{|\mu_{m+1}|}\Big(\frac{1}{|\mu_{m+1}+4\Lambda|}+\frac{1}{|\mu_{m+1}-4\Lambda|}\Big) \\
&&\cdot\sup_{\hat{s}\in [0,\tau),\delta\in \mathbb{R}}\dfrac{{\rm e}^{-2\Lambda|\hat{s}|}}{|\delta|}\mathbb{E}\int_{\mathbb{R}} |({\cal D}_{r+\delta}^l-{\cal D}_{r}^l)(Y^N)(\hat{s},\cdot)|^2dr<\infty.
\end{eqnarray*}
Analogously, 
\begin{eqnarray*}
Q_{3}
&\leq &  24{\rm e}^{2\Lambda \tau}N^2\|\nabla F\|^2_{\infty}\Big\{\frac{1}{\mu_m(\mu_{m}+4\Lambda)}+\frac{1}{\mu_m(\mu_{m}-4\Lambda)}\Big\}\\
&&\cdot\sup_{\hat{s}\in [0,\tau),\delta\in \mathbb{R}}\dfrac{{\rm e}^{-2\Lambda|\hat{s}|}}{|\delta|}\mathbb{E}\int_{\mathbb{R}} |({\cal D}_{r+\delta}^l-{\cal D}_{r}^l)(Y^N)(\hat{s},\cdot)|^2dr<\infty.
\end{eqnarray*}
Secondly, $Q_2$ can be estimated using (\ref{MDesti1}),
\begin{eqnarray*}
Q_2
&\leq &\dfrac{3{\rm e}^{-2\Lambda |t|}\|F\|^2_{\infty}}{|\delta|}  \int_{-\infty}^{t-\delta}\mathbb{E}\Big(\int_{r}^{r+\delta} \|{\cal D}_{r+\delta}^l\Phi_{t-\hat{s},\hat{s}}^NP^-\|d\hat{s}\Big)^2dr\\
&\leq &3N^2{\rm e}^{-2\Lambda |t|}\|F\|^2_{\infty}\|B_l\|(1+2d^3)^2 \dfrac{1}{|\delta|} \int_{-\infty}^{t-\delta}\Big(\int_{r}^{r+\delta}{\rm e}^{\frac{1}{2}\mu_{m+1}(t-\hat{s})}{\rm e}^{\Lambda|\hat{s}|}d\hat{s}\Big)^2dr \\ 
&\leq &6N^2\|F\|^2_{\infty}\|B_l\|(1+2d^3)^2 \int_{-\infty}^{t-\delta}{\rm e}^{(\mu_{m+1}-2\Lambda)(t-\delta-r)}\int_{r}^{r+\delta}{\rm e}^{(\frac{1}{2}\mu_{m+1}-\Lambda)(r+\delta-\hat{s})}d\hat{s}dr \\
&&+6N^2\|F\|^2_{\infty}\|B_l\|(1+2d^3)^2 \int_{-\infty}^{t-\delta}{\rm e}^{(\mu_{m+1}+2\Lambda)(t-\delta-r)}\int_{r}^{r+\delta}{\rm e}^{(\frac{1}{2}\mu_{m+1}+\Lambda)(r+\delta-\hat{s})}d\hat{s}dr \\ 
&\leq &12N^2\|F\|^2_{\infty}\|B_l\|(1+2d^3)^2 \Big\{\frac{1}{(\mu_{m+1}+2\Lambda)^2}+\frac{1}{(\mu_{m+1}-2\Lambda)^2}\Big\} 
<\infty.
\end{eqnarray*}
Thus $\hat{K}_1< \infty$.
To consider $\hat{K}_2$ in (\ref{NJ}), note that when $r \leq t\leq r+\delta$, the expressions (\ref{343}) gives us 
\begin{eqnarray*}
\mathcal{D}^l_r {\cal M}^N(Y^N) (t,\omega)&=&\int^{r}_{-\infty}\mathcal{D}^l_r(\Phi^N_{t-\hat{s},\hat{s}}P^-) F(\hat{s},Y^N (\hat{s},\omega))d\hat{s}\\
&&+\int_{-\infty}^t\Phi^N_{t-\hat{s},\hat{s}}P^-\nabla F(\hat{s}, Y^N(\hat{s},\omega)){\cal D}_{r}^l Y^N(\hat{s},\omega)d\hat{s}\\
&&-\int^{+\infty}_t\Phi^N_{t-\hat{s},\hat{s}}P^+\nabla F(\hat{s}, Y^N(\hat{s},\omega)){\cal D}_{r}^l Y^N(\hat{s},\omega)d\hat{s},
\end{eqnarray*}
and
\begin{eqnarray*}
\mathcal{D}^l_{r+\delta} {\cal M}^N(Y^N) (t,\omega)&=&\int_{r+\delta}^{+\infty}\mathcal{D}^l_{r+\delta}(\Phi^N_{t-\hat{s},\hat{s}}P^+) F(\hat{s},Y^N (\hat{s},\omega))d\hat{s}\\
&&+\int_{-\infty}^t\Phi^N_{t-\hat{s},\hat{s}}P^-\nabla F(\hat{s}, Y^N(\hat{s},\omega)){\cal D}_{r+\delta}^l Y^N(\hat{s},\omega)d\hat{s}\\
&&-\int^{+\infty}_t\Phi^N_{t-\hat{s},\hat{s}}P^+\nabla F(\hat{s}, Y^N(\hat{s},\omega)){\cal D}_{r+\delta}^l Y^N(\hat{s},\omega)d\hat{s}.
\end{eqnarray*}
Thus
\begin{eqnarray*}
\hat{K}_2&=&\sup\limits_{\substack{t\in [0,\tau),\delta\in \mathbb{R}}}\dfrac{{\rm e}^{-2\Lambda |t|}}{|\delta|}  \mathbb{E}\int_{t-\delta}^{t}|{\cal D}_{r+\delta}^l{\cal M}^N(Y^N)(t,\cdot)-{\cal D}_{r}^l{\cal M}^N(Y^N)(t,\cdot)|^2dr\\
&\leq &\sup\limits_{\substack{t\in [0,\tau),\delta\in \mathbb{R}}}\dfrac{4{\rm e}^{-2\Lambda |t|}}{|\delta|}  \mathbb{E} \int_{t-\delta}^{t}
 \Big|\int_{-\infty}^{r} {\cal D}_{r}^l\Phi_{t-\hat{s},\hat{s}}^NP^-F(\hat{s},Y^N(\hat{s},\cdot))d\hat{s}\Big|^2dr\\
 &&+\sup\limits_{\substack{t\in [0,\tau),\delta\in \mathbb{R}}}\dfrac{4{\rm e}^{-2\Lambda |t|}}{|\delta|}  \mathbb{E}\int_{t-\delta}^{t}\Big|\int^{+\infty}_{r+\delta} {\cal D}_{r+\delta}^l\Phi_{t-\hat{s},\hat{s}}^NP^+F(\hat{s},Y^N(\hat{s},\cdot))d\hat{s}\Big|^2dr\\
&&+\sup\limits_{\substack{t\in [0,\tau),\delta\in \mathbb{R}}}\dfrac{4{\rm e}^{-2\Lambda |t|}}{|\delta|}  \mathbb{E}\int_{t-\delta}^{t}\Big|\int_{-\infty}^t \Phi_{t-\hat{s},\hat{s}}^NP^-\nabla F(\hat{s},Y^N(\hat{s},\cdot))({\cal D}_{r+\delta}^l-{\cal D}_{r}^l)(Y^N)(\hat{s},\cdot)d\hat{s}\Big|^2dr\\
&&+\sup\limits_{\substack{t\in [0,\tau),\delta\in \mathbb{R}}}\dfrac{4{\rm e}^{-2\Lambda |t|}}{|\delta|}  \mathbb{E}\int_{t-\delta}^{t}\Big|\int^{+\infty}_t \Phi_{t-\hat{s},\hat{s}}^NP^+\nabla F(\hat{s},Y^N(\hat{s},\cdot))({\cal D}_{r+\delta}^l-{\cal D}_{r}^l)(Y^N)(\hat{s},\cdot)d\hat{s}\Big|^2dr\\
&:=&\sup_{\substack{t\in [0,\tau),\delta\in \mathbb{R}}}\sum_{i=4}^{7}Q_i.
\end{eqnarray*}
But
\begin{eqnarray*}
Q_{4}
&\leq & \dfrac{4\|F\|^2_{\infty}}{|\delta|} {\rm e}^{-2\Lambda |t|} \int_{t-\delta}^{t}\mathbb{E}\Big(\int_{-\infty}^{r} \|{\cal D}_{r}^l\Phi_{t-\hat{s},\hat{s}}^NP^-\|d\hat{s}\Big)^2dr\\
&\leq & 32\|F\|^2_{\infty}\|B_l\|^2(1+2d^3)\Big( \dfrac{1}{|\mu_{m+1}-2\Lambda|^2}+ \dfrac{1}{|\mu_{m+1}+2\Lambda|^2}\Big)<\infty.
\end{eqnarray*}
Similarly,
\begin{eqnarray*}
Q_{5}
&\leq & 32\|F\|^2_{\infty}\|B_l\|^2(1+2d^3)\Big( \dfrac{1}{|\mu_{m}-2\Lambda|^2}+ \dfrac{1}{|\mu_{m}+2\Lambda|^2}\Big)
<\infty.
\end{eqnarray*}
Besides, we have by similar calculations as in $Q_1$ and $Q_2$,
\begin{eqnarray*}
Q_6
&\leq & {\rm e}^{2\Lambda \tau}\frac{32N^2\|\nabla F\|^2_{\infty}}{|\mu_{m+1}|}\Big(\frac{1}{|\mu_{m+1}+4\Lambda|}+\frac{1}{|\mu_{m+1}-4\Lambda|}\Big) \\
&&\cdot\sup_{\hat{s}\in [0,\tau),\delta\in \mathbb{R}}\dfrac{{\rm e}^{-2\Lambda|\hat{s}|}}{|\delta|}\mathbb{E}\int_{\mathbb{R}} |({\cal D}_{r+\delta}^l-{\cal D}_{r}^l)(Y^N)(\hat{s},\cdot)|^2dr<\infty,
\end{eqnarray*}
and 
\begin{eqnarray*}
Q_{7}&:=&\dfrac{4{\rm e}^{-2\Lambda |t|}}{|\delta|}  \mathbb{E}\int_{t-\delta}^{t}\Big|\int^{+\infty}_t \Phi_{t-\hat{s},\hat{s}}^NP^+\nabla F(\hat{s},Y^N(\hat{s},\cdot))({\cal D}_{r+\delta}^l-{\cal D}_{r}^l)(Y^N)(\hat{s},\cdot)d\hat{s}\Big|^2dr\\
&\leq & 32{\rm e}^{2\Lambda \tau}N^2\|\nabla F\|^2_{\infty}\Big(\frac{1}{|\mu_m(\mu_{m}+4\Lambda)|}+\frac{1}{|\mu_m(\mu_{m}-4\Lambda)|}\Big) \\
&&\cdot\sup_{\hat{s}\in [0,\tau),\delta\in \mathbb{R}}\dfrac{{\rm e}^{-2\Lambda|\hat{s}|}}{|\delta|}\mathbb{E}\int_{\mathbb{R}} |({\cal D}_{r+\delta}^l-{\cal D}_{r}^l)(Y^N)(\hat{s},\cdot)|^2dr\\
&<&\infty.
\end{eqnarray*}
Now we have shown that $\hat{K}_2<\infty$.
To consider $\hat{K}_3$, note that when $r \geq t$, (\ref{343}) gives
\begin{eqnarray*}
\mathcal{D}^l_r {\cal M}^N(Y^N) (t,\omega)&=&\int_{r}^{+\infty}\mathcal{D}^l_{r}(\Phi^N_{t-\hat{s},\hat{s}}P^+) F(\hat{s},Y^N (\hat{s},\omega))d\hat{s}\\
&&+\int_{-\infty}^t\Phi^N_{t-\hat{s},\hat{s}}P^-\nabla F(\hat{s}, Y^N(\hat{s},\omega)){\cal D}_{r}^l Y^N(\hat{s},\omega)d\hat{s}\\
&&-\int^{+\infty}_t\Phi^N_{t-\hat{s},\hat{s}}P^+\nabla F(\hat{s}, Y^N(\hat{s},\omega)){\cal D}_{r}^l Y^N(\hat{s},\omega)d\hat{s},
\end{eqnarray*}
and
\begin{eqnarray*}
\mathcal{D}^l_{r+\delta} {\cal M}^N(Y^N) (t,\omega)&=&\int_{r+\delta}^{+\infty}\mathcal{D}^l_{r+\delta}(\Phi^N_{t-\hat{s},\hat{s}}P^+) F(\hat{s},Y^N (\hat{s},\omega))d\hat{s}\\
&&+\int_{-\infty}^t\Phi^N_{t-\hat{s},\hat{s}}P^-\nabla F(\hat{s}, Y^N(\hat{s},\omega)){\cal D}_{r+\delta}^l Y^N(\hat{s},\omega)d\hat{s}\\
&&-\int^{+\infty}_t\Phi^N_{t-\hat{s},\hat{s}}P^+\nabla F(\hat{s}, Y^N(\hat{s},\omega)){\cal D}_{r+\delta}^l Y^N(\hat{s},\omega)d\hat{s}.
\end{eqnarray*}
Then
\begin{eqnarray*}
\hat{K}_3&=&\sup\limits_{\substack{t\in [0,\tau),\delta\in \mathbb{R}}}\dfrac{{\rm e}^{-2\Lambda |t|}}{|\delta|}\mathbb{E} \int_{t}^{+\infty}|{\cal D}_{r+\delta}^l{\cal M}^N(Y^N)(t,\cdot)-{\cal D}_{r}^l{\cal M}^N(Y^N)(t,\cdot)|^2dr\\
&\leq &\sup\limits_{\substack{t\in [0,\tau),\delta\in \mathbb{R}}}\dfrac{3{\rm e}^{-2\Lambda |t|}}{|\delta|}\mathbb{E} \int_{t}^{+\infty}\bigg\{\Big|\int_{-\infty}^t \Phi_{t-\hat{s},\hat{s}}^NP^-\nabla F(\hat{s},Y^N(\hat{s},\cdot))({\cal D}_{r+\delta}^l-{\cal D}_{r}^l)(Y^N)(\hat{s},\cdot)d\hat{s}\Big|^2\\
&&+\Big|\int^{+\infty}_{r+\delta_r} {\cal D}_{r+\delta}^l\Phi_{t-\hat{s},\hat{s}}^NP^+F(\hat{s},Y^N(\hat{s},\cdot))d\hat{s}-\int^{+\infty}_{r} {\cal D}_{r}^l\Phi_{t-\hat{s},\hat{s}}^NP^+F(\hat{s},Y^N(\hat{s},\cdot))d\hat{s}\Big|^2\\
&&+\Big|\int^{+\infty}_t \Phi_{t-\hat{s},\hat{s}}^NP^+\nabla F(\hat{s},Y^N(\hat{s},\cdot))({\cal D}_{r+\delta}^l-{\cal D}_{r}^l)(Y^N)(\hat{s},\cdot)d\hat{s}\Big|^2\bigg\}dr\\
&=&\sup_{\substack{t\in \mathbb{R},\delta\in \mathbb{R}}}\sum_{i=8}^{10}Q_i.
\end{eqnarray*}
Now it is easy to see that,
\begin{eqnarray*}
Q_8
&\leq & {\rm e}^{2\Lambda \tau}\frac{24N^2\|\nabla F\|^2_{\infty}}{|\mu_{m+1}|}\Big(\frac{1}{|\mu_{m+1}+4\Lambda|}+\frac{1}{|\mu_{m+1}-4\Lambda|}\Big)\\
&&\cdot\sup_{\hat{s}\in [0,\tau),\delta\in \mathbb{R}}\dfrac{{\rm e}^{-2\Lambda|\hat{s}|}}{|\delta|}\mathbb{E}\int_{\mathbb{R}} |({\cal D}_{r+\delta}^l-{\cal D}_{r}^l)(Y^N)(\hat{s},\cdot)|^2dr
<\infty,
\end{eqnarray*}
and
\begin{eqnarray*}
Q_{10}
&\leq &24{\rm e}^{2\Lambda \tau}N^2\|\nabla F\|^2_{\infty}\Big(\frac{1}{|\mu_m(\mu_{m}+4\Lambda)|}+\frac{1}{|\mu_m(\mu_{m}-4\Lambda)|}\Big) \\
&&\cdot\sup_{\hat{s}\in [0,\tau),\delta\in \mathbb{R}}\dfrac{{\rm e}^{-2\Lambda|\hat{s}|}}{|\delta|}\mathbb{E}\int_{\mathbb{R}} |({\cal D}_{r+\delta}^l-{\cal D}_{r}^l)(Y^N)(\hat{s},\cdot)|^2dr
<\infty.
\end{eqnarray*}
Similar to $Q_2$,
\begin{eqnarray*}
Q_{9}
\leq 24N^2\|F\|^2_{\infty}\|B_l\|(1+2d^3)^2 \Big\{\frac{1}{(\mu_{m}+2\Lambda)^2}+\frac{1}{(\mu_{m}-2\Lambda)^2}\Big\} <\infty.
\end{eqnarray*}
In summary, we have shown that
$$\sup\limits_{\substack{t\in [0,\tau),\delta\in \mathbb{R}}}\dfrac{{\rm e}^{-2\Lambda |t|}}{|\delta|}\int_{\mathbb{R}}\mathbb{E}|{\cal D}_{r+\delta}^l\mathcal{M}^N(Y^N)(t,\cdot)-{\cal D}_{r}^l\mathcal{M}^N(Y^N)(t,\cdot)|^2dr< \infty.$$
 \end{enumerate}
Thus we could conclude that
$\mathcal{M}^N$ maps $C^{\Lambda}_{\tau}(\mathbb{R},{\cal D}^{1,2})$ into itself.
\vskip3pt

{\it \textbf{Step 2}}: Now we can prove that for each $N\in \mathbb{N}$, $\mathcal{M}^N(C^{\Lambda}_{\tau}(\mathbb{R},{\cal D}^{1,2}))|_{[0,\tau)}$ is relatively compact in $C([0,\tau), L^2(\Omega,\mathbb{R}^d))$.
Applying Theorem \ref{B-S1} and what we have proved in {\it \textbf{Step 1}},  we conclude that that for any sequence $\{{\cal M}^N(f_{n})\}_{n\in \mathbb{N}}\in
C^{\Lambda}_{\tau}(\mathbb{R},{\cal D}^{1,2})|_{[0,\tau)}$, there exists a
subsequence, still denoted by $\{{\cal M}^N(f_{n})\}_{n\in \mathbb{N}}$ and $V^{N}\in
C([0,\tau), L^2(\Omega,\mathbb{R}^d))$ such that
\begin{eqnarray}\label{345}
\sup\limits_{t\in [0,\tau)}\mathbb{E}|{\cal M}^N(f_{n})(t,\cdot)-V^{N}(t,\cdot)|^2\to 0
\end{eqnarray}
as $n\to \infty$.  
\end{proof}

\begin{rmk} Note that in Theorem \ref{B-S1}, the relative compactness criterion
allows us to apply $L^2(\Omega,\mathbb{R}^d))$-valued functions only on a finite time interval. But we can push it to the whole 
real line by the random periodicity.
\end{rmk}

\begin{proof} [Proof of Proposition \ref{Main}]
 We prove for any fixed $N$, ${\cal M}^N(C^{\Lambda}_{\tau}(\mathbb{R},{\cal D}^{1,2}))$ is relatively compact in $C^{\Lambda}_\tau(\mathbb{R}, L^2(\Omega,\mathbb{R}^d))$.
Due to the relative compactness in $C([0,\tau), L^2(\Omega,\mathbb{R}^d))$, we are able to find a subsequence, denoted by $\{{\cal M}^N(Y^N_{n_j})\}_{j\in \mathbb{N}}$, from an arbitrary sequence $\{{\cal M}^N(Y^N_{n})\}_{n\in \mathbb{N}}$ such that it will converge to the accumulation point $V^N$, in the norm shown in Eqn. (\ref{345}). Now define for any $t\in [m\tau, m\tau+\tau)$, 
$$V^{N}(t,\omega)=V^{N}(t-m\tau,\theta_{m\tau}\omega).$$ 
By the construction, we can see as $t+\tau\in [(m+1)\tau, (m+2)\tau)$, so
$$V^{N}(t+\tau,\omega)=V^{N}(t+\tau-(m+1)\tau,\theta_{(m+1)\tau}\omega)=V^{N}(t-m\tau,\theta_{m\tau}\theta_\tau\omega)=V^{N}(t,\theta_\tau\omega).$$
Note that
\begin{eqnarray*}
{\cal M}^N(Y^{N}_{n_j})(t,\theta_{m\tau}\omega)={\cal M}^N(Y^{N}_{n_j})(t+m\tau,\omega).
\end{eqnarray*}
With (\ref{345}), the periodic property of ${\cal M}^N(Y^{N}_{n_j})$, and the probability preserving of $\theta_{m\tau}$, we obtain
\begin{eqnarray*}
&&\sup\limits_{t\in [m\tau, m\tau+\tau)}{\rm e}^{-2\Lambda|t|}\mathbb{E}|{\cal M}^N(Y^{N}_{n_j})(t,\cdot)-V^{N}(t,\cdot)|^2\\
&\leq &\sup\limits_{t\in [0,\tau)}\mathbb{E}|{\cal M}^N(Y^{N}_{n_j})(t+m\tau ,\cdot)-V^{N}(t+m\tau,\cdot)|^2\\
&=&\sup\limits_{t\in [0,\tau)}\mathbb{E}|{\cal M}^N(Y^{N}_{n_j})(t,\theta
_{m\tau}\cdot)-V^{N}(t,\theta_{m\tau}\cdot)|^2\\
&=&\sup\limits_{t\in [0,\tau)}\mathbb{E}|{\cal M}^N(Y^{N}_{n_j})(t,\cdot)-V^{N}(t,\cdot)|^2
\to  0,
\end{eqnarray*}
 Thus
\begin{eqnarray*}
\sup\limits_{t\in \mathbb{R}}{\rm e}^{-2\Lambda|t|}E|{\cal M}^N(Y^{N}_{n_j})(t,
\cdot)-V^{N}(t,\cdot)|^2\to 0,
\end{eqnarray*}
as $j\to \infty$.
Therefore ${\cal M}^N( C^{\Lambda}_{\tau}(\mathbb{R},{\cal D}^{1,2}))$ is relatively compact in $C^{\Lambda}_{\tau}(\mathbb{R}, L^2(\Omega,\mathbb{R}^d))$.
\end{proof}
\begin{thm}\label{Main1} Under the same conditions of Proposition \ref{Main}, 
there exists a ${\cal B}(\mathbb{R})\otimes\mathcal{F}$-measurable map $Y: \mathbb{R}\times\Omega\rightarrow \mathbb{R}^d$
satisfying Eqn. (\ref{VT1}) and $Y(t+\tau, \omega)=Y(t, \theta_{\tau}\omega)$ for
any $t\in \mathbb{R}$ and $\omega\in \Omega$.
\end{thm}
\begin{proof} 
 According to 
Schauder's fixed point theorem, ${\cal M}^N$ has a fixed point in
$C^{\Lambda}_{\tau}(\mathbb{R}, L^2(\Omega,\mathbb{R}^d))$. That is to say
there exists a solution $Y^N\in
C^{\Lambda}_{\tau}(\mathbb{R}, L^2(\Omega,\mathbb{R}^d))$ of equation
(\ref{VT1}) such that for any $t\in \mathbb{R}$,
$Y^{N}(t+\tau,\omega)=Y^{N}(t,\theta_{\tau}\omega)$. Moreover, $Y^{N}(t+\tau,\omega)=Y^{N}(t,\theta_{\tau}\omega).$ 

Recall $\Omega_N$ as defined in (\ref{omegan}). As the random variable 
$$\max\left\{\sup_{t\geq 0}\|\Phi(t,\theta_s\omega)P^{-}\|{\rm e}^{-\frac{1}{2}\mu|t|},\ \sup_{t\leq 0}\|\Phi(t,\theta_s\omega)P^{+}\|{\rm e}^{-\frac{1}{2}\mu|t|}\right\}$$
is tempered from above, it is easy to see that
\begin{equation*}
\mathbb{P}(\Omega_N)\to 1, \mbox{\ \ \ \ \ as\ \ \ } N \to \infty.
\end{equation*} 
Note also that $\Omega_N$ is an increasing sequence of sets. Thus $\cup_{N}\Omega_N=\hat{\Omega}$ and $\hat{\Omega}$ has the full measure. In fact
$$\hat{\Omega}:=\left\{\omega:\ \sup_{s\in\mathbb{R}}\max\left\{\sup_{t\geq 0}\|\Phi(t,\theta_s\omega)P^{-}\|{\rm e}^{-\frac{1}{2}\mu|t|-\Lambda|s|},\ \sup_{t\leq 0}\|\Phi(t,\theta_s\omega)P^{+}\|{\rm e}^{-\frac{1}{2}\mu|t|-\Lambda|s|}\right\}<\infty\right\}.$$
Therefore it is invariant with respect to $\theta_s$ for all $s\in {\mathbb{R}}$.
Now we define
\begin{equation*}
\Omega^{\ast}_N=\bigcap_{n=-\infty}^{\infty}\theta_{n\tau}^{-1}\Omega_N.
\end{equation*}
Then it is easy to see that $\Omega^{\ast}_N$ is invariant with respect to $\theta_{n\tau}$ for each n. Besides we have $\Omega^{\ast}_N\subset \Omega^{\ast}_{N+1}$, which leads to
\begin{eqnarray*}
\bigcup_{N}\Omega^{\ast}_N=\bigcup_{N}\bigcap_{n=-\infty}^{\infty}\theta_{n\tau}^{-1}\Omega_N=\bigcap_{n=-\infty}^{\infty}\theta_{n\tau}^{-1}\left(\bigcup_{N}\Omega_N\right)=\bigcap_{n=-\infty}^{\infty}\theta_{n\tau}^{-1}\hat{\Omega}=\bigcap_{n=-\infty}^{\infty}\hat{\Omega}=\hat{\Omega},
\end{eqnarray*}
with $\mathbb{P}(\hat{\Omega})=1$.
Now we can define $Y:\hat{\Omega}\times \mathbb{R}\to \mathbb{R}^d$ as an combinations of $Y_N$ as follows
\begin{eqnarray}\label{NJ5}
Y:=Y_1\chi_{\Omega^{\ast}_1}+Y_2\chi_{\Omega^{\ast}_2\setminus\Omega^{\ast}_1}+\cdots+Y_N\chi_{\Omega^{\ast}_N\setminus\Omega^{\ast}_{N-1}}+\cdots.
\end{eqnarray}
 Thus it is easy to see that Y is $\mathcal{B}(\mathbb{R})\otimes\mathcal{F}$ measurable and satisfies the following property
\begin{eqnarray*}&&
Y(t+\tau,\omega)\\
&=&Y_1(t+\tau,\omega)\chi_{\Omega^{\ast}_1}(\omega)+Y_2(t+\tau,\omega)\chi_{\Omega^{\ast}_2\setminus\Omega^{\ast}_1}(\omega)+\cdots+Y_N(t+\tau,\omega)\chi_{\Omega^{\ast}_N\setminus\Omega^{\ast}_{N-1}}(\omega)+\cdots\\
&=&Y_1(t,\theta_{\tau}\omega)\chi_{\Omega^{\ast}_1}(\omega)+Y_2(t,\theta_{\tau}\omega)\chi_{\Omega^{\ast}_2\setminus\Omega^{\ast}_1}(\omega)+\cdots+Y_N(t,\theta_{\tau}\omega)\chi_{\Omega^{\ast}_N\setminus\Omega^{\ast}_{N-1}}(\omega)+\cdots\\
&=&Y_1(t,\theta_{\tau}\omega)\chi_{\Omega^{\ast}_1}(\theta_{\tau}\omega)+Y_2(t,\theta_{\tau}\omega)\chi_{\Omega^{\ast}_2\setminus\Omega^{\ast}_1}(\theta_{\tau}\omega)+\cdots+Y_N(t,\theta_{\tau}\omega)\chi_{\Omega^{\ast}_N\setminus\Omega^{\ast}_{N-1}}(\theta_{\tau}\omega)+\cdots\\
&=&Y(t,\theta_{\tau}\omega).
\end{eqnarray*}
Moreover $Y$ is a fixed point of ${\cal M}$. 

We can easily extend $Y$ to the whole $\Omega$ as $\mathbb{P}(\hat{\Omega})=1$, which is indistinguishable with $Y$ defined in (\ref{NJ5}). 
\end{proof}
\begin{rmk}
It is easy to see from (\ref{NJ5}) that $|Y|<\infty\ \ \mathbb{P}-a.s.$ Moreover, we don't know whether or not $Y\in C^{\Lambda}_{\tau}(\mathbb{R},L^2(\Omega,\mathbb{R}^d))$. However, for each $N$, $Y^N\in C^{\Lambda}_{\tau}(\mathbb{R},L^2(\Omega,\mathbb{R}^d))$. This suggests that $Y\in C^{\Lambda}_{\tau}(\mathbb{R},L^2(\Omega_N,\mathbb{R}^d))$ for each $N$. That is to say that $Y\in C^{\Lambda}_{\tau}(\mathbb{R},L^2_{loc}(\Omega,\mathbb{R}^d))$.
\end{rmk}

Now we combine the methods introduced in this section and in \cite{F-zh1} to study the following stochastic differential equations
\begin{equation}\label{SDEextra}
du=(Au+F(t,u))dt+\sum_{k=1}^MB_ku\circ dW^k_t+\sum_{k=1}^M \beta_k(t)dW^k_t,
\end{equation}
and
\begin{eqnarray}\label{hatY}
\hat{Y}(t,\omega)&=&\int_{-\infty}^t\Phi(t-s,\theta_s\omega)P^-F(s,\hat{Y}(s,\omega))ds-\int^{\infty}_t\Phi(t-s,\theta_s\omega)P^+F(s,\hat{Y}(s,\omega))ds\nonumber\\
&&+\sum_{k=1}^M\int_{-\infty}^t\Phi(t-s,\theta_s\omega)P^-\beta_k(s)dW^k_s-\sum_{k=1}^M\int^{\infty}_t\Phi(t-s,\theta_s\omega)P^+\beta_k(s)dW^k_s.\nonumber\\
\end{eqnarray}
In fact, this is only a combination of Eqn. (\ref{origT1}) we considered already in this section and \cite{F-zh1}. We include the result here as it is needed in the next section.

\begin{thm}\label{Main2} Assume that $A$, $F$ and $B_k$ satisfy the same conditions as in Proposition \ref{Main}. Let $\beta_k(t)=\beta_k(t+\tau)$ for any $t\in \mathbb{R}$ and there exists a constant $R_1$ s.t. $\|\beta_k(s_1)-\beta_k(s_2)\|^2\leq R_1|s_1-s_2|$. 
Then there exists a ${\cal B}(\mathbb{R})\otimes\mathcal{F}$-measurable map $\hat{Y}: \mathbb{R}\times\Omega\rightarrow \mathbb{R}^d$
satisfying Eqn. (\ref{hatY}) and $\hat{Y}(t+\tau, \omega)=\hat{Y}(t, \theta_{\tau}\omega)$ for
any $t\in \mathbb{R}$ and $\omega\in \Omega$.
\end{thm} 
\begin{proof}
We will adopt the same procedure as in proofs of Proposition \ref{Main} and Theorem \ref{Main1} and first show the fixed-point existence of the following mapping,
\begin{eqnarray}\label{hatYt}
&&
\hat{\cal M}^N(\hat{Y}^N)(t,\omega)\nonumber 
\\
&=&\int_{-\infty}^t\Phi^N(t-s,\theta_s\omega)P^-F(s,\hat{Y}^N(s,\omega))ds-\int^{\infty}_t\Phi^N(t-s,\theta_s\omega)P^+F(s,\hat{Y}^N(s,\omega))ds\nonumber\\
&&+\sum_{k=1}^M\int_{-\infty}^t\Phi^N(t-s,\theta_s\omega)P^-\beta_k(s)dW^k_s-\sum_{k=1}^M\int^{\infty}_t\Phi^N(t-s,\theta_s\omega)P^+\beta_k(s)dW^k_s.\nonumber\\
\end{eqnarray}
Here the proof differs from Proposition \ref{Main} in the Malliavin derivative part only, while all the other steps are similar. The Malliavin derivative can be easily computed as follows,
\begin{eqnarray*}
\mathcal{D}^l_r\hat{\cal M}^N(\hat{Y}^N)(t,\omega)&=&\int_{-\infty}^r\chi_{\{r\leq t\}}(r){\cal D}_{r}^l(\Phi^N_{t-s,s}P^-)F(s, \hat{Y}^N(s,\omega))ds\nonumber\\
&&-\int^{+\infty}_r\chi_{\{r\geq t\}}(r){\cal D}_{r}^l(\Phi^N_{t-s,s}P^+)F(s, \hat{Y}^N(s,\omega))ds\nonumber\\
&&+\int_{-\infty}^t\Phi^N_{t-s,\hat{s}}P^-\nabla F(s, \hat{Y}^N(s,\omega)){\cal D}_{r}^l  \hat{Y}^N(s,\omega)ds\nonumber\\
&&-\int^{+\infty}_t\Phi^N_{t-s,s}P^+\nabla F(s, \hat{Y}^N(s,\omega)){\cal D}_{r}^l  \hat{Y}^N(s,\omega)ds\nonumber\\
&&+\sum_{k=1}^M\int_{-\infty}^r\chi_{\{r\leq t\}}(r){\cal D}_{r}^l(\Phi^N_{t-s,s}P^-)\beta_k(s)dW^k_s\nonumber\\
&&-\sum_{k=1}^M\int^{+\infty}_r\chi_{\{r\geq t\}}(r){\cal D}_{r}^l(\Phi^N_{t-s,s}P^+)\beta_k(s)dW^k_s\nonumber\\
&&+\chi_{\{r\leq t\}}(r)\Phi^N_{t-r,r}P^-\beta_l(r)-\chi_{\{r\geq t\}}(r)\Phi^N_{t-r,r}P^+\beta_l(r).
\end{eqnarray*}
Actually in the above, we only need to take care of the last four terms. The other 
terms in (\ref{hatYt}) can be dealt with using the same method as in the proof of Proposition \ref{Main}. By It\^o isometry, Lemma \ref{340} and the property of $\beta$, it is easy to see that
 \begin{equation*}
\sup_{t\in[0,\tau)} \int_{\mathbb{R}}\mathbb{E}\left|\sum_{k=1}^M\int^{+\infty}_r\chi_{\{r\geq t\}}(r){\cal D}_{r}^l(\Phi^N_{t-s,s}P^+)\beta_k(s)dW^k_s\right|^2dr<\infty,
 \end{equation*}
 and
  \begin{equation*}
\sup_{t\in[0,\tau)} \int_{\mathbb{R}}\mathbb{E}\left|\sum_{k=1}^M\int_{-\infty}^r\chi_{\{r\leq t\}}(r){\cal D}_{r}^l(\Phi^N_{t-s,s}P^-)\beta_k(s)dW^k_s\right|^2dr<\infty,
 \end{equation*}
 Besides, by the estimate (\ref{PhiBound}), we can show for each $l\in \{1,2,\cdots,M\}$, 
   \begin{eqnarray*}
\sup_{t\in[0,\tau)}  \int_{\mathbb{R}}\mathbb{E}\left|\chi_{\{r\leq t\}}(r)\Phi^N_{t-r,r}P^-\beta_l(r)\right|^2dr&<&\infty ,
\\
\sup_{t\in[0,\tau)}  \int_{\mathbb{R}}\mathbb{E}\left|\chi_{\{r\geq t\}}(r)\Phi^N_{t-r,r}P^+\beta_l(r)\right|^2dr&<&\infty.
\end{eqnarray*}
 Also we can show for each $l\in \{1,2,\cdots,M\}$,
 \begin{eqnarray*}
 \sup_{t\in [0,\tau), \delta\in \mathbb{R}}\dfrac{1}{|\delta|}\mathbb{E}\int_{\mathbb{R}}\left|\mathcal{D}^l_{r+\delta}\hat{M}(\hat{Y}^N)(t,\omega)-\mathcal{D}^l_r\hat{M}(\hat{Y}^N)(t,\omega)\right|^2dr<\infty.
 \end{eqnarray*}
 To achieve this, we only need to check the following terms,
\begin{eqnarray*}
&& \sup_{t\in [0,\tau), \delta\in \mathbb{R}}\dfrac{1}{|\delta|}\mathbb{E}\int_{\mathbb{R}}\left|\chi_{\{r+\delta\leq t\}}(r+\delta)\Phi^N_{t-(r+\delta),(r++\delta)}P^-\beta_l(r+\delta)-\chi_{\{r\leq t\}}(r)\Phi^N_{t-r,r}P^-\beta_l(r)\right|^2dr\\
&\leq & \sup_{t\in [0,\tau), \delta\in \mathbb{R}}\dfrac{3}{|\delta|}\mathbb{E}\int_{t-\delta}^t\left|\Phi^N_{t-(r+\delta),(r++\delta)}P^-\beta_l(r+\delta)\right|^2dr\\
&&+ \sup_{t\in [0,\tau), \delta\in \mathbb{R}}\dfrac{3}{|\delta|}\mathbb{E}\int_{-\infty}^t\left|(\Phi^N_{t-(r+\delta),(r++\delta)}P^--\Phi^N_{t-r,r}P^-)\beta_l(r+\delta)\right|^2dr\\
&&+ \sup_{t\in [0,\tau), \delta\in \mathbb{R}}\dfrac{3}{|\delta|}\mathbb{E}\int_{-\infty}^t\left|\Phi^N_{t-(r+\delta),(r++\delta)}P^-(\beta_l(r+\delta)-\beta_l(r))\right|^2dr\\
&=:& \sum_{i=1}^3 A_i<\infty,
\end{eqnarray*}
and by Lemma \ref{340} for each $k\in  \{1,2,\cdots,M\}$,
\begin{eqnarray*}
&&\dfrac{1}{|\delta|}\mathbb{E}\int_{\mathbb{R}}\left|\int_{-\infty}^{r+\delta}\chi_{\{r+\delta\leq t\}}(r){\cal D}_{r+\delta}^l(\Phi^N_{t-s,s}P^-)\beta_k(s)dW^k_s\right .\\
&&\hskip10pt
\left.
-\int_{-\infty}^r\chi_{\{r\leq t\}}(r){\cal D}_{r}^l(\Phi^N_{t-s,s}P^-)\beta_k(s)dW^k_s\right|^2dr\\
&\leq & \dfrac{2}{|\delta|}\left(\mathbb{E}\int_{-\infty}^{t-\delta}\left|\int_{r}^{r+\delta}{\cal D}_{r+\delta}^l(\Phi^N_{t-s,s}P^-)\beta_k(s)dW^k_s\right|^2dr\right .\\
&&\hskip10pt
 \left .+\mathbb{E}\int_{t-\delta}^{t}\left|\int^{\infty}_{r+\delta}{\cal D}_{r+\delta}^l(\Phi^N_{t-s,s}P^-)\beta_k(s)dW^k_s\right|^2dr\right)\\
&=: & B_1+B_2<\infty.  
\end{eqnarray*}
The boundedness of $A_1$ can be derived from the estimate of $\Phi$ and boundedness of $\beta$, and $A_2$ by the same way as we dealt with $T_1$ in the proof of Lemma \ref{Lem1}. As for $A_3$, the Lipschitz condition of $\beta$ works. The boundedness of $B_1$ and $B_2$ can be done through the It\^o isometry, Lemma \ref{LEMMAMD} and boundedness of $\beta$. 

Then by the same argument as in the proof of Proposition \ref{Main}, we can prove the existence of fixed point for Eqn. (\ref{hatYt}). Using the same "measurable glue" method in the proof of Theorem \ref{Main1} one can obtain a measurable solution $Y$ of Eqn. (\ref{hatY}) satisfying $\hat{Y}(t+\tau, \omega)=\hat{Y}(t, \theta_{\tau}\omega)$ for
any $t\in \mathbb{R}$ and $\omega\in \Omega$.  
\end{proof}
\medskip

The existence of random periodic solution results obtained in Theorems \ref{Main1} and \ref{Main2}, together with the "equivalence" of random periodic solutions and periodic measure obtained in \cite{F-zh3}, implies the existence of periodic measure with respect to the skew product of the random dynamical system and metric dynamical system. For this, define $\mu:\mathbb{R}\times\Omega\times \mathcal{B}(\mathbb{R}^d)\to [0,1]$ by
\begin{eqnarray}\label{pm1}
(\mu_s)_{\omega}=\delta_{Y(s,\theta(-s)\omega)}.
\end{eqnarray}
Then 
$$(\mu_{s+\tau})_{\omega}=(\mu_{s})_{\omega}$$
and
$$u(t+s, s, \theta(-s)\omega)(\mu_{s})_{\omega}=(\mu_{t+s})_{\theta(s)\omega}.$$
Define the product space $\hat{\Omega}=\Omega\times \mathbb{R}^d$ with $\sigma$-field $\hat{\mathcal{F}}=\mathcal{F}\otimes \mathcal{B}(\mathbb{R}^d)$ and the skew product $\Theta: \Delta \times \hat{\Omega}\to \hat{\Omega}$ by
$$\hat{\Theta}(t+s,s)(\omega,x)=(\theta(t)\omega, u(t+s,s,\theta(-s)\omega)x).$$
Then for any $t_1, t_2\in \mathbb{R}^{+}$, $s\in \mathbb{R}$,
$$\hat{\Theta}(t_2+t_1+s,t_1+s)\hat{\Theta}(t_1+s,s)=\hat{\Theta}(t_1+t_2+s,s).$$
Set $\mu_s:\hat{\mathcal{F}}\to [0,1]$ has
\begin{eqnarray}\label{pm2}
\mu_s(dx, d\omega)=(\mu_s)_{\omega}(dx)\times \mathbb{P}(d\omega).
\end{eqnarray}
Then $\mu_{s+\tau}=\mu_{\tau},$
and
$$\hat{\Theta}(t+s,s)\mu_s=\mu_{t+s},\ \mbox{for any }t\in \mathbb{R}^{+}, s\in \mathbb{R}.$$
Thus $\mu_s$, $s\in \mathbb{R}$ is a periodic measure of the skew product $\hat{\Theta}$. We omit the full details here, see \cite{F-zh3}.

\section{Applications and examples}

First we consider a simple example of stochastic differential equations with time periodic coefficients.
\begin{exmp} Consider
\begin{equation}\label{X1}
dX(t)=-X(t)dt+c\cos{(t)}dt+10\sin{(t)}dW_t.
\end{equation}
Here $c$ is a constant and $W(t)$ is a one-dimensional Brownian motion on a probability space $(\Omega, {\cal F}, P)$. 
Applying the result in \cite{F-zh1} or Theorem \ref{Main2} in this paper, we can assert that Eqn. (\ref{X1}) has a random periodic solution.  
In fact, according to the equivalence theorem (Theorem \ref{T1-2'}), the random periodic solution of (\ref{X1}) can be written explicitly  
as follows
\begin{eqnarray*}
Y(t)&=&\int _{-\infty}^t{\rm e}^{-t+s}c\cos (s)ds+10\int _{-\infty}^t{\rm e}^{-t+s}\sin (s)dW_s\\
&
=& {c\over 2}(\cos (t)+\sin (t))+10\int _{-\infty}^t{\rm e}^{-t+s}\sin (s)dW_s.
\end{eqnarray*}
Actually it can be verified by direct simple calculations that $Y(t)$ satisfies definition \ref{feng-zhao1}.
When $c=0$, the solution of Eqn. (\ref{X1}) is a Ornstein-Uhlenbeck process with white noise of periodic coefficient. Similar to the case of time independent case that 
the dynamical system generated by the Orsnstein-Uhlenbeck process has a stationary path and an invariant measure, the example suggests that 
the semiflow generated by the time periodic 
Orsnstein-Uhlenbeck process has a random periodic path and a periodic measure.

To study a smilar equation with multiplicative noise, let us consider,
\begin{equation}\label{X2}
dX(t)=-X(t)dt+\cos{(t)}dt+10X(t)\circ dW_t.
\end{equation}
Theorem \ref{Main2} implies the existence of random periodic solutions to Eqn. (\ref{X2}). According to Theorem \ref{feng-zhao60}, 
the random periodic solution of Eqn. (\ref{X2}) is give explicitly by
\begin{eqnarray*}
Y(t)=\int_{-\infty}^t{\rm e}^{-(t-s)+10(W(t)-W(s))}\cos (s)ds.
\end{eqnarray*}

The numerical simulations of Eqn. (\ref{X1}) (taking $c=1$) and Eqn. (\ref{X2}) displayed by 
Fig.\ref {1} and Fig.\ref{2} demonstrate how the random periodic solutions fluctuate around the deterministic
periodic solution of the noiseless ordinary differential equation ${d\over dt}X(t)=-X(t)+\cos (t)$ in the additive noise case 
and the multiplicative noise case respectively.
\begin{figure}\label{1}
\begin{center}
  \scalebox{0.3}[0.25]{\includegraphics{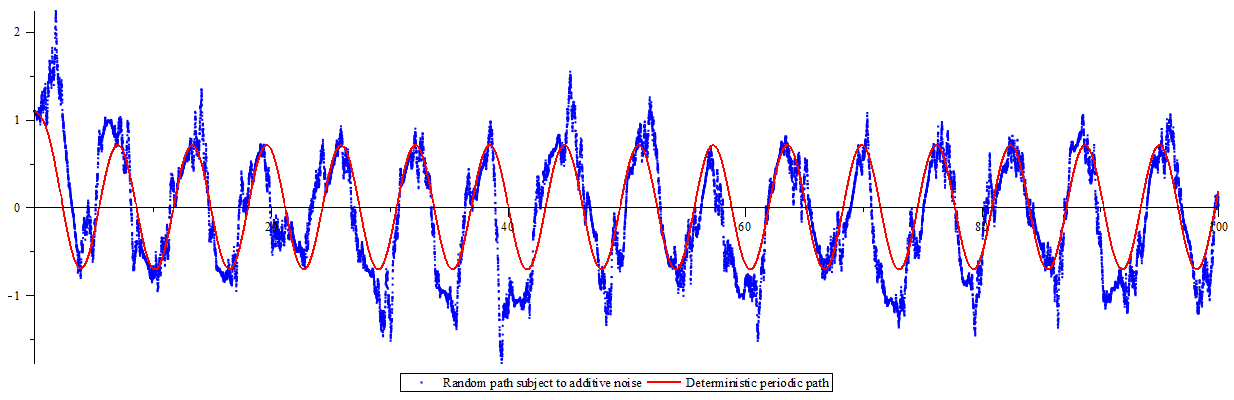}}
    \caption{Random trajectory subject to additive noise}
\end{center}
 \end{figure}
\vskip-40pt

\begin{figure}\label{2}
\begin{center}
  \scalebox{0.3}[0.25]{\includegraphics{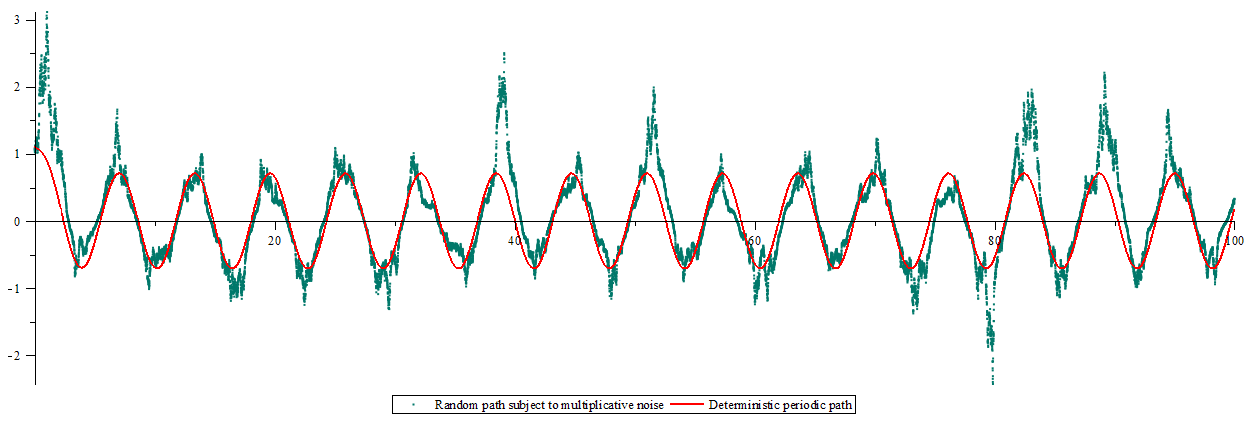}}
    \caption{Random trajectory subject to multiplicative noise}
\end{center}
 \end{figure}
\vskip-20pt
\end{exmp}
\vskip 1cm
Secondly, we apply the results of last section to study the following stochastic differential equations on $R^d$
\begin{equation}\label{FinalEqn}
dx=(Ax+f(x))dt+\sum_{k=1}^MB_kx\circ dW^k_t+\sum_{k=1}^M\gamma_kdW^k_t.
\end{equation}
As some $B_k$ and $\gamma_k$ can be zero, so this equation includes the case when the multiplicative noise and additive noise are independent, though in both the additive and multiplicative noise terms  we use the same multidimensional Brownian motions.
It is easy to see that under the conditions of Theorem \ref{Main2}, Eqn. (\ref{FinalEqn}) generates a cocycle random  dynamical system $\Psi: R^+\times \Omega\times R^d\to R^d$ such that for all $\omega\in \Omega$, (c.f. \cite{ar})
\begin{eqnarray}
  \Psi(t,\theta _s\omega)\circ \Psi(s,\omega)=\Psi (t+s,\omega), {\rm for \ all }\ t,s\geq 0.
  \end{eqnarray}
  Here $W_k$ and $\theta$ are the same as before. In this case
  the skew product $\bar \Theta$ is defined as $\bar \Theta(t)(\omega,x)=(\theta (t)\omega, \Psi(t,\omega)x)$.
   
\begin{thm}\label{Main3}
Let $A$, $B_k$ satisfy the same conditions as in Proposition \ref{Main} and the function $f\in C^3$ be uniformly bounded with bounded first order derivatives.
Assume the deterministic system
\begin{equation}\label{ODEFinal}
\dfrac{dx}{dt}=Ax+f(x)
\end{equation}
has a periodic solution $z$ with period $\tau>0$ and $z(t)$ is $C^3$ in $t$. Then Eqn. (\ref{FinalEqn}) has a random periodic solution of period $\tau$ i.e. there exists 
$Y:R\times \Omega\to R^d$ 
such that for any $t\in R^+,s\in R$,
\begin{eqnarray}\label{DS}
\Psi(t,\theta(s)\omega)Y(s,\omega)=Y(t+s, \omega),\ \ Y(s+\tau,\omega)=Y(s,\theta_{\tau}\omega).
\end{eqnarray}
Moreover, the measure $\mu_s$ given in (\ref{pm2}) is a periodic measure of $\bar \Theta$, and $\bar \mu
={1\over \tau}\int_0^{\tau}\mu_sds$ is an invariant measure whose random factorisation has the support $\{Y(s,\theta(-s)\omega): \ 0\leq s<\tau\}$ which is a closed curve.
\end{thm}
\begin{proof}
Let
$$u(t,s,\omega)u_0=\Psi(t-s,\theta(s)\omega)(u_0+z(s))-z(t),\ \ t\geq s$$
with $\Psi$ satisfying Eqn. (\ref{FinalEqn}). Then $u(t)$ (in short for $u(t,s)$) satisfies 
\begin{equation*}
du(t)=(Au(t)+f(u(t)+z(t))-f(z(t)))dt+\sum_{k=1}^MB_k u(t)\circ dW^k_t+\sum_{k=1}^M (B_kz(t)+\gamma _k)dW^k_t,
\end{equation*}
with $u(s)=u_0$. 
The above equation has a random periodic solution $\hat{Y}$ by Theorem \ref{Main2}, so Eqn. (\ref{FinalEqn}) 
has a random periodic solution $Y(t,\omega)=\hat{Y}(t,\omega)+z(t)$. In fact for any $t\ge 0$,
\begin{eqnarray*}
\Psi(t,\theta(s)\omega)Y(s,\omega)&=&u(t,s,\omega)(Y(s,\omega)-z(s))+z(t+s)\\
&=& u(t,s,\omega)\hat Y(s,\omega)+z(t+s)\\
&=& \hat Y(t+s,\omega)+z(t+s)=Y(t+s,\omega),
\end{eqnarray*}
and the random periodicity of $Y$ is obvious.  The 
last claim follows from the existence of a random periodic solution and the result on periodic and invariant measures in \cite{F-zh3}.
\end{proof}

In the above theorem, 
the main assumption is that the deterministic system (when noise is switched off) has a periodic solution, the other assumptions are very mild.  
Many ordinary differential equations modeling real world problems arising from biology, chemistry, chemical engineering, climate dynamics, economics etc.
have
periodic solutions. Here we can make the function $f$ bounded by a smooth truncation procedure outside 
a sufficiently large ball mentioned in \cite{F-zh2} if necessary without changing their local  
dynamical behavior. 
Therefore Theorem \ref{Main3} gives the existence of random periodic solutions of stochastic differential equations arising from many 
real world applications.
 
As a simple, but typical example, we consider a random dynamical system generated by a random perturbation to the following
ordinary differential equation in $\mathbb{R}^2$,
\begin{eqnarray}\label{Example}
\left\{\begin{array}{l}dy_1(t)=-y_2(t)dt-y_1(t)dt+y_1(t)(2-y^2_1(t)-y^2_2(t))\phi(y_1(t),y_2(t))dt,\\
dy_2(t)=y_1(t)dt-y_2(t)dt+y_2(t)(2-y^2_1(t)-y^2_2(t))\phi(y_1(t),y_2(t))dt,\\
                                     \end{array}\right.
\end{eqnarray}
where $\phi$ is a smooth function such that
\begin{eqnarray}\label{ExampleBounded}
\phi(y_1,y_2)=\left\{\begin{array}{l}1,\ \mbox{when } y_1^2+y^2_2\leq 2^{100},\\
0,\ \mbox{when } y_1^2+y^2_2\geq 2^{101}.\\
                                     \end{array}\right.
\end{eqnarray}
It is not hard to see that the limit cycle of this system is $y_1^2+y_2^2=1$.
\begin{exmp} 
 Consider the corresponding stochastic differential equation with additive noise,
\begin{eqnarray}\label{ExampleA}
\left\{\begin{array}{l}dy_1(t)=-y_2(t)dt-y_1(t)dt+y_1(t)(2-y^2_1(t)-y^2_2(t))\phi(y_1(t),y_2(t))dt+10dW_t^1,\\
dy_2(t)=y_1(t)dt-y_2(t)dt+y_2(t)(2-y^2_1(t)-y^2_2(t))\phi(y_1(t),y_2(t))dt+10dW_t^2.\\
                                     \end{array}\right.
\end{eqnarray}
We can apply the result in \cite{F-zh2} or Theorem \ref{Main3}
to assert that Eqn. (\ref{ExampleA}) has a random periodic solution. 
\begin{figure}\label{3}
\begin{center}
   \scalebox{0.25}[0.25]{\includegraphics{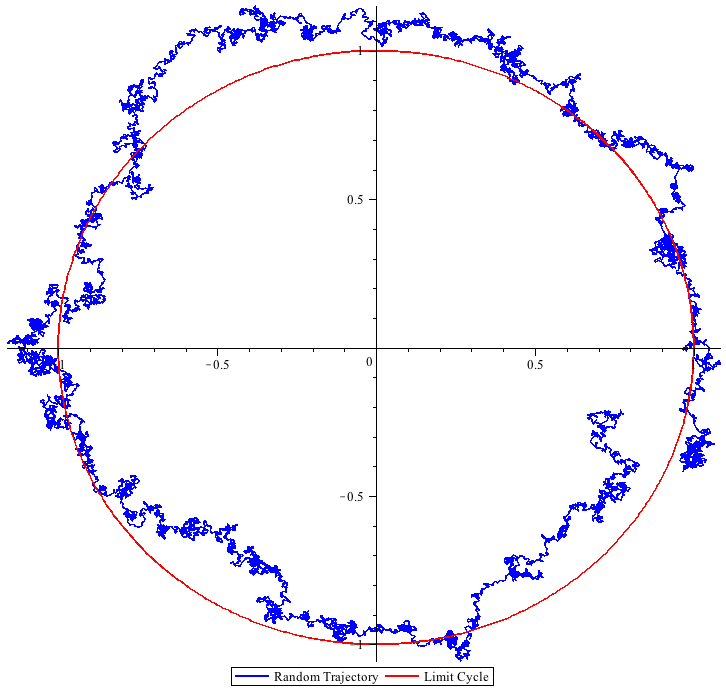}}
    \caption{Random trajectory subject to additive noise}
    \end{center}
\end{figure}
Fig.\ref{3} illustrates a numerical simulation of a sample path of the random periodic solution. 
Fig.\ref{4} illustrates the formulation of the invariant measure.  We start at $t=0$ with uniform distribution on $[-2,2]\times[-2,2]$. 
Subject to the same random perturbations, all those points 
begin to move towards a random closed 
curve simultaneously. In the long run, they evolve into a random closed curve, which represents the support of the invariant measure 
with respect the skew product dynamical system as described in Theorem \ref{Main3}. 
The closed curve moves randomly, its law is a periodic
measure of the corresponding Markovian semigroup. 
\begin{figure}\label{4}
\begin{center}
   {\includegraphics[trim=1cm  10.5cm 1.5cm  1.5cm, clip=true, width=0.7\textwidth]{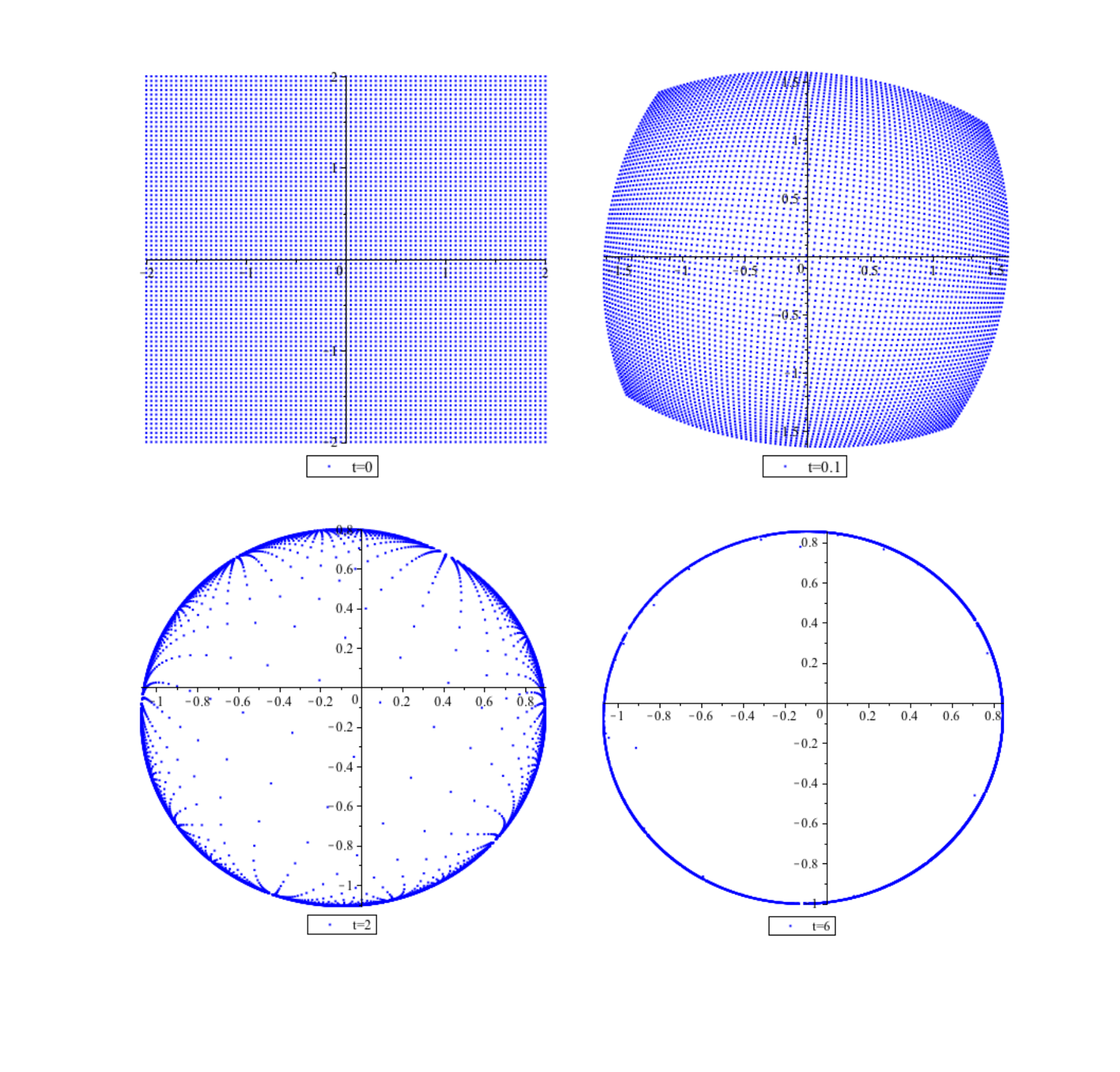}}
    \caption{Random periodic evolution}
    \end{center}
\end{figure}

Now we consider the random perturbation of (\ref{Example}) subject to multiplicative linear noise,
\begin{eqnarray}\label{ExampleML}
\left\{\begin{array}{l}dy_1(t)=-y_2(t)dt-y_1(t)dt+y_1(t)(2-y^2_1(t)-y^2_2(t))\phi(y_1(t),y_2(t))dt+10y_1\circ dW_t^1,\\
dy_2(t)=y_1(t)dt-y_2(t)dt+y_2(t)(2-y^2_1(t)-y^2_2(t))\phi(y_1(t),y_2(t))dt+10y_2\circ dW_t^2.\\
                                     \end{array}\right.
\end{eqnarray}
Note the commutativity assumption in Theorem \ref{Main2} is not satisfied. But following Remark \ref{rmk1}, we can still use the previous result to conclude the Eqn. (\ref{ExampleML}) has a 
random periodic solution.  A sample path is given by a numerical simulation in Fig.\ref{5}.
Similar phenomena as the formation of an invariant measure as demonstrated for Eqn. (\ref{ExampleA}) in Fig.\ref{4}
can be obtained. The detail is omitted here.
\begin{figure}\label{5}
\begin{center}
  \scalebox{0.25}[0.25]{\includegraphics{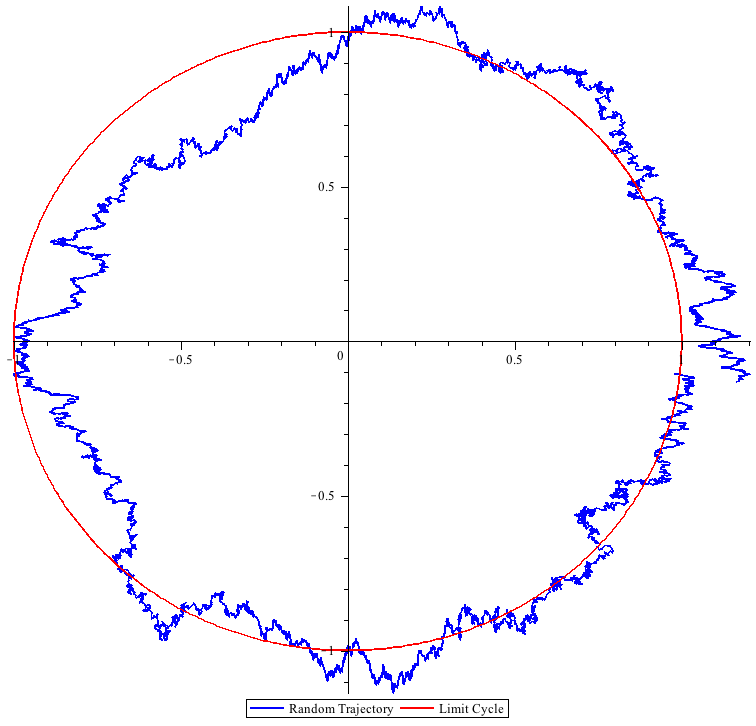}}
    \caption{Random Trajectory subject to multiplicative linear noise}
\end{center}
 \end{figure}
 \end{exmp}
 

\section*{Appendix}
 \begin{proof}[Proof of Lemma \ref{LEMMA1C3}]
(i). It is easy to show that $\Phi$ satisfies the condition $\sup\limits_{0\leq t\leq 1}\log^+\|\Phi(t,\omega)^{\pm1}\|\in L^1(\Omega)$. 
  So the MET theorem ensures the existence of the random Oseledets splitting 
  $$\mathbb{R}^d=E_p(\omega)\oplus E_{p-1}(\omega)\oplus \dots\oplus E_{m+1}(\omega)\dots\oplus E_{1}(\omega)
  ,$$ 
  and the corresponding random projections $P^{\pm}(\omega)$. But if we consider the forward filtration and $\lim_{t\to \infty}(\Phi(t,\omega)^{\ast}\Phi(t,\omega))^{1/2t}:=\Psi(\omega)$, the mutually commutative property of $A$, $A^{\ast}$, $B_k$, and $B_k^{\ast}$ leads to 
\begin{eqnarray*}
\Psi(\omega)=\lim_{t\to \infty}\exp\left\{\frac{1}{2}(A+A^{\ast})+\sum_{k=1}^M\dfrac{(B_k+B^{\ast}_k)W^k_t}{2t} \right\}
=\exp\left\{\frac{A+A^{\ast}}{2}\right\}.
\end{eqnarray*}
Note ${\rm e}^{\mu_p}<{\rm e}^{\mu_{p-1}}\dots < {\rm e}^{\mu_{m+1}}<1<{\rm e}^{\mu_{m}}<\dots<{\rm e}^{\mu_1}$ are the 
eigenvalues of $\exp\left\{\frac{A+A^{\ast}}{2}\right\}$, and  $U_p, \cdots, U_1$ are still the corresponding orthogonal eigenspaces, with multiplicity  $d_i:=\mbox{dim}\ U_i$. Define $V_{p+1}:=\{0\}$, and for $1\leq i\leq p$, $i\in \mathbb{N}$, 
\begin{equation}\label{VV}
V_i:=U_p\oplus U_{p-1}\oplus \cdots\oplus U_{i}.
\end{equation}
Therefore
\begin{equation}\label{VV1}
V_p\subset V_{p-1}\subset \cdots \subset V_i\subset \cdots \subset V_1=\mathbb{R}^d
\end{equation}
defines a forward filtration.

Now we consider the backward filtration and $\lim_{t\to \infty}(\Phi(-t,\omega)^{\ast}\Phi(-t,\omega))^{1/2t}:=\tilde{\Psi}(\omega)$. Note
\begin{eqnarray*}
\tilde{\Psi}(\omega)=\lim_{t\to \infty}\exp\left\{-\frac{1}{2}(A+A^{\ast})-\sum_{k=1}^M\dfrac{(B_k+B^{\ast}_k)W^k_{-t}}{2t} \right\}=\exp\left\{-\frac{A+A^{\ast}}{2}\right\}.
\end{eqnarray*}
Let $\tilde{\mu}_i=-\mu_{p+1-i}$. Then $\tilde{\mu}_p<\tilde{\mu}_{p-1}\dots < \tilde{\mu}_{p+1-m}<0<\tilde{\mu}_{p-m}<\dots<\tilde{\mu}_1$ are the eigenvalues of $-\frac{A+A^{\ast}}{2}$. Let $\tilde{U}_p, \cdots, \tilde{U}_1$ be the corresponding eigenspaces, with multiplicity  $\tilde{d}_i:=\mbox{dim}\ \tilde{U}_i$. Then $\tilde{U}_i=U_{p+1-i}$. 
Define $\tilde{V}_{p+1}:=\{0\}$, and for $1\leq i\leq p$, $i\in \mathbb{N}$, 
\begin{equation}\label{VV3}
\tilde{V}_i:=\tilde{U}_p\oplus \tilde{U}_{p-1}\oplus\cdots\oplus \tilde{U}_{i}=U_1\oplus U_2\oplus\cdots\oplus U_{p+1-i}.
\end{equation}
Therefore
\begin{equation}\label{VV4}
\tilde{V}_p\subset \tilde{V}_{p-1}\subset \cdots \subset\tilde{V}_i\subset \cdots \subset\tilde{V}_1=\mathbb{R}^d
\end{equation}
defines the backward filtration. Then we can construct the space $E_i$ as the intersection of certain spaces from the forward filtration (\ref{VV1}) and the backward filtration (\ref{VV4}), 
\begin{equation}\label{VVV}
E_i:=V_i\cap \tilde{V}_{p+1-i}=U_i.
\end{equation}
Thus the Lyapunov exponents of $\Phi$ depend on $\frac{1}{2}(A+A^{\ast})$ only. This implies that the  
Oseledets spaces are non-random and so are the corresponding projections $P^{\pm}$.

(ii). Note when $t\leq 0$,
\begin{eqnarray*}
\|\Phi(t,\omega)P^+\|&=&\|P^{+{\ast}}\Phi(t,\omega)^{\ast}\Phi(t,\omega)P^+\|^{1/2}\\
&=&\left\|P^{+{\ast}}\exp\left\{(A+A^{\ast})t+\textstyle\sum\nolimits_{k=1}^M(B_k+B^{\ast}_k)W^k_t\right\}P^+ \right\|^{1/2}\\
&\leq &\left\|P^{+{\ast}}\exp\left\{(A+A^{\ast})t\right\}P^+ \right\|^{1/2}\left\|\exp\left\{\textstyle\sum\nolimits_{k=1}^M(B_k+B^{\ast}_k)W^k_t \right\}\right\|^{1/2}\\
&\leq &\left\|\exp\left\{\frac{1}{2}(A+A^{\ast})t\right\}P^+ \right\|\exp\left\{\frac{1}{2}\textstyle\sum\nolimits_{k=1}^M\|B_k+B^{\ast}_k\||W^k_t| \right\}\\
&\leq & e^{\frac{1}{2}\mu t+\sum_{k=1}^M\|B\||W^k_t|}e^{\frac{1}{2}\mu_m t},
\end{eqnarray*}
where  $\|B\|:=\frac{1}{2}\max_{k\in\{1,2,\cdots,M\}}\|B_k+B^{\ast}_k\|$, $\mu:=\min\{-\mu_{m+1},\mu_m\}>0$. 
Define 
 \begin{equation}\label{Tempered}
C(\omega):=\sup_{t\in \mathbb{R}}C(t,\omega):=\sup_{t\in \mathbb{R}} {\rm e}^{-\frac{1}{2}\mu |t|+\sum_{k=1}^M\|B\||W^k_t|}\geq 1.
\end{equation}
Now it suffices to check that $C(\omega)$ is tempered from above. Similarly as in \cite{duan-lu}, 
 using
 $|W(t+s)|\leq C_{\delta,\omega}+|s|^{\delta}+|t|^{\delta}\ \ \mathbb{P}-a.s.$ for some $\frac{1}{2}<\delta<1$, from the iterated logarithm law of Brownian motion,
 we have 
 \begin{eqnarray*}
\lim_{s\to \pm\infty}\frac{1}{|s|}\log^{+} \sup_{t\in \mathbb{R}}  C(t,\theta_{s}\omega)&=&\lim_{s\to \pm\infty}\frac{1}{|s|}\log \sup_{t\in \mathbb{R}}  C(t,\theta_{s}\omega)\\
&=&\lim_{s\to \pm\infty}\frac{1}{|s|}\sup_{t\in \mathbb{R}}\log {\rm e}^{-\frac{1}{2}\mu |t|+\sum_{k=1}^M\|B\||\theta_{s}W^k_t|}\\
&\leq &\lim_{s\to \pm\infty}\frac{1}{|s|}\sup_{t\in \mathbb{R}} \left(-\frac{1}{2}\mu |t|+\sum_{k=1}^M\|B\||W^k_{t+s}|\right)+\lim_{s\to \pm\infty}\sum_{k=1}^M\|B\|\frac{|W^k_{s}|}{|s|}\\
&\leq &\lim_{s\to \pm\infty}\frac{1}{|s|}\sup_{t\in \mathbb{R}} \left(-\frac{1}{2}\mu |t|+M\|B\||t|^{\delta}\right)\\
&&+\sup_{t\in \mathbb{R}} \lim_{s\to \pm\infty}M\|B\|\frac{|s|^{\delta}}{|s|}+\sup_{t\in \mathbb{R}} \lim_{s\to \pm\infty}\frac{M\|B\||C_{\delta,\omega}|}{|s|}\\
&= &\lim_{s\to \pm\infty}\frac{1}{|s|}\sup_{t\in \mathbb{R}} \left(-\frac{1}{2}\mu |t|+M\|B\||t|^{\delta}\right)= 0,\ \ \ \mathbb{P}-a.s.,
\end{eqnarray*}
where the last inequality holds due to the fact that $\sup_{t\in \mathbb{R}} (-\frac{1}{2}\mu |t|+M\|B\||t|^{\delta})<\infty$. This together with the fact that
$$\lim_{s\to \pm\infty}\frac{1}{|s|}\log \sup_{t\in \mathbb{R}}  C(t,\theta_{s}\omega)\geq 0,$$ 
leads to that $C(\omega)$ is a tempered random variable. 
Similar argument applies to $\Phi(t,\theta_{s}\omega)P^-$. Finally by definition of random variable tempered from above, we can easily conclude that $\Phi(t,\theta_s\omega)P^-$ and $\Phi(t,\theta_s\omega)P^+$ satisfy (\ref{PhiBound}).  
 \end{proof}

\begin{proof}[Proof of Corollary \ref{LEMMA2C3}]
We consider $P^-$ case only. The estimation for $P^+$ case can be derived analogously.
From Eqn. (\ref{phiT1}) and the definition of $P^-$, it is natural to express
 the projection $ \Phi P^-: \mathbb{R}\times \Omega\rightarrow \mathcal{L}(\mathbb{R}^d, E^- )$ as follows,
 \begin{eqnarray*}
\Bigg\{\begin{array}{l}d\Phi(t,\omega)P^-=A\Phi(t)P^-\,dt+\sum\limits_{k=1}^MB_k \Phi(t)P^-\circ  d W^k_t,   \\
\Phi(0,\omega)P^-=P^-\in \mathcal{L}(\mathbb{R}^d, E^- ).
                                     \end{array}
\end{eqnarray*}
Then for any $t,\hat{s}\in \mathbb{R}$, by the ergodic property of $\theta$ and Holder's inequality we have that
\begin{eqnarray*}
&& \mathbb{E} \|P^--\Phi(t, \theta_{\hat{s}}\cdot)P^-\|^2\\
 &=& \mathbb{E}\Big\|\int_0^{t}\Big(A+\frac{1}{2}\sum_{k=1}^MB^2_k\Big)\Phi_{{\hat{h}},{\hat{s}}}P^-d\hat{h}+\sum_{k=1}^M\int_0^{t}B_k \Phi_{{\hat{h}},{\hat{s}}}P^-dW^k_{{\hat{h}}+{\hat{s}}} \Big\|^2\\
  &\leq &(M+1)\Big\|A+\frac{1}{2}{\sum_{k=1}^M}B^2_k\Big\|^2|t| \int_0^{t}\mathbb{E}\|\Phi_{{\hat{h}}}P^-\|^2d\hat{h}+(M+1){\sum_{k=1}^M}\|B_k\|^2\int_{\hat{s}}^{t+\hat{s}}\mathbb{E}\|\Phi_{{\hat{h}}-{\hat{s}}}P^-\|^2d\hat{h}\\
&\leq& 2^M(M+1)\Big(2\|A\|^2|t|+\Big(\textstyle\sum\nolimits_{k=1}^M\|B_k \|^2\Big)^2|t|+\sum_{k=1}^M\|B_k \|^2\Big){\rm e}^{2\|A\||t|+2M\|B\|^2|t|}|t|,
\end{eqnarray*}
where 
\begin{eqnarray*}
\int_0^{t}\mathbb{E}\|\Phi_{{\hat{h}}}P^-\|^2d\hat{h}&\leq& \int_0^{t}\mathbb{E}\|{\rm e}^{A\hat{h}+\sum_{k=1}^{M}B^kW_{\hat{h}}^k}\|^2d\hat{h}\leq  \int_0^{t}{\rm e}^{2\|A\|\hat{h}}\prod_{k=1}^{M}\mathbb{E}{\rm e}^{2\|B\||W_{\hat{h}}^k|}d\hat{h}\nonumber\\
&=& 2^{M}\int_0^{t}{\rm e}^{2\|A\|\hat{h}}{\rm e}^{2M\|B\|^2\hat{h}}d\hat{h}\leq  2^{M}{\rm e}^{2\|A\||t|+2M\|B\|^2|t|}|t|.
\end{eqnarray*}
Finally the last inequality can be easily drawn from above.
\end{proof}

\begin{proof}[Proof of Theorem \ref{B-S1}]
According to the generalized Arzel\`a$-$Ascoli lemma (c.f. \cite{kelley}), it suffices to check the uniform equicontinuity and pointwise relative compactness of $v_n$.
First we claim that $\{v_n(t,\cdot),\ n\in \mathbb{N}\}$ is relatively compact in $L^2(\Omega)$ for any fixed $t$. To achieve this, we decompose $v_n$ as Wiener-It\^o chaos expansions (c.f.\cite{bernt}),
\begin{equation}\label{NJ0404}
v_n(t,\omega)=\sum_{m=0}^\infty I_m(f_n^m(\cdot,t))(\omega),
\end{equation}
where $f_n^m(\cdot,t)$  are symmetric elements in $L^2(\mathbb{R}^m)$ for each $m\geq 0$ and each $t\in [a,b]$. By the similar argument in the proof of Theorem 1 in \cite{bally}, the relative compactness of $\{v_n\}_{n\in \mathbb{N}}$ is reduced to the relative compactness of $\{f_n^m\}_{n\in \mathbb{N}}$ for each finite $m\in \mathbb{N}$.

When $m=0$, $f_n^0(t)=\mathbb{E}v_n(t)$, and for any $t_1, t_2\in  [a,b]$, hypotheses (i) and (ii) imply the uniform boundedness of  $f_n^0$, 
\begin{eqnarray*}
&&\sup_n\sup_{t\in  [a,b]} |f_n^0(t)|\leq \sup_n\sup_{t\in  [a,b]}\sqrt {\mathbb{E}|v_n(t)|^2}\leq\sup_{n\in N}\sup_{t\in  [a,b]} ||v_n(\cdot, t)||_{{1,2}} <\infty.\\
\end{eqnarray*}
Besides, applying Jensen's inequality gives the uniform equicontinuity of  $f_n^0$, 
\begin{eqnarray*}
\sup_n|f_n^0(t_1)-f_n^0(t_2)|&=&\sup_n |\mathbb{E}(v_n(t_1)-v_n(t_2))|\leq \sup_n\mathbb{E} |v_n(t_1)-v_n(t_2)|\\
&\leq &\sup_n\sqrt {\mathbb{E}|v_n(t_1)-v_n(t_2)|^2}\leq \sqrt{C|t_1-t_2|}.
\end{eqnarray*}
So $\{f_n^0\}_{n=1}^\infty$ is relatively compact in $C( [a,b])$ according to the classical Arzel\`a-Ascoli lemma.

Using a similar argument as in the proof of Theorem 2 in \cite{bally} for each $m\geq 1$ with the general relative compactness criterion (c.f.Theorem 2.32 in \cite{adams}), we claim that  $\{f_n^m(\cdot, t)\}_{n\in \mathbb{N}}$ is relatively compact in $L^2(\mathbb{R}^m)$ for each fixed $t$. 
To see this, let $h=(h_1, \ldots,h_m)\in \mathbb{R}^{m}$. It holds that
\begin{eqnarray*}
&&\|\tau_{h}f^m_{n}-f^m_{n}\|^2_{L^2( \mathbb{R}^m)} \\
&=&\int_{\mathbb{R}^m}|f_n^m(t,t_1+h_1,\cdots ,t_m+h_m)-f_n^m(t,t_1,\cdots ,t_m)|^2dt_1\cdots dt_m\\ 
&\leq & C\sum^m_{i=1} \int_{\mathbb{R}^m}|f_n^m(t,t_1,\cdots,t_{i-1}, t_{i}+h_i, t_{i+1}+h_{i+1}, \cdots, t_m+h_m)\\
&&\hspace{2cm}-f_n^m(t,t_1,\cdots,t_{i-1}, t_{i}, t_{i+1}+h_{i+1}, \cdots, t_m+h_m)|^2dt_1\cdots dt_m
\end{eqnarray*}
\begin{eqnarray*}
&= & C\sum^m_{i=1} \int_{\mathbb{R}}\|f^m_{n}(t,\ldots,t_i+h_i,\ldots)-f^m_{n}(t,\ldots,t_i,\ldots)\|^2_{L^2(\mathbb{R}^{m-1})}dt_i\\
&=&\frac{C}{(m-1)!}\sum^m_{i=1}\int_{\mathbb{R}} \mathbb{E}|I_{m-1}(f^m_{n}(t,\cdots, t_i+h_i,\cdots)-f^m_{n}(t,\cdots,t_i,\cdots))|^2dt_i\\
&\leq &\frac{C}{m m!} \sum^m_{i=1} \int_{\mathbb{R}}\mathbb{E}\Big|\sum_{m\geq 1}mI_{m-1}\big(f^m_{n}(t,\cdots, t_i+h_i,\cdots)-f^m_{n}(t,\cdots, t_i,\cdots)\big)\Big|^2dt_i\\
&\leq &\frac{C}{m!}\int_{\mathbb{R}}\mathbb{E}\Big|\mathcal{D}_{r+h_1}v_n(t)-\mathcal{D}_{r}v_n(t)\Big|^2dr
\leq  C|h_1|,
\end{eqnarray*}
where $C$ is a constant depending on $m$. Moreover, for any $\epsilon>0$, there exists $[\alpha, \beta]\subset \mathbb{R}$ such that 
$$
\int_{\mathbb{R}\setminus [\alpha, \beta]}\mathbb{E}|\mathcal{D}_{r}v_n(t)|^2dr<\epsilon.
$$
Let $G=[\alpha, \beta]\times[-1,1]^{d-1}$. Then we have
\begin{eqnarray*}
\int_{R^m\setminus \bar{G}}|f_n^m(t,t_1,\cdots,t_m)|^2dt_1\cdots dt_m
&\leq &C\int_{\mathbb{R}\setminus [\alpha, \beta]}\|f_n^m(t,r,t_2,\cdots,t_m)\|_{L^2(\mathbb{R}^{m-1})}^2dr\\
&\leq & \frac{C}{m!}\int_{\mathbb{R}\setminus [\alpha, \beta]}\mathbb{E}\Big|\sum_{m\geq 1}mI_{m-1}(f^m_{n}(t,r,\cdot))\Big|^2dr\\
&\leq & \frac{C}{m!}\int_{\mathbb{R}\setminus [\alpha, \beta]}\mathbb{E}|\mathcal{D}_{r}v_n(t)|^2dr
\leq   C\epsilon.
\end{eqnarray*}
By now it has been showed that $\{f_n^m(\cdot,t), n\in\mathbb{N}\}$ is relatively compact in $L^2(\mathbb{R}^m)$ for each finite $m$ and fixed $t\in  [a,b]$, which is equivalent to $\{v_n(t),\ n\in\mathbb{N}\}$
being pointwise relatively compact in $L^2(\Omega)$ for any fixed $t$.

But it is known from hypothesis (ii) that $v_n$ are equi-continuous in time. So by generalized Arzel\`a-Ascoli Lemma, we conclude that  $\{v_n\}_{n=1}^\infty$ is relatively compact in $C([a,b], L^2(\Omega))$.
\end{proof}

\section*{Acknowledgements}

We would like to thank K. D. Elworthy, Y.Z. Hu, K. Khanin, Z. Lian, D. Nualart, S.G. Peng and Y.X. Yang 
for useful conversations. We are grateful to the referee for his/her useful comments.


\begin{thebibliography}{[99]}

\bibitem{adams} R. Adams and J. Fournier, {\it Sobolev Spaces},
  Academic press (2003). 

\bibitem{ar} L. Arnold, {\it Random dynamical systems}, Springer-Verlag Berlin Heidelberg New York (1998).


\bibitem{bally} V. Bally and B. Saussereau, A relative compactness criterion in Wiener-Sobolev spaces and application to semi-linear stochastic PDEs, {\it Jounals of Functional Analysis,} 210 (2004), 465-515.

\bibitem{TB} T. Bartsch, F. Revuelta, R. Benito, F. Borondo, Reaction rate calculation with time-dependent invariant manifolds, {\it Journal of Chemical Physics}, 136 (2012), 224-510.

\bibitem{blw} P. W. Bates, K.N. Lu and B.X. Wang,  Attractors of non-autonomous stochastic lattice systems in weighted spaces,
{\it Physica D}, 289 (2014), 32-50.

\bibitem{bismut} J.-M. Bismut,
A generalized formula of Ito and some other properties of stochastic flows, {\it Z. Wahrscheinlichkeitstheorie Verw. Gebiete},
Vol. 55 (1981), 331-350.

\bibitem{brzezniak} Z. Brzezniak and J. van Neerven, Banach apace valued Ornstein-Uhlenbeck processes indexed by the circle, in: {\it Evolution equations and their applications in physical and life sciences, Proceedings of Internatio}, 215 (2001), 435-452.  

\bibitem{chekroun} M. D. Chekroun, E. Simonnet and M. Ghil, Stochastic climate dynamics: random
attractors and time-dependent invariant measures, {\it Physica D}, 240 (2011), 1685-1700.

\bibitem{da-mall} G. Da Prato, P. Malliavin, D. Nualart, Compact families of Wiener functionals,{\it C. R. Acad. Sci. Paris, Ser. I Math.} 315 (1992), 1287-1291.



\bibitem{duan-lu} J. Duan,  K. Lu, B. Schmalfuss, Invariant manifolds for stochastic partial differential equations, {\it Annals of Probability,} 31 (2003), 2109-2135.


\bibitem{kh-ma-si} W. E, K. Khanin, A. Mazel, Ya. Sinai, Invariant measures for Burgers equation with stochastic forcing, {\it Ann. of Math.}, 151 (2000), 877-960.

\bibitem{elworthy} K.D. Elworthy, Stochastic dynamical systems and their flows, In: {\it Stochastic Analysis}, ed. A.
Friedman, M. Pinsky, London-New York: Academic Press 1978, 79-95.


\bibitem{F-zh1} C.R. Feng, H.Z. Zhao and B. Zhou, Pathwise random periodic solution of Stochastic differential Equations, 
{\it J. Differential Equations}, 251 (2011), 119-149.

\bibitem{F-zh2} C.R. Feng, H.Z. Zhao, Random periodic solutions of SPDEs via integral equations and Wiener-Sobolev compact embedding,
{\it J. Funct. Anal.},  262 (2012), 4377-4422.

\bibitem{F-zh3} C.R. Feng and H.Z. Zhao, Random periodic processes, periodic measures and ergodicity, {\it arXiv: 1408.1897}, submitted.

\bibitem{fwz2} C.R. Feng, Y. Wu and H.Z. Zhao, Anticipating Random Periodic Solutions--II. 
 SPDEs with Multiplicative Linear Noise, in preparations.

\bibitem{franses} P. H. Franses, {\em Periodicity and Stochastic Trends in Economic Times Series}, Oxford University Press, 1996.

\bibitem{ik-wa} N. Ikeda, S. Watanabe, {\em Stochastic Differential Equations and Diffusion Processes}, North Holland Publ. Co., Amsterdam Ñ Oxford Ñ New York 1981.



\bibitem{kelley} J. Kelley,  {\it General topology}. Springer (1975).

\bibitem{ku2} H. Kunita,{\it Stochastic flows and stochastic differential equations.} Cambridge University Press (1990).



\bibitem{mattingly} J. Mattingly, Ergodicity of 2D Navier-Stokes equations with random forcing and large viscosity, {\it Comm. Math. Phys.,} 206 (1999), 273-288.

\bibitem{meyer} P.-A. Meyer,
Flot dÕune \'equation diff\'erentielle stochastique, {\it
S\'eminaire de probabilit\'es (Strasbourg)}, Vol. 15 (1981), p. 103-117.

\bibitem{mo-zh-zh} S.-E. A. Mohammed, T. Zhang, H.Z. Zhao, The stable manifold theorem for semilinear stochastic evolution equations and stochastic partial differential equations, {\it Mem. Amer. Math. Soc.,} 196 (2008), 1-105.

\bibitem{bernt}G. Nunno, B. Oksendal and F. Proske, {\it Malliavin calculus for L\/evy processes with applications to finance}, Springer (2008).


\bibitem{nualart2} D. Nualart, {\it Malliavin Calculus and related topics}, Berlin, Springer (2000).


\bibitem{peszat} S. Peszat, On a Sobolev space of functions of infinite number of variable,
{\it Bull. Polish Acad. Sci. Math.,} 41 (1993), 55-60.






\bibitem{si2} Ya. Sinai, Burgers system driven by a periodic stochastic flows, In: {\it It$\hat {\rm o}$'s stochastic calculus and probability theory}, Springer, Tokyo (1996), 347-353.



\bibitem{peng} S.G. Peng, Problem of eigenvalues of stochastic Hamiltonian systems with boundary conditions, {\it Stochastic Processes and their Applications}, 
88 (2000), 259-290.

\bibitem{kampen} N.G. Van Kampen, {\it Stochastic Processes in Physics and Chemistry}, Elsevier, 2007.


\bibitem{wang} B.X. Wang, Existence, stability and bifurcation of random complete and
periodic solutions of stochastic parabolic equations, {\it Nonlinear Analysis}, 103 (2014), 9-25.

\bibitem{knoblock} B. Weiss and E. Knoblock, A stochastic return map for stochastic differential equations, {\it J. Stat. Phys.}, 58 (1990), 863-883.


\bibitem{zh-zh} Q. Zhang, H.Z. Zhao, Stationary solutions of SPDEs and infinite horizon BDSDEs, {\it Journal of Functional Analysis},  252 (2007), 171-219.

\bibitem{zh-zheng} H.Z. Zhao, Z.H. Zheng, Random periodic solutions of random dynamical systems, {\it Journal of Differential Equations},  246 (2009), 2020-2038.

\end{thebibliography}
   \end{document}